\documentclass[11pt,a4paper]{article}
\usepackage{graphicx} 
\usepackage{adjustbox}
\usepackage[margin=1.175in]{geometry}
\usepackage{amsmath}
\usepackage{amsthm}
\usepackage{thmtools}
\usepackage{thm-restate}
\usepackage{amssymb}
\usepackage{mathtools}
\usepackage{amsfonts}
\usepackage{booktabs}
\usepackage{pgfplots}
\pgfplotsset{compat=1.18}
\usepackage{enumitem}
\usepackage{comment}
\usepackage[
    colorlinks=true,
    linkcolor=blue,
    citecolor=blue,
    urlcolor=blue,
    hypertexnames=false
]{hyperref}
\usepackage[backend=biber,style=alphabetic, isbn=false,sorting=nyt]{biblatex}
\renewbibmacro{in:}{}
\AtEveryBibitem{%
  \clearfield{primaryClass}%
  \clearfield{EDITOR}
   \clearfield{eid}
   \clearfield{DOI}
}
\usepackage{xfrac}
\usepackage{bbold}
\usepackage{tikz} 
\usepackage{tikz-cd}
\usepackage{zref-clever}
\usepackage{quiver}
\zcsetup{cap,nameinlink}
\newcommand{\cref}{\zcref}

\usetikzlibrary{ 
	calc,%
	arrows,%
	shapes,
	shapes.geometric,
    positioning,
    fit,
    decorations.markings
} 
\usepackage{tikz-3dplot}
\usetikzlibrary{backgrounds, calc, math, patterns, arrows.meta}
\tikzcdset{
diagrams={>=stealth}}
 \tikzset{
             mid arrow/.style={
    decoration={
      markings,
      mark=at position 0.55 with {\arrow{>}}
    },
    postaction={decorate}
            },
    vertex/.style = {circle, fill, inner sep=2pt
            },
    HookedArrow/.style={right hook-latex
    }
            }

\usepackage[T1]{fontenc}
\usepackage{lmodern}
\usepackage{microtype}
\newcommand{\closedball}[1][r]{\overline{B}_{#1}}
\newcommand{\catRad}{\rho^{\kappa}_{M}}
\newcommand{\circCent}{c}

\newcommand{\cat}[1][C]{\textnormal{\textbf{#1}}}
\newcommand{\Top}{\cat[Top]}
\newcommand{\sSet}{\cat[sSet]}
\newcommand{\sCplx}{\cat[sCplx]}
\newcommand{\Set}{\cat[Set]}

\newcommand{\ho}{\textnormal{ho}}

\newcommand{\sd}{\textnormal{sd}}

\newcommand{\op}{\textnormal{op}}

\newcommand{\iSpaces}{\mathcal{S} \cat[pc]}

\newcommand{\met}{M}
\newcommand{\metapprx}{\mathbb{M}}
\newcommand{\fmet}[1][M]{(#1, f)}
\newcommand{\mbull}[1][\bullet]{M^{#1}}

\newcommand{\mapprxbull}[1][\bullet]{\mathbb{M}^{#1}}
\newcommand{\fmetapprx}[1][]{(\mathbb{M}_{#1}, \mathbb{f}_{#1})}
\newcommand{\fapprx}[1][f]{\mathbb{#1}}
\newcommand{\Rips}[1][\delta]{\mathcal{R}^{#1}}

\newcommand{\MRips}[1][\delta]{\mathcal{R}_m^{#1}}

\newcommand{\pos}{\mathbb{P}}
\newcommand{\inPos}{\mathbb{U}}
\newcommand{\flow}{+}
\newcommand{\supp}{\textnormal{supp}}

\newcommand{\persTop}[1][{\cat[I]}]{\Top^{#1}}

\newcommand{\abs}[1]{|#1|}
\newcommand{\real}[1]{|#1|}

\newcommand{\catK}[1][\kappa]{\textnormal{CAT}(#1)}
\newcommand{\cabK}[1][\kappa]{\textnormal{CBA}(#1)}

\newcommand{\R}{\mathbb{R}}
\renewcommand{\RN}{\mathbb{R}^N}
\newcommand{\RNPlus}{\mathbb{R}_{\geq 0}^N}
\newcommand{\rhoPos}{[0,\rho]}
\newcommand{\prodPos}{\mathbb{R}_{\geq 0} \times \inPos }
\newcommand{\prodPosR}{\mathbb{R}_{\geq 0} \times \RN }
\newcommand{\prodPosrho}{[0,\rho] \times \inPos}
\newcommand{\prodPosRhoRN}{ [0,\rho] \times \RN }
\newcommand{\conShr}{S}
\newcommand{\conShrM}[1][\rho]{\mathcal{J}_{M}^{#1}}
\newcommand{\conShrK}[1][\rho]{\mathcal{J}_{\kappa}^{#1}}
\newcommand{\conShrN}{\mathcal{J}_{n}}
\newcommand{\conShrTheta}{S_{\Theta}}
\newcommand{\conCosTheta}{C_{\Theta}}
\newcommand{\rhoTheta}{\rho_{\Theta}}
\newcommand{\conShrOne}[1][\rho]{\mathcal{J}_{1}^{#1}}

\newcommand{\lipConstRho}[1][\rho]{\mathcal{L}^{#1}_M}

\newcommand{\conCos}{C}
\newcommand{\sNerve}{N_s}
\newcommand{\sSetK}{\sSet_{\textnormal{Kan}}}
\newcommand{\diam}{\textnormal{diam}}
\newcommand{\circRad}[1][A]{\mathfrak{R}}
\makeatletter
\newcommand{\simeqdotted}{\mathrel{\mathpalette\simeqdotted@{}}}
\newcommand{\simeqdotted@}[2]{
  \setbox\z@=\hbox{$\m@th#1\sim$}
  \vcenter{\offinterlineskip
    \halign{##\cr
      \copy\z@\cr
      \noalign{\kern.15ex}
      \hbox to \wd\z@{\xleaders\hbox{\vrule height .06ex depth 0pt width .6ex\kern .3ex}\hfill}\cr
    }
  }
}
\makeatother

\newcommand{\define}[1]{\textit{#1}}

\newcommand{\dIntHo}{d_{\textnormal{intHo}}}
\newcommand{\MetRN}{\cat[Met]_{N}}
\newcommand{\dCor}[1][\alpha]{d_{\textnormal{Cor},#1}}
\newcommand{\vecepsilon}{\vec{\varepsilon}}
\newcommand{\epsmet}{\varepsilon_{\textnormal{met}}}
\newcommand{\epsfun}{\varepsilon_{\textnormal{fun}}}
\newcommand{\Lipf}[1][f]{L_{#1}}
\newcommand{\flexLipConstM}[1][\rho]{\mathcal{K}^{#1}_{M,L,K}}
\newcommand{\flexLipConstkappa}[1][\rho]{\mathcal{K}^{#1}_{\kappa,L,K}}
\newcommand{\eucLipConst}{\mathcal{L}_{n}}
\newcommand{\projError}[1][\rho]{\mathcal{E}^{#1}_{\tau_M}}
\newlength{\blockheadsep}
\usepackage{etoolbox} 

\makeatletter
\renewcommand\section{%
  \@startsection{section}{1}%
    {0pt}
    {-3.5ex \@plus -1ex \@minus -.2ex}
    {2.3ex \@plus .2ex}
    {\normalfont\large\bfseries}
}

\renewcommand\subsection{%
  \@startsection{subsection}{2}%
    {0pt}
    {0.9\baselineskip}
    {-0.6em}
    {\normalfont\bfseries}
}

\renewcommand{\@seccntformat}[1]{%
  \ifstrequal{#1}{subsection}%
    {\csname the#1\endcsname\hspace{\blockheadsep}}%
    {\csname the#1\endcsname\hspace{\blockheadsep}}%
}
\let\memoir@subsection\subsection
\renewcommand{\subsection}[2][]{%
  \ifstrempty{#1}%
    {\memoir@subsection[#2]{#2\@addpunct{.}}}
    {\memoir@subsection[#1]{#2\@addpunct{.}}}
}

\makeatother

\newtheoremstyle{blockyThm}%
  {0.9\baselineskip}
  {0.9\baselineskip}
  {\itshape}
  {0pt}
  {\bfseries}
  {.}
  {\blockheadsep}
  {\thmname{#1}\thmnumber{\ #2}\thmnote{ \normalfont(#3)}}

\theoremstyle{blockyThm}
\newtheorem{theorem}{Theorem}[section]
\newtheorem*{theorem*}{Theorem}
\newtheorem{lemma}[theorem]{Lemma}
\AddToHook{env/lemma/begin}{\zcsetup{countertype={theorem=lemma}}}

\newtheorem{corollary}[theorem]{Corollary}
\zcRefTypeSetup{corollary}{Name-sg =Corollary, Name-pl =Corollaries, name-sg =corollary, name-pl =corollaries}
\AddToHook{env/corollary/begin}{\zcsetup{countertype={theorem=corollary}}}
\newtheorem{proposition}[theorem]{Proposition}
\zcRefTypeSetup{proposition}{Name-sg =Proposition, Name-pl =Propositions, name-sg =proposition, name-pl =propositions}
\AddToHook{env/proposition/begin}{\zcsetup{countertype={theorem=proposition}}}

\zcRefTypeSetup{innercustomthm}{Name-sg =Theorem, Name-pl =Theorems, name-sg =theorem, name-pl =theorems}

\newtheoremstyle{blocky}%
  {0.9\baselineskip}
  {0pt}
  {\normalfont}
  {0pt}
  {\bfseries}
  {.}
  {\blockheadsep}
  {\thmname{#1}\thmnumber{\ #2}\thmnote{ \normalfont(#3)}}
\theoremstyle{blocky}
\newtheorem{definition}[theorem]{Definition}
\zcRefTypeSetup{definition}{Name-sg =Definition, Name-pl =Definitions, name-sg =definition, name-pl =definitions}
\AddToHook{env/definition/begin}{\zcsetup{countertype={theorem=definition}}}
\newtheorem{construction}[theorem]{Construction}
\zcRefTypeSetup{construction}{Name-sg =Construction, Name-pl =Constructions, name-sg =construction, name-pl =constructions}
\AddToHook{env/construction/begin}{\zcsetup{countertype={theorem=construction}}}
\newtheorem{example}[theorem]{Example}
\zcRefTypeSetup{example}{Name-sg =Example, Name-pl =Examples, name-sg =example, name-pl =examples}
\AddToHook{env/example/begin}{\zcsetup{countertype={theorem=example}}}
\newtheorem{recollection}[theorem]{Recollection}
\zcRefTypeSetup{recollection}{Name-sg =Recollection, Name-pl =Recollections, name-sg =recollection, name-pl =recollections}
\AddToHook{env/recollection/begin}{\zcsetup{countertype={theorem=recollection}}}
\newtheorem{remark}[theorem]{Remark}
\zcRefTypeSetup{remark}{Name-sg =Remark, Name-pl =Remarks, name-sg =remark, name-pl =remarks}
\AddToHook{env/remark/begin}{\zcsetup{countertype={theorem=remark}}}
\newtheorem{notation}[theorem]{Notation}
\zcRefTypeSetup{notation}{Name-sg =Notation, Name-pl =Notations, name-sg =notation, name-pl =notations}
\AddToHook{env/notation/begin}{\zcsetup{countertype={theorem=notation}}}

\zcRefTypeSetup{claim}{Name-sg =Claim, Name-pl =Claims, name-sg =claim, name-pl =claims}
\AddToHook{env/claim/begin}{\zcsetup{countertype={theorem=claim}}}

\zcRefTypeSetup{conjecture}{Name-sg =Conjecture, Name-pl =Conjectures, name-sg =conjecture, name-pl =conjectures}
\AddToHook{env/conjecture/begin}{\zcsetup{countertype={theorem=conjecture}}}

\zcRefTypeSetup{caveat}{Name-sg =Caveat, Name-pl =Caveats, name-sg =caveat, name-pl =caveats}
\AddToHook{env/caveat/begin}{\zcsetup{countertype={theorem=caveat}}}

\zcRefTypeSetup{question}{Name-sg =Question, Name-pl =Questions, name-sg =question, name-pl =questions}
\AddToHook{env/question/begin}{\zcsetup{countertype={theorem=question}}}

\zcRefTypeSetup{observation}{Name-sg =Observation, Name-pl =Observations, name-sg =observation, name-pl =observations}
\AddToHook{env/observation/begin}{\zcsetup{countertype={theorem=observation}}}

\newcounter{diagram}  
\zcRefTypeSetup{diagram}{Name-sg =Diagram, Name-pl =Diagrams, name-sg =diagram, name-pl =diagrams}
\newenvironment{diagram}[1][]{%
  \zcsetup{countertype={equation=diagram}}%
  \begin{equation}%
  \begin{tikzcd}[#1]%
}{%
  \end{tikzcd}%
  \end{equation}%
}
\newtheoremstyle{numonly}%
  {0.9\baselineskip}
  {0pt}
  {\normalfont}
  {0pt}
  {\bfseries}
  {}
  {\blockheadsep}
  {\thmnumber{#2}\thmnote{ \normalfont(#3)}}
\theoremstyle{numonly}
\newtheorem{block}[theorem]{Block} 
\zcRefTypeSetup{block}{Name-sg =Paragraph, Name-pl =Paragraphs, name-sg =paragraph, name-pl =paragraphs}
\AddToHook{env/block/begin}{\zcsetup{countertype={theorem=block}}}

\zcRefTypeSetup{pseudopropi}{Name-sg =Property, Name-pl =Properties, name-sg =property, name-pl =properties}
\zcRefTypeSetup{pseudopropii}{Name-sg =Property, Name-pl =Properties, name-sg =property, name-pl =properties}

\makeatletter
\newcommand{\ComputeEquiAngle}[2]{%
  \pgfmathsetmacro{\myside}{#1}%
  \pgfmathsetmacro{\mycurv}{#2}%
  \pgfmathparse{abs(\mycurv) < 0.000001 ? 1 : 0}%
  \ifnum\pgfmathresult=1
    \pgfmathsetmacro{\EquiAngle}{60}%
  \else
    \pgfmathparse{\mycurv > 0 ? 1 : 0}%
    \ifnum\pgfmathresult=1
      \pgfmathsetmacro{\u}{sqrt(\mycurv)*\myside} 
      \pgfmathsetmacro{\tmp}{cos(\u r)/(1+cos(\u r))}%
      \pgfmathsetmacro{\EquiAngle}{acos(\tmp)}%
    \else
      \pgfmathsetmacro{\u}{sqrt(-\mycurv)*\myside}%
      \pgfmathsetmacro{\tmp}{cosh(\u)/(1+cosh(\u))}%
      \pgfmathsetmacro{\EquiAngle}{acos(\tmp)}%
    \fi
  \fi
}
\makeatother
\newcommand{\curvedside}[5][]{%
  \path let
    \p1 = (#2),
    \p2 = (#3),
    \n1 = {atan2(\y2-\y1,\x2-\x1)}
  in
    \pgfextra{%
      \xdef\curvedsideoutangle{\n1 + #4}%
      \xdef\curvedsideinangle{\n1 + 180 - #4}%
    };
  \draw[#1] (#2)
    to[
      out=\curvedsideoutangle,
      in=\curvedsideinangle,
      looseness=#5
    ] (#3);
}

\newcommand{\DrawEquiTriangle}[3]{%
  \begin{tikzpicture}[scale=1, line cap=round, line join=round]
   \ComputeEquiAngle{#1}{#2}%

  \pgfmathsetmacro{\D}{#3}          
  \pgfmathsetmacro{\H}{0.82*\D}     
  \pgfmathsetmacro{\Loose}{1.3}    
    \pgfmathsetmacro{\angleLine}{-(0.5*\EquiAngle-30)}
    \coordinate (A) at (0,0);
  \coordinate (B) at (\D,0);
  \coordinate (C) at (0.5*\D,\H);
    \curvedside{A}{B}{\angleLine}{\Loose}
    \curvedside{B}{C}{\angleLine}{\Loose}
    \curvedside{C}{A}{\angleLine}{\Loose}
    \fill (A) circle (1.2pt);
    \fill (B) circle (1.2pt);
    \fill (C) circle (1.2pt);
  \end{tikzpicture}%
}

\title{Function-Rips complexes in persistent homotopy theory: \\
Stability and persistent Latschev theorems}
\author{Steve Oudot, Lukas Waas}
\date{\today}
\begin{document}
\maketitle
\begin{abstract}
    Classical results of Hausmann and Latschev show that Vietoris-Rips complexes can recover the homotopy type of a manifold, even from finite metric spaces that are nearby in Gromov-Hausdorff distance. We prove persistent homotopical versions of these theorems for metric spaces equipped with filtration functions. The central object of study is the so-called persistent homotopy type of the function-Rips complex, a filtered simplicial complex that combines a fixed Rips scale with the filtration data on the underlying space. Using techniques from $\catK$-geometry and persistent simplicial homotopy theory, we generalize Latschev's and Hausmann's theorems to the setting of spaces with filtration functions and homotopical interleavings. A fundamental ingredient is a new homotopical stability theorem. The fixed-scale function-Rips construction is known not to be globally stable with respect to function Gromov-Hausdorff distance and homotopical interleaving distance. Here, we show that it is nevertheless stable for appropriate choices of the Rips parameter at such pairs $(M,f)$ for which $M$ is a complete metric space of curvature bounded above, and $f$ is a Lipschitz continuous multivariate function.
\end{abstract}
\tableofcontents
\section{Introduction}
This work connects three lines of research in 
topological data analysis, metric topology and homotopy theory. The first one aims to generalize Latschev's result about the homotopy type of Vietoris-Rips complexes built on Gromov-Hausdorff approximations of compact Riemannian manifolds. 
\begin{theorem*}[Latschev~\cite{Latschev2001}]\label{thm:int_class_latschev}
 Let $\met$ be a closed Riemannian manifold. Then there exists $\delta_0>0$ such that, for every $0<\delta\leq\delta_0$, there exists an $\varepsilon_0>0$ such that the following holds:\\
 For any metric space $\metapprx$ with Gromov-Hausdorff distance to $\met$ less than $\varepsilon_0$, $\met$ is homotopy equivalent to the geometric realization $|\Rips(\metapprx)|$ of the Vietoris-Rips complex $\Rips(\metapprx)$ of parameter~$\delta$.
\end{theorem*}
To our knowledge, this result has been extended in the following directions:
replacing Riemannian manifolds with spaces of curvature bounded above~\cite{sushBoundedAbove}; quantifying the upper bounds on $\varepsilon$ and $\delta$~\cite{lim2024vietoris,Majhi2024_DCG_DemystifyingLatschev};  
  letting $\delta$ exceed the usual small scale bound on specific spaces, for instance, the circle~\cite{adamaszek2017vietoris} or ellipses~\cite{adamaszek2019vietoris}; considering 
variants of the Vietoris-Rips complex, such as the selective Rips complex~\cite{lemevz2022finite}. \\
Here, we extend Latschev's result in another direction: to spaces equipped with filtration functions (\cref{intThm:pers_lat_qual_version,intThm:pers_lat_quant_version} below). This requires the following generalizations:
\begin{enumerate}
	\item The role of the homotopy type of $M$ is taken by the so-called \textit{persistent homotopy type} of the sublevel set filtration of a function $f\colon M\to \RN$, denoted~$\mbull$~(\cref{notation:pers_homotopy_type}). 
	\item The role of the Vietoris-Rips complex of $\mathbb{M}$ is now taken by the filtered simplicial complex $\Rips[\delta](\mapprxbull)$, called the \textit{function-Rips} complex, obtained by filtering the ordinary Rips complex of fixed parameter $\delta$, $\Rips[\delta](\metapprx)$, through the Rips complexes of the same parameter~$\delta$ built on the sublevel sets of a function $\mathbb{f} \colon \metapprx \to \RN$ (\cref{ex:function-Rips complex}). 
	\item The role of correspondences of metric spaces (for the Gromov-Hausdorff distance) is taken by \textit{correspondences of space-function pairs}, which also take the distortion of function values into account (\cref{def:filtered_correspondence}). Such correspondences will be denoted in the form $ - \approx_{\epsmet,\epsfun} -$, where $\epsmet \geq 0$ controls the distortion in distances and $\epsfun \geq 0$ controls the distortion in function values.
	\item Finally, the role of homotopy equivalences is taken by the notion of \textit{interleavings in the persistent homotopy category} (\cref{rec:interleaving_dist}, see also \cite{LanariScoccola2023Rectification}), which are denoted in the form $- \simeq_{\varepsilon}-$ with interleaving parameter $\varepsilon \geq 0$.
\end{enumerate}
Throughout, we suppress geometric realization from the notation, and $\RN$ is always assumed to be equipped with the metric arising from the $\infty$-norm. \\
As our result is phrased in the language of persistence theory, it is fundamentally quantitative in nature. Let us, however, first state a version of the result that is as qualitative as possible to illustrate the analogy to the classical version of Latschev's theorem. We purposefully state it in a slightly weaker form than possible. We also assume that $f \colon M \to \RN$ is $1$-Lipschitz to simplify the statement. Recall that a metric space is called proper if every bounded closed subset in it is compact. 
\begin{theorem}[Persistent Latschev's Theorem - Qualitative Version]\label{intThm:pers_lat_qual_version}
 Let $\met$ be a proper metric space of curvature bounded above, and let $f \colon M \to \RN$ be a $1$-Lipschitz map. Then, there exists a $\delta_0 >0$ such that for every
 $0<\delta\leq\delta_0$, there exists an $\varepsilon_0>0$ such that the following holds: \\
 For any metric space-function pair $( \metapprx, \fapprx \colon \metapprx \to \RN)$ with function-correspondence distance to $\fmet$ less than $\varepsilon_0$, there is an interleaving
 \[
    \mbull \simeq_{\delta} \Rips[\delta](\mapprxbull)
 \]
 in the persistent homotopy category. 
 \end{theorem}
Note that, in contrast to the setup of Latschev's theorem, we cannot hope to get a homotopy equivalence between filtered spaces here. This is because the individual sublevel sets of~$f$ may be highly pathological, even though $\met$ itself satisfies strict regularity conditions. Note also that we cannot expect interleavings at the level of topological spaces: any such interleaving would induce homeomorphisms of the colimits obtained by letting the filtration parameter pass to infinity, which evidently do not exist in general. Nevertheless, at the level of persistent homotopy theory, we can quantitatively bound the error between $\mbull$ and $\Rips[\delta](\mapprxbull)$.
In the special case where $N = 0$ (i.e., $\RN$ is a singleton), the homotopy-theoretic interleaving  $\mbull 
  \simeq_{\delta} \Rips[\delta](\mapprxbull)$ is just an isomorphism in the homotopy category, i.e., a zig-zag of weak homotopy equivalences. So, by an application of Whitehead's theorem, \cref{intThm:pers_lat_qual_version} recovers Latschev's result as stated above.
Notably, \cref{intThm:pers_lat_qual_version} holds at the homotopy level directly, not just at the homology level, as is often the case with interleaving results in persistence theory. Any interleaving in the persistent homotopy category will, in particular, induce an interleaving at the persistent homology level. This connects our result to another line of research, the aim of which is to estimate the persistent homology of functions from finite point samples. In~\cite{chazal2011scalar}, the authors proposed an estimator for the persistent homology of  $L^{\bullet}$, for a Lipschitz map $g\colon L\to \R$, from a finite sampling~$\mathbb{L} \subset L$ of its domain~$L$. The proof technique leveraged in that line of work made use of the interleaving between Vietoris-Rips and \v Cech complexes, and ultimately led to an estimator given by the image of morphisms of persistence modules $\textnormal{im}(H_{\ast} \Rips[\delta] (\mathbb{L}^{\bullet}) \to H_{\ast} \Rips[2\delta] (\mathbb{L}^{\bullet}))$, where $\mathbb{L}$ is filtered by $g|_{\mathbb{L}}$.
This estimator was then specialized to the case where $g$ is a density estimator in the context of unsupervised learning, yielding the clustering algorithm ToMATo~\cite{chazal2013persistence}. It was later extended to more general noise models~\cite{buchet2015topological} and, more recently, to $\RN$-valued maps:
\begin{theorem}[\cite{andre2025estimating}]\label{AndreAndSteve}
  Let $(L,g)$ be such that $L$ is a compact geodesic metric space with convexity radius $\varrho_L>0$ and $g\colon L \to \RN$ is 
  a Lipschitz map with Lipschitz constant $\Lipf[g]$. Let $\delta <\frac{1}{2} \varrho_L$ and $\varepsilon \leq \frac{1}{2}\delta.$ Then, for any subset $\mathbb{L} \subset L$ that is $\varepsilon$-dense in $L$, there is a 
  $2\Lipf[g]\delta$-interleaving $H_{\ast}(L^{\bullet}) \simeq_{2 \Lipf[g] \delta} \textnormal{im}(H_{\ast} \Rips[\delta] (\mathbb{L}^{\bullet}) \to H_{\ast} \Rips[2\delta] (\mathbb{L}^{\bullet}))$ of persistence modules. 
\end{theorem}

\cref{intThm:pers_lat_qual_version} complements
 this result by showing that, under additional regularity assumptions on $\met$ (namely, that $\met$ has curvature bounded above), a single filtered Rips complex is enough to estimate the persistent homology of~$f$, via a different proof approach that does not proceed through the interleaving between Vietoris-Rips and \v Cech complexes. 
The question of whether this is possible had been open since the beginning of this line of work, and it has important implications, including algorithmic ones. For instance, computing a free presentation of $\textnormal{im}(H_{\ast} \Rips[\delta] (\mapprxbull) \to H_{\ast} \Rips[2\delta] (\mapprxbull))$
involves computing the free cover of a certain pullback, for which efficient specialized algorithms exist only when $N=1$ or $2$ (see \cite{andre2025estimating}); otherwise, one must resort to Schreyer's algorithm \cite{schreyer1991standard} with doubly exponential complexity in $N$. By contrast,  free presentations of $H_{\ast}\Rips[\delta](\mapprxbull)$ can be computed efficiently 
for any~$N\geq 1$ via specialized algorithms for the computation of Gr\"obner bases~\cite{gafvert2021topological}.

For the purpose of topological data analysis, one is naturally interested in a more quantitative version of \cref{intThm:pers_lat_qual_version}, specifying the bounds $\delta_0$ and $\varepsilon_0$, and giving a more informative interleaving constant. For the non-filtered case, such a result was first proven in \cite{Majhi2024_DCG_DemystifyingLatschev} for Riemannian manifolds, and then later extended to spaces of bounded curvature in \cite{sushBoundedAbove}. We will compare with that result later on in the introduction. For now, let us state our quantitative persistent version and note that it is not a formal consequence of the non-persistent case. \footnote{Indeed, the proof requires a refinement of techniques on both the geometric and the homotopy-theoretic sides.}
 To this end, we will require the following constants relating to the geometry of $M$.
\begin{enumerate}
    \item $\catRad$ denotes the minimum of the scale at which $M$ behaves like a $\catK$ space and a certain diameter bound depending on the curvature of $M$ (see \cref{not:conJung}). For example, if $M$ is compact, then $\catRad$ is simply the minimum of the convexity radius of $M$ and $\frac{\varpi_{\kappa }}{2}$, where $\varpi_{\kappa}$ denotes the diameter of the model space of constant curvature $\kappa$. Specifically, this means $\varpi_{\kappa} =\frac{\pi}{\sqrt{\kappa}}$ if $\kappa >0$ and $\varpi_{\kappa}=\infty$ otherwise.
    \item By $\conShrM \leq 1$, we denote Jung's constant of $M$ at scale $\rho \leq \catRad$ (see \cref{not:conJung}). It is given by taking the supremum over the ratios $\frac{\circRad(A)}{\diam (A)}$, where $A$ ranges over subsets of $M$ with positive diameter smaller than $\rho$ (see \cref{not:conJung}). Here $\circRad(A)$ denotes the radius of a minimal enclosing ball of $A$, and $\diam(A)$ denotes the diameter of $A$. Generally, $\conShrM  <1$ whenever $\rho < \frac{\varpi_{\kappa}}{2}$ (see \cref{cor:global_bound}). \footnote{Some extra care needs to be taken in the degenerate case where $M$ is discrete and no such set $A$ of positive diameter smaller than $\rho$ exist. Then formally $\conShrM = 0$. In this case, one needs to make the extra assumption $\epsmet <\delta$ in the theorems below, which are otherwise a consequence of the bound $\epsmet \leq \frac{1-\conShrM}{1 + \conShrM} \delta$. For the sake of presentation in this introduction, we just assume $\conShrM >0$, but note that we also discuss this special case in  \cref{rem:discrete_case,rem:discrete_later,appendix:discrete_case}.}
    \item We furthermore denote $\lipConstRho:= \frac{2\conShrM}{1-\conShrM}$. Note that we allow for $\conShrM =1$ here, in which case we formally set $\lipConstRho = \infty$. All stated results remain true with the formal convention $\lipConstRho \cdot 0 = 0$.
\end{enumerate}
\begin{theorem}[Persistent Latschev's Theorem - Quantitative Version (\cref{thm:strengthened_persistent_latschev})]\label{intThm:pers_lat_quant_version}
Let $M$ be a proper metric space of curvature bounded above by $\kappa$, and $f \colon M \to \RN$ a Lipschitz map, with Lipschitz constant $\Lipf$. \\
Let $0 < \delta \leq \catRad$. 
     Fix any $\rho$ with $\delta \leq\rho \leq \catRad$. Let $\epsmet \geq 0$ be such that
	 \[
	 \epsmet \leq \rho - \delta \quad \text{and} \quad \epsmet \leq \frac{1 - \conShrM}{1+ \conShrM} \delta,
	\]
	and let $\epsfun \geq 0$.
    Then, for any metric space $\metapprx$ equipped with a not necessarily continuous map $\fapprx \colon \metapprx \to \RN$, any correspondence $\fmetapprx \approx_{\epsmet,\epsfun} \fmet$ induces an interleaving 
     \[\Rips[\delta](\mapprxbull) \simeq _{\Lipf(\conShrM[\delta] \delta + \lipConstRho\epsmet) + \epsfun} \mbull\]  
     in the persistent homotopy category.
\end{theorem}
Observe that, in this quantitative version, there are three summands to the homotopy theoretic interleaving error:
\begin{enumerate}
	\item The first part, $\Lipf \conShrM[\delta] \delta$, is independent of the approximation parameters $\epsmet$ and $\epsfun$, and only depends on the choice of $\delta$, the geometry of $M$ and the Lipschitz constant of $f$.
	\item The second part, $\Lipf \lipConstRho \epsmet$, is linear in the metric distortion parameter $\epsmet$, with coefficient depending on the geometry of $M$ and the Lipschitz constant of $f$.
	\item The third part, $\epsfun$, is linear in the function distortion parameter $\epsfun$ (with coefficient $1$), and independent of the geometry of $M$ and the Lipschitz constant of $f$.
\end{enumerate}
Let us spell out some special cases of this result, providing explicit constants. 
It follows from a classical result of \cite{LangSchroeder1997Jung} that Jung's constant $\conShrM$, for $\rho \in (0, \frac{\varpi_{\kappa}}{2}]$, is bounded above by 
\[
\conShrK := \begin{cases} \frac{ \arcsin({\sqrt{2} \sin( \frac{\rho\sqrt{\kappa}}{2})) }}{\rho \sqrt{\kappa}} & \textnormal{, for $\kappa >0$} \\
\frac{1}{\sqrt{2}} & \textnormal{, for $\kappa \leq 0$.}
\end{cases} \]
In particular, for $\rho \leq \frac{\varpi_{\kappa}}{4}$ one can explicitly compute 
\[
\conShrK \leq 0.73 <\frac{3}{4}.
\]
Using this simplified bound of $\frac{3}{4}$ one obtains
\[
\frac{1 - \conShrM}{1+ \conShrM} > \frac{1- \frac{3}{4}}{1 + \frac{3}{4}} = \frac{1}{7}, \quad \text{and} \quad \lipConstRho < \frac{2 \frac{3}{4}}{1- \frac{3}{4}} = 6,
\]
independently of $\kappa$ and $M$ (see \cite{sushBoundedAbove} where similar constants appear). It follows that with these choices, it suffices to have $\epsmet \leq \min \{\rho - \delta, \frac{1}{7} \delta\}$ to obtain an interleaving 
\[\Rips[\delta](\mapprxbull) \simeq _{\Lipf(\frac{3}{4}\delta + 6\epsmet) + \epsfun} \mbull. \]
For the special case where $N = 0$, i.e., the non-filtered case, these choices recover the quantitative version of Latschev's theorem from \cite{Majhi2024_DCG_DemystifyingLatschev,sushBoundedAbove} \footnote{The result is stated in somewhat different form there. Firstly, one needs to note the difference between Hausdorff distance and correspondence distance, which introduces a bound $\frac{1}{14}\delta$ as opposed to $\frac{1}{7}\delta$. Showing that the two statements are then equivalent is a straightforward rearrangement of inequalities. Note that we could derive a statement with marginally weaker assumptions by using $0.73$ instead of $\frac{3}{4}$.}.
Stating the results in terms of the Jung constant $\conShrM$ instead of a global $\frac{3}{4}$ bound has the additional advantage of providing a larger range for $\delta$, as well as better bounds for spaces $M$ of low curvature and dimension. Suppose, for example, we are trying to estimate the sublevel set persistence of a function $f$ defined on a convex set $M$ in $\mathbb{R}^n$ (with non-empty interior) with $n=2$. Then we have 
\[
\conShrM = \sqrt{\frac{n}{2(n+1)}} = \sqrt{\frac{1}{3}}
\]
and thus 
\[
\frac{1 - \conShrM}{1+ \conShrM} \geq \frac{1- \sqrt{\frac{1}{3}}}{1 + \sqrt{\frac{1}{3}}} = 2 -\sqrt{3} \approx 0.268 \quad \text{and} \quad \lipConstRho \leq \frac{2 \sqrt{\frac{1}{3}}}{1- \sqrt{\frac{1}{3}}} = 1 + \sqrt{3} \approx 2.73.
\]\\
Similarly to the proof of the classical Latschev theorem, the proof of our persistent version (\cref{intThm:pers_lat_quant_version}) is a two-step process. We think that, in the quantitative case, it is particularly helpful to discuss these steps separately. Firstly, one needs a persistent version of Hausmann's Theorem (see \cite{Hausmann1995}), allowing us to relate $\mbull$ to $\Rips[\delta](\mbull)$. 
\begin{theorem}[Persistent Hausmann's Theorem (\cref{thm:persistent_hausmann})]\label{thm:persistent_hausmann_intro} 
Let $M$ be a proper metric space of curvature bounded above by $\kappa$, $f \colon M \to \RN$ be Lipschitz with Lipschitz constant $\Lipf$, and $0 <\delta \leq \catRad$. 
Then there is an interleaving \[\mbull \simeq _{\Lipf\conShrM[\delta]\delta} \Rips[\delta](\mbull) \] in the persistent homotopy category.
\end{theorem}%
This result, which we prove in \cref{sec:Hausmann}, is a straightforward consequence of a powerful new technique called metric thickening, which provides a more geometric model for the Vietoris-Rips complex~\cite{metric_reconstruction_via_optimal_transport,Gillespie2024VietorisThickenings}. \\
The crucial and technically more difficult part of proving \cref{intThm:pers_lat_quant_version} is to relate $\Rips[\delta](\mbull)$ and $\Rips[\delta](\mapprxbull)$. Note that if the operation 
\[
\mapprxbull \mapsto \Rips[\delta](\mapprxbull)
\]
were stable with respect to the distances associated with correspondences and homotopy-theoretic interleavings, then a bound on the distance of $\Rips[\delta](\mbull)$ and $\Rips[\delta](\mapprxbull)$ would follow. This connects our investigation to a third line of research: the study of stability properties of multivariate Vietoris-Rips constructions (see, for example, \cite{blumberg2024stability,BlumbergLesnick2023HomotopyInterleaving,rolle2024stable}.) Such stability results are generally obtained by letting all relevant persistence parameters vary simultaneously.
Indeed, we prove that the bivariate function-Rips construction $\Rips[\bullet](\mapprxbull)$ is $1$-Lipschitz stable with respect to correspondences of metric space-function pairs (see \cref{section:bivariate_stabiity}). The proof of this is rather straightforward; in homological settings, it was carried out in \cite{ccgmo-ghsssp-09} and then extended to $\RN$-valued functions in~\cite{andre2025estimating}. \\
While such multiparameter results are often appealing from a theoretical perspective, the addition of another parameter usually comes at a cost, both with respect to interpretability (the lack of a barcode) and computability (\cite{BotnanLesnick2023}).
Furthermore, to obtain our version of a persistent Latschev theorem, we fundamentally need a version in which we keep the Vietoris-Rips parameter $\delta$ fixed.
\\
The operation of fixing a single parameter in a multivariate persistent object is, however, generally unstable, at least on a global level. This forces our investigation into a direction much less discussed in the literature of TDA and persistence theory: the study of the behavior of generally unstable persistent constructions at a geometrically well-behaved (filtered) space $\fmet$. As these types of \define{stability-at-a-point results} describe the behavior of an invariant under small perturbations of nice objects, we also refer to them as \define{perturbative stability results} here.
Unlike global stability results, the proofs of which are generally formal in nature, such results require explicitly taking the geometry of $M$ into account. Indeed, here we prove the following result, which constitutes the technical heart of this article.
\begin{theorem}[Stability at $\cabK$ spaces (\cref{thm:local_stability_function_rips})]\label{thm:local_stab_intro}
Let $M$ be a complete metric space of curvature bounded above by $\kappa$, and $f \colon M \to \RN$ be Lipschitz, with Lipschitz constant $\Lipf \geq 0$. 
Let $0 <\delta \leq \catRad$ and fix $\rho$ with $\delta \leq \rho \leq \catRad$. Finally, let $\epsmet \geq 0$ be such that
	 \[
	 \epsmet \leq \rho - \delta \quad \text{and} \quad \epsmet \leq \frac{1 - \conShrM}{1+ \conShrM} \delta,
	\]
	and let $\epsfun \geq 0$.
    Then, for any metric space $\metapprx$ equipped with a not necessarily continuous map $\fapprx \colon \metapprx \to \RN$, any correspondence $\fmetapprx \approx_{\epsmet, \epsfun} \fmet$ induces an interleaving  \[
	\Rips[\delta](\mapprxbull) \simeq _{\Lipf \lipConstRho \epsmet + \epsfun} \Rips[\delta](\mbull) \] in the persistent homotopy category.
\end{theorem}
First, observe two special boundary cases of this result. 
\begin{enumerate}
	\item If $N =0$, \cref{thm:local_stab_intro} specializes to the fact that (under appropriate assumptions on $\delta$) the homotopy type of $\Rips[\delta](\metapprx)$ is constant in a neighborhood of $M$ in the Gromov-Hausdorff distance (compare with \cite{sushBoundedAbove}).
	\item If $\epsmet = 0$, then a correspondence $\fmetapprx \approx_{0, \epsfun} \fmet$ specifies an isometry of the underlying spaces. Identifying $\metapprx$ with $M$ via this isometry, one recovers the classical, much easier to prove fact that the operation $\Rips[\delta](-)$ is Lipschitz with respect to the supremum distance on functions on a fixed metric space $M$ (\cite{ChazalDeSilvaOudot2014}).
\end{enumerate}
Beyond these boundary cases, it follows immediately from the composability of interleavings that \cref{thm:local_stab_intro} together with \cref{thm:persistent_hausmann_intro} implies \cref{intThm:pers_lat_quant_version}.\\
Aside from this, while fundamentally not a global stability result, \cref{thm:local_stab_intro} guarantees Lipschitz continuity \define{at the point $\fmet$} in a sufficiently small neighborhood (see \cref{subsec:lipschitz_at_a_point}).
In particular, it gives precise quantitative descriptions of the asymptotic behavior of the persistent homotopy types of $\Rips[\delta](\mapprxbull)$ as $\fmetapprx$ approaches $\fmet$.
Let us make this explicit in a specialized case, namely when we assume that $\Lipf \leq 1$ and $\rho \leq \frac{\varpi_{\kappa}}{4}$. We denote by $\dCor[\infty]$ the infimum over all non-negative real numbers $\varepsilon= \max\{\epsmet, \epsfun\}$ such that there exists a correspondence $\fmetapprx \approx_{\epsmet, \epsfun} \fmet$ (see \cref{def:filtered_correspondence,con:corresp_distance}), and let $\dIntHo$ be the interleaving distance in the ($\RN$-)persistent homotopy category (see \cref{rec:interleaving_dist}). Then one obtains the following corollary of \cref{thm:local_stab_intro}. We purposefully state a version that is weaker than necessary here, for ease of readability (see \cref{subsec:lipschitz_at_a_point} for stronger and more general statements).
\begin{corollary}\label{cor:local_lipschitz_continuity}
Let $M$ be a complete metric space of curvature bounded above by $\kappa$, and $f \colon M \to \RN$ be $1$-Lipschitz. Let $0 <\delta < \min \{\catRad, \frac{\varpi_{\kappa}}{4}\}$. 
Then the assignment 
\[
\fmetapprx \mapsto \Rips[\delta](\mapprxbull)
\]
is $7$-Lipschitz continuous at $\fmet$ with respect to the distance $\dCor[\infty]$ and the interleaving distance $\dIntHo$. More precisely, one has 
\[
\dIntHo(\Rips[\delta](\mapprxbull), \Rips[\delta](\mbull)) \leq 7  \dCor[\infty](\fmetapprx, \fmet) 
\]
for all $\fmetapprx$ in an open ball of radius $\min\{\frac{1}{7}\delta, \min \{\catRad, \frac{\varpi_{\kappa}}{4}\} - \delta\}$ around $\fmet$.
\end{corollary}
Finally, let us discuss the proof strategy for \cref{thm:local_stability_function_rips}. Our techniques are quite general and may be of independent interest.
Our strategy is to leverage the bivariate interleaving $\Rips[\bullet](\mbull) \simeq_{\epsmet,\epsfun} \Rips[\bullet] (\mapprxbull)$ from the correspondence $\fmetapprx \approx_{\epsmet,\epsfun} \fmet$ of \cref{section:bivariate_stabiity} (see \cref{prop:bivariate_interleaving}) and turn it into a univariate interleaving of $\Rips[\delta](\mbull)$ and $\Rips[\delta] (\mapprxbull)$.
To this end, in \cref{section:shrinking_trick}, we introduce and study the general notion of a \textit{shrinking transformation}: a transformation of the form $F^{\bullet_1,\bullet_2} \to F^{S\bullet_1,\bullet_2 + C\bullet_1}$, with $C \geq 0$ and $0 < S < 1$ (we also allow $S=1$ for formal reasons). This additional structure allows one to move backward in one parameter at the cost of increasing the other in a bivariate persistent object $F^{\bullet,\bullet}$, as illustrated in \cref{fig:shrinking-trans}. Shrinking transformations can be used to turn interleavings in two variables $F^{\bullet,\bullet} \simeq_{\varepsilon_1, \varepsilon_2} G^{\bullet, \bullet}$ into interleavings of the form $F^{ \delta,\bullet} \simeq_{C\delta + C \varepsilon_1 + \varepsilon_2} G^{ \delta,\bullet}$ (\cref{lem:munchkin_lemma}). In this manner, they provide a general tool to obtain stability properties of persistent operations that fix a single parameter. 
\begin{figure}[t]
\centering
\includegraphics[width=0.75\textwidth]{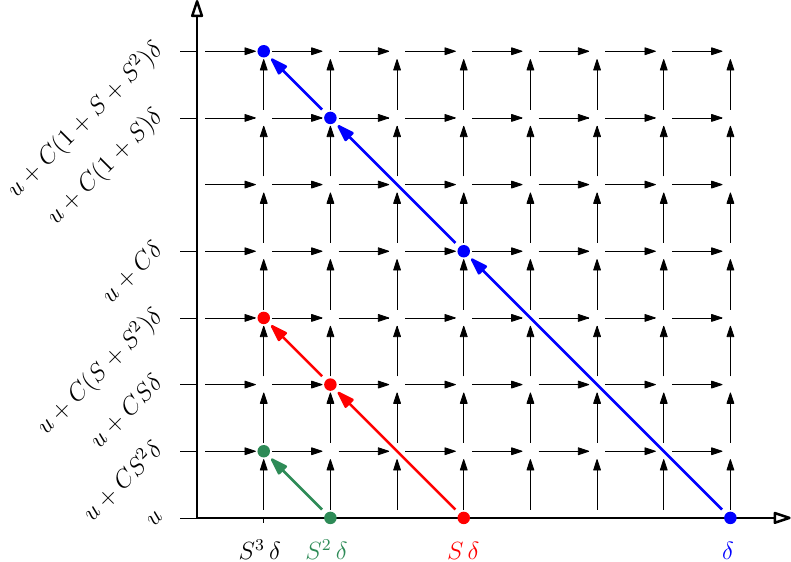}
\caption{Shrinking transformation arrows (colored) and the object's structure morphisms (black).}
\label{fig:shrinking-trans}
\end{figure}
Then, in \cref{section:shrinking_for_rips}, we provide conditions on $\met$ that guarantee the existence of such a shrinking transformation on $\Rips[\bullet](\mbull)$ (\cref{thm:back_propagation_rips}), namely the existence of what we call a \define{pseudo-barycenter map} (see \cref{def:pseudo_bar}). This is a map from the vertex set of a subdivision of $\Rips[\delta](M)$ into $M$ fulfilling certain distance conditions. One example of such a map - for the case of the barycentric subdivision - can be constructed by using centers of minimal enclosing balls - which explains the relevance of Jung's constant in our results. In this manner, one obtains a shrinking transformation with $\conShr=\conShrM$ and $\conCos = \Lipf \conShr$.\footnote{Such a map was already used implicitly in the proof of Latschev's theorem in \cite{sushBoundedAbove}.} (\cref{prop:existence_of_pseudo_bar_classical}). Together with
\cref{prop:bivariate_interleaving,lem:munchkin_lemma}, this already implies a version of \cref{thm:local_stab_intro} that provides an interleaving of the form 
\[
\Rips[\delta](\mapprxbull) \simeq _{\Lipf (\conShrM \delta + \conShrM\epsmet) + \epsfun} \Rips[\delta](\mbull) 
\]
as we state in the preliminary result \cref{thm:weaklocalStabFull}.
From this interleaving, one can already obtain a weakened version of \cref{intThm:pers_lat_quant_version}.
However, the bound depending on $\delta$ still restricts this result from being a proper stability result, as asserted in \cref{thm:local_stab_intro}. Indeed, as $\epsmet$ and $\epsfun$ converge to $0$ this weakened form guarantees no convergence of $\Rips (\mapprxbull)$ to $\Rips (\mbull)$.
Our results on pseudo-barycenter maps are, however, phrased in greater generality, pertaining to alternative subdivision functors and other constructions besides centers of minimal enclosing balls. In \cref{sec:local_stability_function_rips}, we construct such a subdivision functor (\cref{thm:properties_alt_subdiv}) and use several convexity properties of distance functions on $\cabK$ spaces (see \cref{prop:distance_to_point_conv,cor:pseudo_convexity_of_dist}) to prove the existence of a pseudo-barycenter map with cost constant $C$ arbitrarily close to $0$. As a consequence, we obtain \cref{thm:local_stab_intro} (\cref{thm:local_stability_function_rips} in the main body of the text).\\
Another advantage of the generality in which the notions of shrinking transformation and pseudo-barycenter map are developed is that they also apply to other geometric assumptions than spaces whose metric is (essentially) geodesic. Indeed, in \cref{sec:Euclidean}, we also prove versions of \cref{thm:local_stab_intro,thm:persistent_hausmann_intro,intThm:pers_lat_quant_version} for the case of subspaces $A$ of Euclidean space, where the role of the curvature bound $\kappa$ is replaced by the reach of $A$ (see \cref{thm:hausmann_Euclidean,thm:local_stability_Euclidean,thm:latschev_Euclidean}).
\section{Background}\label{sec:background}
In this section, we recall some of the relevant background information on persistent homotopy theory as well as metric spaces of bounded curvature.
\subsection{Spaces of curvature bounded above}
As mentioned in the introduction, our results will pertain to \textit{metric spaces of curvature bounded above}. We recommend \cite{BBI2001,BridsonHaefliger1999} as a more detailed source for the following definitions. 
\begin{recollection}[Geodesics]
    Let $M$ be a metric space with metric $d$. By a \textit{constant speed geodesic} we mean a map $\gamma \colon [0,1] \to M$ such that $d(\gamma(s),\gamma(t))=|s-t|d(\gamma(0), \gamma(1))$, for all $s,t \in [0,1]$. 
\end{recollection}
\begin{notation}[Comparison spaces]\label{rec:comparison_spaces}
    Let $\kappa \in \mathbb{R}$.
    We denote by $M_{\kappa}$ the following spaces:
    \begin{enumerate}
        \item For $\kappa > 0$, the sphere of radius $\frac{1}{\sqrt{\kappa}}$ in $\R^3$, equipped with the geodesic distance;
        \item For $\kappa = 0$, the Euclidean plane $\mathbb R^2$ equipped with the Euclidean distance;
        \item For $\kappa < 0$, the hyperbolic plane $\mathbb{H}$ with its metric rescaled by the factor  $\frac{1}{\sqrt{-\kappa}}$.
    \end{enumerate}
    We denote by $\varpi_{\kappa}$ the diameter of $M_{\kappa}$ ($\varpi_{\kappa} = \frac{\pi}{\sqrt{\kappa}}$ if $\kappa >0 $ and $\varpi_{\kappa} = \infty$ if $\kappa \leq 0$).
\end{notation}
\begin{recollection}[$\catK$ spaces]\label{rec:catK}
    Let $M$ be a metric space. By a \define{geodesic triangle in $M$} we mean three points $x,y,z \in M$, together with geodesics $\gamma_{x,y}, \gamma_{y,z}, \gamma_{z,x}$, connecting the points $x,y,z$ in the obvious manner. We will write $T(x,y,z)$ for such a triangle (less accurately so, if the connecting geodesics are not unique). When we refer to a triangle $T(x,y,z)$ as a subspace of $M$, we will mean the union of the images of its defining geodesics.
    By a \textit{comparison triangle} of $x,y,z$ in $M_{\kappa}$ we mean a triangle $T(\overline{x}, \overline{y},\overline{z})$ in $M_{\kappa}$ with the same pairwise distances $d(x,y) = d(\overline{x},\overline{y})$, $d(y,z) = d(\overline{y},\overline{z})$ and $d(z,x) = d(\overline{z},\overline{x})$. 
    If the perimeter of $T(x,y,z)$, $ d(x,y) + d(y,z) + d(z,x)$, is smaller than $2\varpi_{\kappa}$ then such a comparison triangle always exists and is unique up to isometry. Following the constant-speed parametrizations, we obtain a unique geodesic-preserving map $T(\overline{x}, \overline{y}, \overline{z}) \to T(x,y,z)$ mapping $\overline{x}\mapsto x$, $\overline{y} \mapsto y$ and $\overline{z} \mapsto z$. We say that $M$ \textit{is a $\catK$ space} if any two points of distance smaller than $\varpi_{\kappa}$ are connected by a geodesic, and $T(\overline{x}, \overline{y}, \overline{z}) \to T(x,y,z)$ is $1$-Lipschitz for every triangle $T(x,y,z)$ of perimeter smaller than $2\varpi_{\kappa}$ (see \cref{fig:equitriangles} for an illustration). 
\end{recollection}
\begin{figure}[h]
    \centering
    \begin{minipage}{0.3\textwidth}
        \centering
        \DrawEquiTriangle{2}{-1}{2}
    \end{minipage}
    \hfill
    \begin{minipage}{0.3\textwidth}
        \centering
        \DrawEquiTriangle{2}{0}{2}
    \end{minipage}
    \hfill
    \begin{minipage}{0.3\textwidth}
        \centering
        \DrawEquiTriangle{2}{1}{2}
    \end{minipage}
    \caption{Illustration of equilateral triangles in curvatures $\kappa<0$, $\kappa=0$, and $\kappa>0$.}
    \label{fig:equitriangles}
\end{figure}
\begin{recollection}[$\cabK$ spaces]\label{rec:cbaK}
	Let $\kappa \in \mathbb{R}$. We say that a metric space $M$ is \textit{of curvature bounded above by $\kappa$} if every point $x \in M$ admits a neighborhood that is a $\catK$ space. We will often just say that \textit{$M$ is a $\cabK$ space} to indicate this. 
\end{recollection}
\begin{notation}
Given $r >0$ and $x \in M$, we denote by $\closedball[r](x) := \{ y \in M \mid d(x,y) \leq r\}$ the closed ball of radius $r$ centered at $x$.
\end{notation}
\begin{notation}\label{not:catRad}
Let $M$ be a metric space, $\kappa \in \mathbb{R}$ and $x \in M$. We denote by 
\[
\catRad(x) := \min \{\sup \{ r > 0 \mid \closedball[r](x) \text{ is a } \catK \text{ space} \} , \frac{\varpi_{\kappa}}{2}\}.
\]
Furthermore, we denote 
\[
\catRad :=\inf_{x \in M} \catRad (x).
\]
\end{notation}
\begin{recollection}
Recall that any ball of radius smaller than $\frac{\varpi_{\kappa}}{2}$ in a $\catK$ space is itself a $\catK$ space. Thus, it follows that whenever $r < \catRad(x)$, the closed ball $\closedball[r](x)$ is a $\catK$ space. 
\end{recollection}
\begin{recollection}
	If $M$ is a $\cabK$ space, then by definition $\catRad(x) > 0$ for all $x \in M$. If, in addition to this, $M$ is compact, then $\catRad > 0$. To see this, observe that the function $\catRad(\cdot) \colon M \to (0,\infty]$ is lower semi-continuous, and thus attains a minimum on the compact space $M$.
\end{recollection}
\begin{recollection}\label{rec:convex}
   We adopt the following strong notion of convexity (this is slightly stronger than the one in \cite{BBI2001} which does not require uniqueness): We say that a subset $A \subset M$ is \textit{convex}, if every pair of points $x,y \in A$ is connected by a unique constant speed geodesic $\gamma$ from $x$ to $y$ in $M$, and that geodesic lies in $A$. 
\end{recollection}
\begin{recollection}
Recall that the \textit{convexity radius} of a metric space $M$ is defined as the supremum of all $r > 0$ such that every closed ball of radius smaller than $r$ in $M$ is convex. 
\end{recollection}
\begin{remark}
	Suppose that $M$ is a $\cabK$ space that is additionally proper, i.e., such that bounded closed sets are compact. Denote by $\rho$ the convexity radius of $M$. Then 
	\[
	\min \{\rho, \frac{\varpi_{\kappa}}{2}\} = \catRad.
	\]
	Indeed, any convex compact ball of radius smaller than $\min \{\rho, \frac{\varpi_{\kappa}}{2}\}$ in a $\cabK$ space is a $\catK$ space (see \cite[Ch II, Cor. 4.12]{BridsonHaefliger1999}). Conversely, and independently of the properness of $M$, any closed ball of radius smaller than $\frac{\varpi_{\kappa}}{2}$ in a $\catK$ space is convex.
\end{remark}
	
    \begin{remark}\label{rec:facts_about_balls_in_bounded_curvature}
		Throughout this article, we will constantly be working at scales smaller than $\catRad$. It follows by the definition of $\catRad$ that at such scales, we can generally freely cite results about $\catK$ spaces. 
\end{remark}
\subsection{Circumcenters and Jung constants in $\cabK$-spaces}
The crucial geometric ingredient in the proof of Latschev's theorem in \cite{Majhi2024_DCG_DemystifyingLatschev} is the notion of a circumcenter. We will also be leveraging this notion in our persistent setting. We therefore recall some basic properties of circumcenters in bounded curvature. For the remainder of this subsection, we fix a complete metric space $M$ of curvature bounded above by $\kappa \in \mathbb{R}$.
We recall the fundamental properties for the convenience of the reader.
\begin{definition}\label{def:chebcenter}\cite{LangSchroeder1997Jung}
    Given a finite non-empty subset $A \subset M$, suppose that the function 
    \begin{align*}
       \mathfrak{r} \colon M &\to [0,\infty) \\
        x &\mapsto \max_{a \in A} d(x,a)
    \end{align*}
    admits a unique minimizer $c \in M$. Then $c$ is called the \define{circumcenter}, or \define{Chebyshev center}, of $A$, and denoted $\circCent(A)$\footnote{One should be warned that this notion of circumcenter does not generally agree with the classical notion of a circumcenter of a triangle.}. The value $ \mathfrak{r}(\circCent(A))$ is called the \define{circumradius of $A$}, and denoted $\circRad(A)$.
\end{definition}
\begin{remark}
	Note that $\circCent(A)$ is the center of a minimal enclosing ball of $A$. In particular, if $A$ is contained in a ball of radius $r$ around some point $x$, then $\circRad(A) \leq r$. 
\end{remark}
In bounded curvature, circumcenters exist at a sufficiently small scale:
\begin{theorem}\cite{LangSchroeder1997Jung}\label{theo:ex_of_circumcenter}
Let $A \subset M$ be a finite non-empty subset such that $\diam(A) < \catRad$. Then $\circCent(A)$ exists. Furthermore, $\circCent(A)$ is contained in any closed convex subset containing $A$. 
\end{theorem}
\begin{proof}
	\cite[Thm. A]{LangSchroeder1997Jung} treats the case of a $\catK$ space, stating the existence and uniqueness of a minimizer of $\mathfrak{r}$. We now generalize this statement to spaces of bounded curvature. To this end, fix any closed ball $B$ of radius less than or equal to $\diam(A)$ that contains $A$, for example, the ball of radius $\diam(A)$ centered at some $x \in A$. 
    By the assumption $\diam(A) <\catRad$, such a ball is a complete $\catK$ space. 
    Hence, it follows from the $\catK$ case that $A$ admits a circumcenter in $B$, which we denote by $c_B(A)$. 
    Note, furthermore, that any ball of diameter smaller than $\catRad \leq \frac{\varpi_{\kappa
    }}{2}$ in a $\catK$ space is convex. Next, we make the following observations:
    \begin{enumerate}
        \item Given any other $x \in M$, with $\mathfrak{r}(x) \leq \mathfrak{r}(c_{B}(A))$, it holds by definition that the closed ball $B' = \closedball[ {\mathfrak{r}(x)}](x)$ contains $A$. 
        \item Suppose now that we have shown that $c_B(A)$ is independent of the choice of $B$. Then it follows that $c_B(A)$ is also the unique minimizer of $\mathfrak{r}$ on $B'$. As $\mathfrak{r}(x) \leq \mathfrak{r}(c_B(A))$, we obtain that $x = c_B(A)$. Hence, we have shown that $c_B(A)$ is a global minimizer.
    \end{enumerate}
 To prove the independence of $B$, we show the following stronger statement: Given any closed convex set $C$ containing $A$, it holds that $c_B(A) \in C$. Note that this does indeed imply the independence of the choice of $B$, as it implies that for any other such $B'$, we have $c_B(A), c_{B'}(A) \in B \cap B'$ and thus $c_B(A) = c_{B'}(A)$ by the uniqueness of minimizers of $\mathfrak{r}$ on the complete $\catK$ space $B \cap B'$. Furthermore, observe that this stronger statement also implies the second statement of the theorem. Now, let $K$ be the intersection of all closed convex subsets containing $A$ and let $K'$ be the intersection of all closed convex subsets containing $A$ and $c_{B}(A)$. By definition, $K$ and $K'$ are again convex and closed. Note, furthermore, that $K \subset K'$ and $\diam(K') \leq \diam(A) < \catRad \leq \frac{\varpi_\kappa}{2}$ (see \cref{lem:intersection_trick} below). 
    We can now apply \cite[Cor. 2.10]{KimuraSasaki2023PerturbationsResolvent}, which states that under this diameter bound on $K'$ the nearest point projection map $\pi \colon K' \to K$ is well-defined and has the property that 
    \[
    d(x, \pi(y)) \leq d(x,y) 
    \]
    for any $x \in K$ and $y \in K'$. In particular, it follows that 
    \[
    d(a, \pi (c_B(A))) \leq d(a, c_B(A)),
    \]
    for all $a \in A$. As $K' \subset B$, it follows that $ \pi (c_B(A)) $ is also a minimizer of $\mathfrak{r}$ on $B$. Consequently, $\pi (c_B(A)) = c_B(A)$ and thus $c_B(A) \in K$. This finishes the proof.
\end{proof}
\begin{lemma}\label{lem:intersection_trick}
    Suppose that $X$ is a metric space and that $\mathcal{I}$ is a set of subsets of $X$ that is stable under intersection. Suppose, furthermore, that $\delta \geq 0$ is such that every closed ball of radius $\delta$ in $X$ is in $\mathcal{I}$. Let $A \subset X$ be such that $\diam(A) \leq \delta$. Then $\diam( \cap_{B \in \mathcal{I},A \subset B} B) \leq \delta$. 
\end{lemma}
\begin{proof}
Denote $C_A = \cap_{B \in \mathcal{I},A \subset B} B$. 
    Note that for any $x\in C_A$, we have $x\in \closedball[\delta](a)$ for every $a \in A$, as $\closedball[\delta](a)$ is a set in the defining intersection of $C_A$. Consequently, $A \subset \closedball[\delta](x)$. As the latter ball is again in $\mathcal{I}$, it follows that $C_A \subset \closedball[\delta](x)$. Consequently, $d(x,y) \leq \delta$ for all $y \in C_A$. As $x$ was arbitrary, we deduce that $\diam(C_A) \leq \delta$.
\end{proof}
A crucial quantity to consider for our purposes will then be the following:
\begin{notation}\label{not:conJung}
	Given $0<\rho \leq \catRad$, we denote 
	\[
	\conShrM:= \sup \{ \frac{\circRad(A)}{\diam(A)} \mid A \subset M, \textnormal{ finite, } 0 < \diam(A) < \rho \}.
	\]
	We refer to $\conShrM$ as the \textit{Jung constant} of $M$ at scale $\rho$.
\end{notation}
\begin{remark}\label{rem:discrete_case}
	Note that, as long as the set over which we are taking the supremum is non-empty, it follows from the reverse triangle inequality that $\conShrM \geq \frac{1}{2}$. Note that this set can only be empty if $M$ is discrete at scale $\rho$, i.e., if there are no non-degenerate pairs of points in $M$ at a distance smaller than $\rho$; in this case, the supremum is taken over an empty set (and in $\mathbb{R}_{\geq 0}$), and thus $\conShrM = 0$.
\end{remark}
\begin{block}
The Jung constant $\conShrM$ quantifies the worst-case ratio of the circumradius and diameter of a set at scale $\rho$. 
In the case where $M$ is a convex subset of $\mathbb{R}^n$ (of full dimension), one has 
\[
\conShrM = \sqrt{\frac{n}{2(n+1)}},
\]
independently of $\rho$, by the classical Jung's theorem. 
\end{block}
\begin{lemma}\label{lem:distance_bounds_circum}
Let $A\subset A'$ be two finite, non-empty subsets of $M$ such that $\diam(A') < \rho \leq \catRad$. Then 
\( d(\circCent(A), \circCent(A')) \leq \conShrM \, \diam(A')\).\\ Furthermore, the inequality \(d(\circCent(A),x) \leq \diam(A')\) holds for all \(x \in A'\).
\end{lemma}
\begin{proof}
	By \cref{theo:ex_of_circumcenter}, we have that $\circCent(A) \in \closedball[\diam(A')](x)$, for any $x \in A'$, and thus $d(\circCent(A),x) \leq \diam(A')$, for all $x \in A'$. 
	Since $\closedball[\circRad(A')](\circCent(A'))$ also contains $A'$, it follows from \cref{theo:ex_of_circumcenter} that $\circCent(A) \in \closedball[\circRad(A')](\circCent(A'))$, or in other words, $d(\circCent(A), \circCent(A')) \leq \circRad(A') \leq \conShrM \diam(A')$, as claimed.
\end{proof}
\begin{remark}
	As the scale approaches zero, the geometry of a $\cabK$ space $M$ approaches that of a $\cabK[0]$ space. More precisely, for $r>0$ it holds that the rescaling of $M$ by $\frac{1}{r}>0$ is (generally at best) a $\cabK[\kappa r^2]$ space. As $\conShrM$ is defined in terms of a ratio that is invariant under such rescaling, and takes arbitrarily small scales into account, we can generally not expect it to behave better than in the case of curvature bounded above by $0$. 
\end{remark}
Let us recall some global bounds on $\conShrM$, depending only on $\kappa \in \mathbb{R}$.
In the general case of bounded curvature, one can derive from \cite[Thm. A]{LangSchroeder1997Jung} the following bounds on $\conShrM$.
\begin{notation}\label{not:conJungGlob}
 We denote by $\conShrOne[]$ the continuous function 
	\begin{align*}
	\conShrOne[] \colon [0, \frac{\pi}{2}] &\to \mathbb{R} \\
	\rho &\mapsto \begin{cases} \frac{1}{\sqrt{2}} & \text{if } \rho = 0 \\
	\frac{\arcsin ( \sqrt{2} \sin (\frac{\rho}{2}) )}{\rho} & \text{if } \rho > 0
	\end{cases}
	\end{align*}
	and by $s_{\kappa} \colon [0, \frac{\varpi_{\kappa}}{2}] \to [0, \frac{\pi}{2}]$ the function given by $\rho \mapsto \sqrt{\kappa}\rho$, when $\kappa>0$. Furthermore, we denote
    \begin{align*}
        \conShrK[] \colon [0, \frac{\varpi_{\kappa}}{2}] &\to \mathbb{R} \\
        \rho &\mapsto \begin{cases}
	(\conShrOne[] \circ s_{\kappa})(\rho) & \text{for } \kappa > 0; \\
	\frac{1}{\sqrt{2}} & \text{for } \kappa \leq 0.
	\end{cases}
    \end{align*}
	Given $\rho \in [0, \frac{\varpi_{\kappa}}{2}]$, we will denote the value $\conShrK[](\rho)$ in the form $\conShrK$ and call it the \textit{model Jung constant} for curvature bounded by $\kappa$ at scale $\rho$.
\end{notation}
It then follows from \cite[Thm. A]{LangSchroeder1997Jung} that the following holds.
\begin{theorem}\cite{LangSchroeder1997Jung}
	Let $0<\rho \leq \catRad$. Then $\conShrM \leq \conShrK$.
\end{theorem}
\begin{figure}
	\centering
	\includegraphics[width=0.8\textwidth]{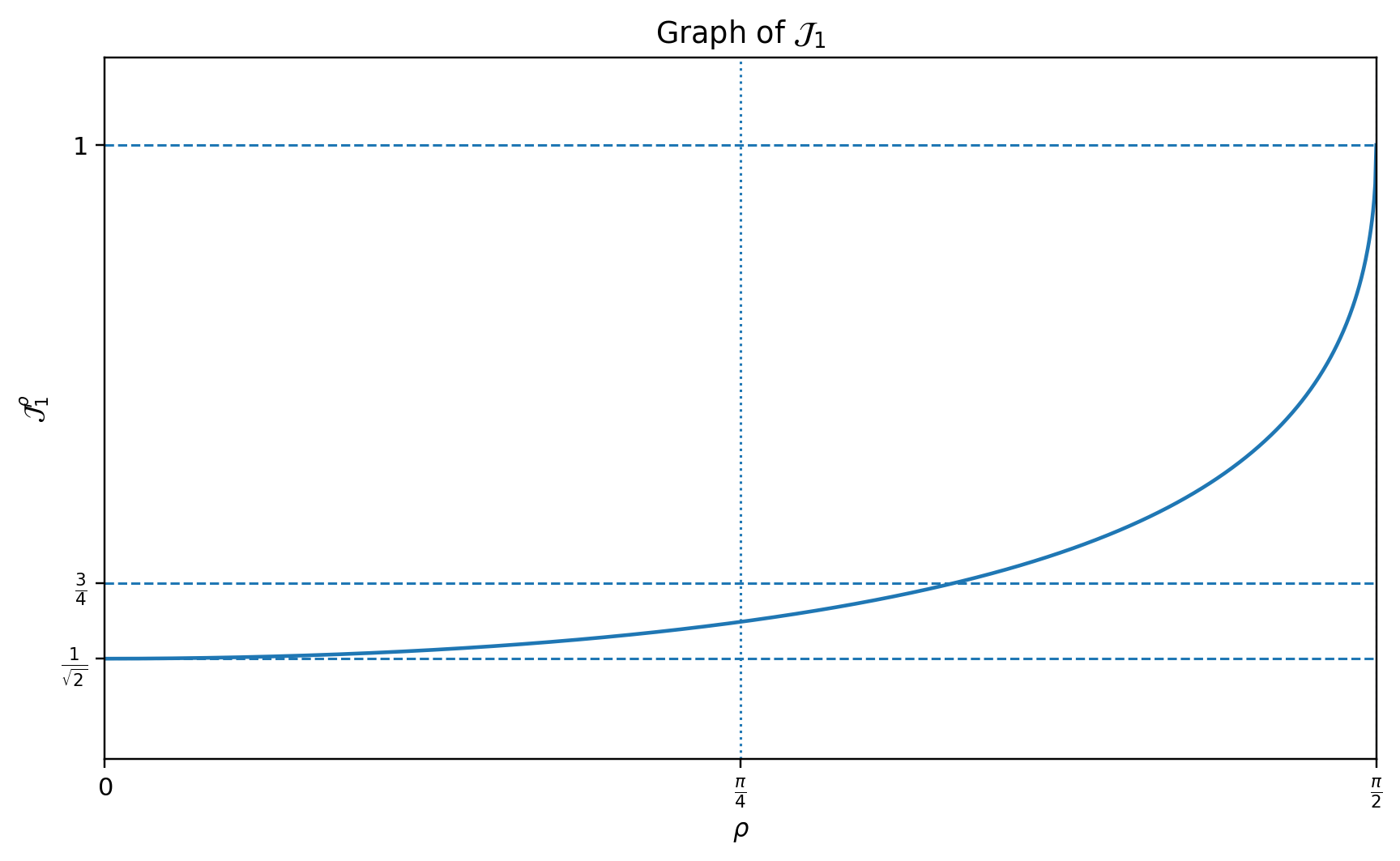}
	\caption{Graph of $\conShrK[]$ for $\kappa = 1$.}
    \label{fig:graph_S}
\end{figure}
\begin{remark}\label{rem:explicit_numbers}
	It will be helpful to have some explicit properties of the function 
	\(
	\conShrK[]
	\)
	in mind (see \cref{fig:graph_S} for a plot of the case $\kappa=1$). Observe that for all $\kappa >0$ the function $\conShrK[]$ has the following properties:
	\begin{enumerate}
		\item $\conShrK$ is strictly increasing in $\rho$;
		\item $\conShrK[0] = \frac{1}{\sqrt{2}}$, $\conShrK[\frac{\varpi_{\kappa}}{4}] \approx 0.728$ and $\conShrK[\frac{\varpi_{\kappa}}{2}] =1$;
		\item In particular, $\conShrK[\rho] \in (\frac{1}{\sqrt{2}}, 1)$, for $\rho \in (0, \frac{\varpi_{\kappa}}{2})$.
	\end{enumerate}
\end{remark}
Hence, one obtains the following bounds on $\conShrM$ independently of $\kappa$ (see also \cite[Appendix]{Majhi2024_DCG_DemystifyingLatschev}, for a similar computation). 
\begin{corollary}\label{cor:global_bound}
	Let $0 <\rho \leq \min \{ \catRad, \frac{\varpi_{\kappa}}{4} \}$. Then $\conShrM \leq \conShrK[\rho] < \frac{3}{4}$.
\end{corollary}
The number $\frac{3}{4}$ is the global bound of $\conShrM$ employed in \cite{Majhi2024_DCG_DemystifyingLatschev,sushBoundedAbove}.
\subsection{Models for persistent homotopy types}
As explained in the introduction, stability results in this article will be expressed at the level of \define{persistent homotopy theory} rather than persistent homology. Our objects of concern are thus \define{persistent homotopy types}. Conceptually speaking, these are functors from an appropriate poset such as $\R$ into the homotopy theory ($(\infty,1)$-category) of spaces. In this text, we will use three explicit models for \textit{persistent homotopy types}, arising from spaces, simplicial complexes, and simplicial sets, the last of which we discuss in \cref{subsec:simplicial_sets}.
\begin{notation}
    We denote by $\sCplx$ the category of simplicial complexes and simplicial maps, and by $\Top$ the category of topological spaces and continuous maps. 
	We denote by $\real{K}$ the topological realization of a simplicial complex $K$. Often, we will just treat a complex as a space, and leave the realization implicit. The standard $n$-simplex, given by the set of non-empty subsets of $\{0, \dots, n\}$, will be denoted $\Delta^n_c$. 
\end{notation}
\begin{definition}
In the following, $\pos$ will always denote a partially ordered set. 
	When we treat $\pos$ as a category, we mean the category whose objects are given by the elements of $\pos$, and where there is a unique morphism $x \to y$ whenever $x \leq y$. \\
	Let $\cat[C]$ be a category. A \textit{$\pos$-persistent object} in $\cat[C]$ is a functor $F \colon \pos \to \cat[C]$. A morphism of $\pos$-persistent objects is a natural transformation between such functors. We denote by $\cat[C]^{\pos}$ the category of $\pos$-persistent objects in $\cat[C]$ and their morphisms.
\end{definition}
\begin{notation}
We will usually denote persistent objects $F \colon \pos \to \cat[C]$ in the form $F^{\bullet}$, to indicate their functoriality in $\pos$.
When the objects in a category $\cat[C]$ have a specific name, such as simplicial complexes, topological spaces or metric spaces, we will often refer to $\pos$-persistent objects in $\cat[C]$ by adding the prefix persistent to that name. 
\end{notation}
\begin{notation}
    $N \geq 0$ will always be a non-negative integer and $\RN$ will be considered as a partially ordered set via componentwise comparison. When treated as a metric space $\RN$ will be equipped with the metric arising from the $\infty$-norm.
\end{notation}
\begin{notation}
By a \textit{space-function pair} (often just pair), we will mean a topological space $X$ together with a function $f \colon X \to \RN$. When we speak of a \textit{metric pair}, this will mean that we equip $X$ with a metric, inducing its topology.  
\end{notation}
\begin{example}\label{ex:sublevel_pers_space}
    Given a pair $(X,f)$, the sublevel sets $X^u := f^{-1}\{v \in \RN \mid v \leq u\}$, together with functoriality given by inclusions, give rise to a persistent space $u \mapsto X^{u}$, which we denote by $X^{\bullet}$. We abuse notation here insofar as the construction evidently depends on $f$. As all structure maps are given by inclusions $X^{\bullet}$ is often also referred to as a \textit{filtered space}.
\end{example}
\begin{example}[Function-Rips complex]
\label{ex:function-Rips complex}
Given a metric space $\mathbb{M}$ and $\delta \geq 0$, the open Vietoris-Rips complex of $\mathbb{M}$ (only Rips complex henceforth), denoted $\Rips[\delta](\mathbb{M})$, is the simplicial complex whose set of simplices is given by 
\[
\{\{x_0, \dots, x_n\} \mid x_i \in \mathbb{M}, d(x_i,x_j) <\delta \textnormal{, for all }0  \leq i,j \leq n \}.
\]
Varying the parameter $\delta \geq 0$, one obtains a persistent simplicial complex $\Rips[\bullet](\mathbb{M}) \colon \mathbb{R}_{\geq 0} \to \sCplx$, with functoriality on relations given by inclusions. Suppose now that $\mathbb{M}$ is additionally equipped with a (not-necessarily continuous) function $\mathbb{f} \colon \mathbb{M} \to \RN$. Then, for every $u \in \RN$, one can consider the Vietoris-Rips complex of $\mathbb{M}^{u} := \{ x \in \mathbb{M} \mid \mathbb{f}(x) \leq u\}$. 
Varying $u$ with $\delta$ fixed gives rise to a persistent simplicial complex $\Rips[\delta](\mapprxbull) \colon \RN \to \sCplx$, $u \mapsto  \Rips[\delta](\mathbb{M}^{u})$, called the \textit{function-Rips complex}.
Varying both $u$ and $\delta$ gives rise to the bivariate variant $\Rips[\bullet](\mapprxbull) \colon \prodPosR \to \sCplx$, $(u,\delta) \mapsto  \Rips[\delta](\mathbb{M}^{u})$, called the \textit{bivariate function-Rips complex}. 
\end{example}
\subsection{Persistent homotopy theory}
\label{sec:pers_homotopy_theory}
Conceptually speaking, a $\pos$-persistent homotopy type should be a $\pos$-indexed functor valued in the \textit{homotopy theory} or \textit{$\infty$-category} of spaces, not just the ordinary $1$-category of spaces (see also \cite{Jardine2020PersistentHomotopy,blumberg2024stability}).
As we do not expect familiarity with $\infty$-categorical language, we will instead work with the more elementary notion of relative categories, which we refer to as homotopy theories here.\footnote{This is justified insofar as relative categories provide a model for $(\infty,1)$-categories \cite{LurieHTT,BarwickKan2012RelativeCategories,Bergner2009Survey}.} 
\begin{definition}
	By a \textit{homotopy theory} we will mean a relative category; that is, a pair consisting of a category $\cat[C]$ and a wide subcategory $W \subset \cat[C]$. The morphisms in $W$ are called \textit{weak equivalences}. A morphism in $\cat[C]$ that is in $W$ will often be denoted by the symbol $\xrightarrow{\simeq}$.
\end{definition}
\begin{remark}
   By taking $W$ to be the class of identity morphisms in $\cat[C]$, we can treat any ordinary category (1-category) $\cat[C]$ as a homotopy theory ($\infty$-category) given by $(\cat[C], \{1_{c} \mid c\in \cat[C]\} )$. In this sense, any statement made concerning general homotopy theories in this article also applies to ordinary categories such as vector spaces. In particular, all of the interleaving results discussed in \cref{section:shrinking_trick} specialize to results for ordinary persistence modules.
\end{remark}
\begin{example}
Recall that a weak homotopy equivalence between topological spaces is a continuous map that induces isomorphisms on the sets of path components and on all homotopy groups. The category $\Top$ of topological spaces equipped with the subcategory of weak homotopy equivalences forms a homotopy theory which we denote by $\iSpaces$.\footnote{It follows as a consequence of Whitehead's theorem that this homotopy theory is equivalent to the one given by CW complexes and homotopy equivalences. 
}
\end{example}
\begin{notation}
	Given a homotopy theory $\mathcal{C} = ( \cat[C], W)$ and another category $\mathbb{I}$, we denote by $\mathcal{C}^{\mathbb{I}}$ the homotopy theory $(\cat[C]^{\mathbb{I}}, W^{\mathbb{I}})$, where $W^{\mathbb{I}}$ is the wide subcategory of $\cat[C]^{\mathbb{I}}$ consisting of those natural transformations $\varphi \colon F \Rightarrow G$ such that for every object $i \in \mathbb{I}$, the morphism $\varphi_i \colon F(i) \to G(i)$ is in $W$\footnote{Note that this will generally only produce the $\infty$-categorical functor category when $\mathcal{C}$ is sufficiently well-behaved, for example, when it extends to a model category, as is the case for our examples (\cite{LurieHTT}).}. To indicate that we study functors (persistent objects) $F \colon \mathbb{I} \to \cat[C]$ in the context of the whole homotopy theory $\mathcal{C}$ we will often use the notation $F \colon \mathbb{I} \to \mathcal{C}$. 
\end{notation}
\begin{example}
	When $\mathcal{C} = \iSpaces$ is the homotopy theory of spaces, then $\iSpaces^{\pos}$ will be referred to as the \textit{$\pos$-persistent homotopy theory of spaces}.
\end{example}
\begin{notation}
	Given a homotopy theory $\mathcal{C} = ( \cat[C], W)$, one can associate to it its homotopy category $\ho(\mathcal{C})$, defined as the ($1$-categorical) localization $\cat[C][W^{-1}]$ of $\cat[C]$ at the weak equivalences $W$. 
\end{notation}
\begin{recollection}
	Together with the canonical localization functor, $\cat[C] \to \cat[C][W^{-1}] = \ho(\mathcal{C})$, $\ho(\mathcal{C})$ is characterized by the universal property that any functor $\cat[C] \to \cat[D]$ that sends weak equivalences to isomorphisms factors uniquely through $\ho(\mathcal{C})$. Explicitly, objects of $\ho(\mathcal{C})$ are the same as those of $\cat[C]$, and morphisms are given by equivalence classes of zig-zags of morphisms in $\cat[C]$, where backward-pointing arrows are weak equivalences.
	In this sense, $\ho(\mathcal{C})$ is the $1$-category obtained by formally inverting the weak equivalences.
    \end{recollection}
    \begin{notation}\label{notation:pers_homotopy_type}
We abuse notation insofar as we use the same symbols for objects and morphisms in $\cat[C]$ and their images in $\ho (\mathcal{C})$. In the context of the persistent homotopy theory $\iSpaces^{\pos}$, a persistent space or persistent simplicial complex $X^{\bullet}$ will often be referred to as a \textit{persistent homotopy type}. Note that while, set-theoretically speaking, the persistent homotopy type $X^{\bullet}$ is the same as its underlying persistent space, in the context of the homotopy category $\ho(\iSpaces^{\inPos})$, two persistent homotopy types are isomorphic when they are connected by a zig-zag of weak equivalences. 
\end{notation}
\begin{remark}
The reader not familiar with abstract homotopy theory should have the following important caveat in mind: Given some indexing category $I$ and a homotopy theory $\mathcal{C}$, there is a canonical comparison functor $\ho (\mathcal{C}^I) \to (\ho \mathcal{C})^I.$
This functor is generally far from being essentially surjective or fully faithful. For example, when $\mathcal{C} = \iSpaces$, the left-hand side encodes homotopy coherent diagrams (i.e., diagrams together with choices of homotopies of different dimensions), and the right-hand side encodes homotopy commutative diagrams. It is, by this point, a widely established fact that the former is the natural, \define{conceptually correct}, and mathematically richer setting to study diagrammatic phenomena in homotopy theory.
However, in some special examples, such as when $I = \mathbb{N}$ or a finite linear poset, the functor turns out to induce a bijection on isomorphism classes (see, for example, \cite[Lemma A.2]{MaderWaas2024}). This has the effect that in many $1$-parameter persistence settings, the difference can be neglected, at least on the object level.
\end{remark}
\subsection{Interlude: Some remarks on persistent homotopy types}
In this subsection, we recall some facts about the persistent homotopy theory of spaces. Although none of these will be needed for the proofs in this article, we expect that they will help to obtain a more concrete picture of persistent homotopy types and of the category \(\ho(\iSpaces^{\pos})\), for readers less familiar with abstract homotopy theory.
For instance, one may ask for a concrete description of the morphism sets in \(\ho(\iSpaces^{\pos})\). One way to obtain such a description is through the theory of combinatorial model categories and projective model structures (see \cite{Hirschhorn2003}), which we will use freely below. As the material in this section is primarily for the reader's intuition, and as the insights are standard from a model categorical perspective, we will keep the proofs to a minimum.
\begin{recollection}
    Recall that a topological cell complex is a space $C$ obtained by transfinitely iterated attachments of disks $D^n$ of varying dimensions along their boundaries $\partial D^n=S^{n-1}$; see \cite{Hirschhorn2018QuillenTop} for a precise definition.\footnote{Unlike for a CW complex, one does not require the cells to be attached in ascending order of dimension.}
A chosen cell structure on $C$ - encoded by the maps $D^n \to C$ arising through the above gluing process - determines a set-theoretic decomposition into disjoint open cells
\[
C=\coprod_i \sigma_i.
\]
We write $\sigma \preceq \tau$ if $\sigma \cap \overline{\tau}\neq\emptyset$, and denote by $\le$ the transitive closure of this relation.
\end{recollection}
 \begin{definition}
     Suppose we are given a pair $(C, C \xrightarrow{f} \pos)$, with $C \in \Top$ a cell complex with a fixed cell decomposition $C= \coprod_{i}\sigma_i$ such that $f$ is constant on open cells and monotone with respect to the partial order $\leq$ on open cells, when seen as a map from open cells into $\pos$. Denote by $C^{\bullet}$ the associated persistent space, defined as in \cref{ex:sublevel_pers_space}. A persistent space $C^{\bullet}$ arising in this manner will be referred to as a \define{filtered cell complex}.
 \end{definition}
 \begin{remark}\label{rem:cof_gen}
     For the reader familiar with the theory of cofibrantly generated model categories, we note that filtered cell complexes are precisely the absolute cell complexes with respect to the standard cofibrant generators in the projective model structure (see \cite[Thm. 11.6.1]{Hirschhorn2018QuillenTop}), inherited from the standard cofibrant generators of the Quillen model structure (see \cite{Hirschhorn2003} and also \cite[Ch. 9]{waas2025stratified}).
 \end{remark}
 \begin{notation}
 Given $X^{\bullet} \in \Top^{\pos}$, we denote by $X^{\bullet} \times [0,1]$ the persistent space obtained by taking indexwise products with the unit interval. This persistent space will be referred to as the \define{persistent cylinder}, and morphisms $X^{\bullet} \times [0,1] \to Y^{\bullet}$ will be referred to as \define{persistent homotopies}.
 Given another persistent space $Y^{\bullet} \in \Top^{\pos}$, we denote by $[X^{\bullet}, Y^{\bullet}]$ the set of persistent homotopy classes, i.e., the quotient of $ \Top^{\pos}(X^{\bullet}, Y^{\bullet})$ under the equivalence relation 
 \[
 h_0 \sim h_1 \iff \textnormal{There exists } H \colon X^{\bullet} \times [0,1] \to Y^{\bullet} \textnormal{ such that } H_0 =h_0 \textnormal{ and } H_1=h_1.
 \]
 \end{notation}
 \begin{remark}\label{rem:pers_is_filtered}
     If $X^{\bullet}$ and $Y^{\bullet}$ both arise from pairs $(X, f\colon X \to \pos)$ and $(Y, g\colon Y \to \pos)$, then $[X^{\bullet}, Y^{\bullet}]$ is in canonical bijection with equivalence classes of maps $h \colon X \to Y$ such that $g(h(x)) \leq f(x)$, under homotopies $H \colon X \times [0,1] \to Y$ such that $g(H(x,t)) \leq f(x)$.
 \end{remark}
 \begin{block}
     It is not hard to see that persistently homotopic arrows are identified in $\ho(\iSpaces^{\pos})$. Hence, there are canonical maps
     \[
     [X^{\bullet}, Y^{\bullet}] \to \ho( \iSpaces^{\pos})(X^{\bullet}, Y^{\bullet}).
     \]
     However, these maps are generally neither injective nor surjective. Nevertheless, in the case of filtered cell complexes, this turns out to be the case. In fact, the theory of projective model structures guarantees that one has the following two results:
  \end{block}
 \begin{proposition}\label{prop:compute_mapping_spaces}
     Let $Y^{\bullet} \in \Top^{\pos}$ and $C^{\bullet}$ be a filtered cell complex. Then the canonical natural map 
     \(
     [C^{\bullet}, Y^{\bullet}] \to \ho( \iSpaces^{\pos})(C^{\bullet}, Y^{\bullet})
     \)
     is a bijection.
 \end{proposition}
 \begin{proof}
     This is the classical characterization of hom-sets in the homotopy category between a cofibrant and a fibrant object, together with the fact that every object in the projective model structure (with the Quillen model structure on $\Top$) is fibrant, and every filtered cell complex is cofibrant (see \cref{rem:cof_gen}).
 \end{proof}
 \begin{proposition}\label{prop:cof_replacement}
     Given any $X^{\bullet} \in \Top^{\pos}$, there exists a filtered cell complex $C^{\bullet} \in \Top^{\pos}$, together with a weak equivalence $C^{\bullet} \xrightarrow{\simeq} X^{\bullet}$.
 \end{proposition}
 \begin{proof}
     This follows by the small object argument (see \cite[§10.5]{Hirschhorn2003}) together with \cref{rem:cof_gen}.
 \end{proof}
 \begin{block}
     Together, these two results have the following consequence. Denote by $\ho \cat[Cell](\pos)$ the category of ($\pos$)-filtered cell complexes with morphisms given by persistent homotopy classes. By \cref{rem:pers_is_filtered}, we may equivalently think of the hom-sets as certain homotopy classes of filtration-preserving maps.
 \end{block}
 \begin{theorem}\label{thm:equ_with_bifib}
     The canonical functor 
     \[
     \ho \cat[Cell](\pos) \to \ho( \iSpaces^\pos) 
     \]
     is an equivalence of categories.
 \end{theorem}
 \begin{block}
     In this sense, one can equivalently use the setting of filtered cell complexes and filtered homotopies to perform persistent homotopy theory.
 \end{block}
 \begin{remark}
    Let us end this section with a remark on the purpose of persistent homotopy theory.
     A common misconception about persistent homotopy theory seems to be that it primarily aims to compute persistent versions of homotopy groups from data. From a practical computation and interpretation perspective, this does not currently seem like a promising approach. Instead, the role of persistent homotopy theory is to provide a general level at which one can prove inference results that then descend to computable invariants such as persistent homology, as well as other invariants. One simple and complementary invariant to persistent homology is the persistent path components.
 \end{remark}
 \begin{example}
Given a persistent homotopy type $X^{\bullet} \in \iSpaces^{\pos}$, one can consider the persistent set $\pi_0(X^{\bullet})$, which is obtained by computing path-components indexwise. This is a more informative invariant than the $0$-th persistent homology (compare \cite{Curry2018Fiber,CurryHangMioNeedhamOkutan2022DMT,BeersGrindstaff2025Intrinsic}). Indeed, 
\[
H_0(X^{\bullet}, \mathbb{k}) \cong \mathbb{k}\langle \pi_0(X^{\bullet}) \rangle
\]
where $\mathbb{k}\langle - \rangle$ denotes the functor mapping a set $S$ to the free $\mathbb{k}$-vector space on $S$. If $X^{\bullet} = \Rips[\bullet](\metapprx)$ is the Rips-persistent homotopy type of a metric space $\metapprx$, then $\pi_0(X^{\bullet})$ encodes the same information as the merge tree of $\metapprx$. This is generally richer information. Consider, for example, the illustration of two finite point clouds in $\mathbb{R}^1$, their merge trees, and their $0$-persistent homology in \cref{fig:mergetrees}.
 \end{example}
 
\begin{figure}[ht]
\centering
\resizebox{\textwidth}{!}{
\begin{tikzpicture}[
    dot/.style={circle,fill=black,inner sep=1.5pt},
    axis/.style={-{Latex[length=2mm]}, thin},
    bar/.style={line width=0.9pt},
    panel/.style={draw, rounded corners=2pt, thin},
    assign/.style={|-{Latex[length=2.2mm]}, line width=0.9pt},
    every node/.style={font=\scriptsize}
]

\def\W{4.6}      
\def\H{3.6}      
\def\GapX{1.4}   
\def\GapY{1.2}   

\pgfmathsetmacro{\XA}{0}
\pgfmathsetmacro{\XB}{\XA+\W+\GapX}
\pgfmathsetmacro{\XC}{\XB+\W+\GapX}

\pgfmathsetmacro{\YB}{0}
\pgfmathsetmacro{\YT}{\H+\GapY}

\node[font=\normalsize] at (\XA+0.5*\W,\YT+\H+0.55) {point cloud in $\mathbb{R}^1$};
\node[font=\normalsize] at (\XB+0.5*\W,\YT+\H+0.55) {merge tree};
\node[font=\normalsize] at (\XC+0.5*\W,\YT+\H+0.55) {$H_0$ barcode};


\draw[panel] (\XA,\YT) rectangle ++(\W,\H);
\draw[panel] (\XB,\YT) rectangle ++(\W,\H);
\draw[panel] (\XC,\YT) rectangle ++(\W,\H);

\draw[panel] (\XA,\YB) rectangle ++(\W,\H);
\draw[panel] (\XB,\YB) rectangle ++(\W,\H);
\draw[panel] (\XC,\YB) rectangle ++(\W,\H);

\draw[assign] (\XA+\W+0.18,\YT+0.5*\H) -- (\XB-0.18,\YT+0.5*\H)
  node[midway,above] {$\pi_0 \circ \Rips[\bullet]$};

\draw[assign] (\XA+\W+0.18,\YB+0.5*\H) -- (\XB-0.18,\YB+0.5*\H)
  node[midway,above] {$\pi_0 \circ \Rips[\bullet]$};

\draw[assign] (\XB+\W+0.18,\YT+0.5*\H) -- (\XC-0.18,\YT+0.5*\H)
  node[midway,above] {$\mathbb{k} \langle - \rangle$};

\draw[assign] (\XB+\W+0.18,\YB+0.5*\H) -- (\XC-0.18,\YB+0.5*\H)
  node[midway,above] {$\mathbb{k} \langle - \rangle$};

\node[font=\Large] at (\XC+0.5*\W,\YB+\H+0.5*\GapY) {$=$};

\begin{scope}[shift={(\XA,\YT)}]
  \draw[axis] (0.45,1.95) -- (4.10,1.95) node[below] {$\mathbb{R}$};

  \foreach \x/\lab in {0.55/1,1.01/2,1.93/3,2.39/4,3.77/5}{
    \draw (\x,2.03) -- (\x,1.87);
    \node[dot,label=above:$\lab$] at (\x,1.95) {};
  }

  \node at (0.78,1.55) {$1$};
  \node at (1.47,1.55) {$2$};
  \node at (2.16,1.55) {$1$};
  \node at (3.08,1.55) {$3$};
\end{scope}
\begin{scope}[shift={(\XC,\YT)}]
  \draw[axis] (0.45,0.55) -- (3.90,0.55) node[below] {$\delta$};
  \foreach \x/\lab in {0.45/0,1.35/1,2.25/2,3.15/3}{
    \draw (\x,0.61) -- (\x,0.49);
    \node[below=2pt] at (\x,0.49) {$\lab$};
  }

  \draw[bar] (0.45,2.95) -- (1.35,2.95);
  \draw[bar] (0.45,2.45) -- (1.35,2.45);
  \draw[bar] (0.45,1.95) -- (2.25,1.95);
  \draw[bar] (0.45,1.45) -- (3.15,1.45);
  \draw[bar,->] (0.45,0.95) -- (3.80,0.95);
\end{scope}

\begin{scope}[shift={(\XA,\YB)}]
  \draw[axis] (0.45,1.95) -- (4.10,1.95) node[below] {$\mathbb{R}$};

  \foreach \x/\lab in {0.55/1,1.01/2,1.93/3,3.31/4,3.77/5}{
    \draw (\x,2.03) -- (\x,1.87);
    \node[dot,label=above:$\lab$] at (\x,1.95) {};
  }

  \node at (0.78,1.55) {$1$};
  \node at (1.47,1.55) {$2$};
  \node at (2.62,1.55) {$3$};
  \node at (3.54,1.55) {$1$};
\end{scope}
\begin{scope}[shift={(\XB,\YT)}]
  \draw[axis] (0.45,0.55) -- (3.90,0.55) node[below] {$\delta$};
  \foreach \x/\lab in {0.45/0,1.35/1,2.25/2,3.15/3}{
    \draw (\x,0.61) -- (\x,0.49);
    \node[below=2pt] at (\x,0.49) {$\lab$};
  }

  \foreach \y/\lab in {2.95/1,2.45/2,1.95/3,1.45/4,0.95/5}
    \node[left] at (0.37,\y) {$\lab$};

  \draw (0.45,2.95) -- (1.35,2.95);
  \draw (0.45,2.45) -- (1.35,2.45);
  \draw (1.35,2.95) -- (1.35,2.45);

  \draw (0.45,1.95) -- (1.35,1.95);
  \draw (0.45,1.45) -- (1.35,1.45);
  \draw (1.35,1.95) -- (1.35,1.45);

  \draw (1.35,2.70) -- (2.25,2.70);
  \draw (1.35,1.70) -- (2.25,1.70);
  \draw (2.25,2.70) -- (2.25,1.70);

  \draw (0.45,0.95) -- (3.15,0.95);
  \draw (2.25,2.20) -- (3.15,2.20);
  \draw (3.15,2.20) -- (3.15,0.95);

  \draw (3.15,1.575) -- (3.80,1.575);
\end{scope}

\begin{scope}[shift={(\XB,\YB)}]
  \draw[axis] (0.45,0.55) -- (3.90,0.55) node[below] {$\delta$};
  \foreach \x/\lab in {0.45/0,1.35/1,2.25/2,3.15/3}{
    \draw (\x,0.61) -- (\x,0.49);
    \node[below=2pt] at (\x,0.49) {$\lab$};
  }

  \foreach \y/\lab in {2.95/1,2.45/2,1.95/3,1.45/4,0.95/5}
    \node[left] at (0.37,\y) {$\lab$};

  \draw (0.45,2.95) -- (1.35,2.95);
  \draw (0.45,2.45) -- (1.35,2.45);
  \draw (1.35,2.95) -- (1.35,2.45);

  \draw (1.35,2.70) -- (2.25,2.70);
  \draw (0.45,1.95) -- (2.25,1.95);
  \draw (2.25,2.70) -- (2.25,1.95);

  \draw (0.45,1.45) -- (1.35,1.45);
  \draw (0.45,0.95) -- (1.35,0.95);
  \draw (1.35,1.45) -- (1.35,0.95);

  \draw (2.25,2.325) -- (3.15,2.325);
  \draw (1.35,1.20) -- (3.15,1.20);
  \draw (3.15,2.325) -- (3.15,1.20);

  \draw (3.15,1.7625) -- (3.80,1.7625);
\end{scope}
\begin{scope}[shift={(\XC,\YB)}]
  \draw[axis] (0.45,0.55) -- (3.90,0.55) node[below] {$\delta$};
  \foreach \x/\lab in {0.45/0,1.35/1,2.25/2,3.15/3}{
    \draw (\x,0.61) -- (\x,0.49);
    \node[below=2pt] at (\x,0.49) {$\lab$};
  }

  \draw[bar] (0.45,2.95) -- (1.35,2.95);
  \draw[bar] (0.45,2.45) -- (1.35,2.45);
  \draw[bar] (0.45,1.95) -- (2.25,1.95);
  \draw[bar] (0.45,1.45) -- (3.15,1.45);
  \draw[bar,->] (0.45,0.95) -- (3.80,0.95);
\end{scope}

\end{tikzpicture}
}
\caption{Two point clouds in $\mathbb{R}^1$ with different merge trees but identical $H_0$ barcode.}\label{fig:mergetrees}
\end{figure}
\subsection{Persistent simplicial sets}\label{subsec:simplicial_sets}
At least half of the persistent homotopy types we are studying in this article arise from purely combinatorial data provided in the form of persistent simplicial complexes. It is thus convenient to have a model for the persistent homotopy theory of spaces that is more combinatorial in nature. This can be achieved by working with simplicial sets. The reader not familiar with the theory of simplicial sets can treat them as a black box that extends the category of simplicial complexes in a convenient way (see \cite{GoerssJardine2009,Jardine2020PersistentHomotopy} for an introduction). 
\begin{recollection}
Recall that, conceptually speaking, a simplicial set is like a simplicial complex, where the simplices are ordered, and one allows for faces of simplices to collapse to lower dimensions. Categorically, this idea can be formalized as follows: Denote by $\Delta$ the category of finite linear posets $[n] = \{ 0 \leq \dots \leq n\}$, for $n \geq 0$, with order-preserving maps. The category of simplicial sets $\sSet$ is the category of functors from $\Delta^{\op}$ into $\Set$, i.e., $\Set^{\Delta^{\op}}$. Given a simplicial set $X \in \sSet$ and $n \geq 0$, the set $X([n])$ is denoted $X_n$, and called the set of $n$-simplices of $X$. A simplex $\sigma \in X_n$ that is not in the image of a structure map $X_k \to X_n$, for some $k<n$, is called non-degenerate.
\end{recollection}
\begin{example}\label{ex:function_rips_sset} 
As a simplicial set, the function-Rips complex at $u \in \RN$ and $\delta \geq 0$ can be modeled by a simplicial set with $n$-simplices
\[
\Rips[u,\delta]\fmet_n = \{ (x_0, \dots, x_n) \mid x_i \in M; d(x_i,x_j) < \delta; f(x_i) \leq u \textnormal{ for all }i,j \in [n] \},
\]
and functoriality on $\Delta^{\op}$ given by precomposition.
\end{example}
\begin{notation}
 Given $n \in \mathbb{N}$, the image of $[n]$ under the Yoneda embedding $\Delta \hookrightarrow \Set^{\Delta^{\op}}$, $[k] \mapsto \Delta([k],[n])$, is denoted by $\Delta^n$ and referred to as the $n$-simplex. The functor $[n] \mapsto \Delta^n$ defines a fully faithful embedding $\Delta \hookrightarrow \sSet$, by which we treat $\Delta$ as a subcategory of $\sSet$. 
\end{notation}
\begin{recollection}
	By the Yoneda lemma, the set of $n$-simplices of a simplicial set $X$ is in canonical bijection with the set of simplicial maps $\Delta^n \to X$. From this perspective, the \define{$0$-dimensional faces or vertices} of a simplex $\sigma \colon \Delta^n\to X$ are given by $n+1$ (not necessarily distinct) compositions $x_k \colon \Delta^0 \xrightarrow{i_k} \Delta^n \to X$, with $i_k$ specified by $0 \mapsto k$ for $k \in [n]$. This equips the vertices of a simplex $\sigma$ with a canonical ordering $x_0, \dots, x_n$. When we speak of a $1$-simplex $\sigma$ from $x$ to $y$, we mean a $1$-simplex $\sigma \colon \Delta^1 \to X$ such that the vertices of $\sigma$ are given by $x$ and $y$, in that order.
\end{recollection}
\begin{recollection}
	Simplicial sets admit a topological realization functor. Observe that $\Delta$ embeds into $\sCplx$ by sending $[n]$ to the standard simplex $\Delta^n_c$. The composition $\Delta \to \sCplx \xrightarrow{\real{-}} \Top$ defines a topological realization functor on $\Delta \subset \sSet$, which 
    extends canonically to a colimit-preserving functor $\real{-} \colon \sSet \to \Top$, so that the realization of a simplicial set $X$, $\real{X}$, is glued from realizations of its simplices.
\end{recollection}
One crucial reason why simplicial sets are so useful is that they can be used to define a homotopy theory equivalent to that of topological spaces.
\begin{recollection}\label{rec:equ_with_ss}
A simplicial map $\varphi \colon X \to Y$ in $\sSet$ is called a weak homotopy equivalence if its topological realization $\real{\varphi} \colon \real{X} \to \real{Y}$ is a homotopy equivalence. Denote the wide subcategory of weak homotopy equivalences by $W_{\textnormal{Kan}}$, and the resulting relative category $(\sSet, W_{\textnormal{Kan}})$ by $\sSetK$. It is a fundamental fact of homotopy theory, that the topological realization functor $\real{-} \colon \sSet \to \Top$ then defines a so-called \textit{equivalence of homotopy theories ($\infty$-categories)} $\sSetK \to \iSpaces$ (see \cite{Quillen1967HomotopicalAlgebra}). Conceptually speaking, this means that any homotopy-theoretic construction or argument concerning $\iSpaces$ can equivalently be performed in $\sSetK$.  For our purposes, however, it suffices to observe that the functor $\real{-} \colon \sSet^{\pos} \to \Top^{\pos}$ descends to an equivalence of categories $\real{-} \colon \ho (\sSetK^{\pos} ) \to \ho (\iSpaces^{\pos})$,  compatible with any reparametrization of the indexing poset $\pos$.
\end{recollection}
\begin{recollection}
    In the same way as homotopies induce identifications of maps in the homotopy category $\ho ( \iSpaces)$, morphisms of persistent simplicial sets of the form $X^{\bullet} \times \Delta^1 \to Y^{\bullet}$  (where $X^{\bullet} \times \Delta^1$ denotes the indexwise product) -- so-called \textit{elementary homotopies} --  induce identifications of persistent simplicial maps in $\ho ( \sSetK^{\pos})$.  
\end{recollection}
\begin{recollection}\label{rec:equv_model_simp}
We frequently want to treat simplicial complexes as simplicial sets. There is a canonical fully faithful embedding $\sNerve \colon \sCplx \hookrightarrow \sSet$, mapping a simplicial complex, $K$, to the simplicial set given by $\sNerve(K)_n = \sCplx (\Delta_c^n, K)$, functorial in the obvious way in $n$ and $K$ (see \cref{rem:adjoint_to_NS}). For example, the simplicial set obtained by applying $\sNerve$ to a function-Rips complex is precisely the one described in \cref{ex:function_rips_sset}.  
We will usually omit $\sNerve$ from the notation. Observe that from a homotopy-theoretic perspective this is justified by the fact that, given $K \in \sCplx$, there is a canonical homotopy equivalence $\real{\sNerve{(K)}} \xrightarrow{\simeq}\real{K}$ (see \cite{OtterMagnitudePersistence2022,CamarenaSsetsFromComplexes}, \cref{rem:equ_of_nerve_w_complex}) even though the equality $\real{K} \cong \real{\sNerve(K)}$ does not hold on the homeomorphism level.
Hence, for our purposes, we can freely identify the two realizations. \end{recollection}
\begin{remark}
    There are several reasons why simplicial sets have largely replaced simplicial complexes for the purpose of homotopy theory. It is not the goal of this article to argue this point in detail. Let us, however, point out three explicit points in which the shift to simplicial sets becomes relevant in this article, which would have required a lot of additional technical effort to express in the language of simplicial complexes:
    \begin{enumerate}
        \item In \cref{section:shrinking_for_rips}, we will make use of the last vertex map, a natural transformation $\sd X \to X$, where $\sd X$ denotes the barycentric subdivision of a simplicial set (see \cref{ex:barycentric_subdiv}). To procure such a map for simplicial complexes, one needs to choose an ordering of the vertices. However, given such choices, the last vertex map will generally not be natural with respect to simplicial maps that are not order-preserving.
        \item Simplicial sets form a presheaf category. This makes it generally easy to define colimit-preserving functors on $\sSet$ through the techniques of left Kan extension. In \cref{sec:local_stability_function_rips}, these techniques are used to define an alternative subdivision functor, which we use in the proof of the perturbative stability theorem for the function-Rips complex.
        \item Simplicial sets admit a so-called model structure, which allows for explicit computations involving the homotopy theory $\sSetK$. The extremely well-developed techniques of model categories (see, for example, \cite{Hirschhorn2003,LurieHTT}) allow for accessible control over the associated homotopy theory defined by the weak equivalences. 
    \end{enumerate}
\end{remark}
\subsection{Interleavings in the homotopy category}
\label{sec:interleavings_homotopy_cat}
One of the core advantages of persistent settings is that they allow for approximate notions of equivalence, so-called interleavings (see \cite{deSilvaMunchStefanou2018} for a more general setting). These approximate notions of equivalence allow for the \define{quantitative treatment} that is necessary to use algebraic or homotopical notions in a data analysis context.
In the following, $\mathcal{C}$ will denote some homotopy theory. To define interleavings, we need the following notation.
\begin{notation}
	For the remainder of this subsection, we fix natural numbers $N,N' \geq 0$.
By $\RNPlus$, we denote the set of vectors in $\RN$ with non-negative entries.
From here on out, $\inPos \subset \mathbb{R}^N$ and $\inPos' \subset \mathbb{R}^{N'}$ will always denote upsets. Recall that an upset $\inPos \subset \RN$ is a subset fulfilling $x \in \inPos, x \leq y \implies y \in \inPos$. Equivalently, these are the subsets of Euclidean space that are closed under addition with vectors in $\RNPlus$. 
\end{notation}
\begin{notation}
	Suppose we are given a map of posets $S \colon \inPos \to\inPos'$ and a persistent object $F \colon \inPos' \to \mathcal{C}$. We write $F^{S(\bullet)}$ to denote the persistent object obtained by precomposing $F$ with $S$. In the special case where $S$ is an inclusion of posets, we use the notation $F^{\bullet}|_{\inPos}$ to denote the restriction of $F$ to $\inPos$. Suppose we are given another such map $S' \colon\inPos \to\inPos'$ such that $S(x) \leq S'(x)$, for all $x \in\inPos$. In this case, we will use the notation $s \colon F^{S(\bullet)} \to F^{S'(\bullet)}$ to denote the natural transformation induced by the relations $S(x) \leq S'(x)$. The most relevant example of this will be the case where we are given a vector $\vecepsilon \in \RNPlus$ and write $- + \vecepsilon \colon \inPos \to \inPos$ for the map of posets given by $x \mapsto  x + \vecepsilon$.
 When $\varepsilon$ is just a non-negative scalar, we will abuse notation insofar as we write $- + \varepsilon$ for the map $x \mapsto x + (\varepsilon, \dots, \varepsilon)$.
\end{notation}
\begin{remark}
    Much of what we describe below could also be proven and investigated in the context of a poset with a flow (see \cite{deSilvaMunchStefanou2018}). We will, however, restrict to the setting of subsets of $\RN$ here, primarily to limit additional technicalities arising from non-associative flows. 
\end{remark}
\begin{definition}\label{rec:interleaving_dist} Let $\vecepsilon \in \RNPlus$.
An \textit{$\vecepsilon$-interleaving in the homotopy category} between $F^{\bullet}, G^{\bullet} \colon \mathbb{U} \to \mathcal{C}$ consists of morphisms
$\varphi \colon F^{\bullet} \to G^{\bullet \flow {\vecepsilon}} \quad \text{ and } \quad \psi \colon G^{\bullet} \to F^{\bullet \flow {\vecepsilon}}$ in $\ho( \mathcal{C}^{\inPos})$
such that the diagrams 
\begin{diagram}
	{F^{\bullet}} && {F^{\bullet \flow 2{\vecepsilon}}} & {G^{\bullet}} && {G^{\bullet \flow 2{\vecepsilon}}} \\
	& {G^{\bullet \flow {\vecepsilon}}} &&& {F^{\bullet \flow {\vecepsilon}}}
	\arrow["s", from=1-1, to=1-3]
	\arrow["\varphi"', from=1-1, to=2-2]
	\arrow["s", from=1-4, to=1-6]
	\arrow["\psi"', from=1-4, to=2-5]
	\arrow["{\psi^{\flow {\vecepsilon} }}"', from=2-2, to=1-3]
	\arrow["{\varphi^{\flow {\vecepsilon} }}"', from=2-5, to=1-6]
\end{diagram}
in $\ho( \mathcal{C}^{\inPos})$ commute. We will denote such interleavings in the form $\varphi \colon F \simeq_{\vecepsilon} G \colon \psi.$
\end{definition}
We are really primarily concerned with two special cases of this vectorial definition, obtained as follows: 
For the purpose of our investigation, we are in the situation where $N' = N +1$, and $\inPos' = \prodPos$. In this case, we will group the interleaving parameters as follows.
\begin{definition}
Let $\varepsilon \geq 0$. By an $\varepsilon$-interleaving in the homotopy category between $F^{\bullet}, G^{\bullet} \colon \inPos \to \mathcal{C}$, we mean an $\vecepsilon$-interleaving in the homotopy category, where $\vecepsilon = (\varepsilon, \dots, \varepsilon)$. We use notation analogous to the vectorial case.
\end{definition}
\begin{definition}
	Assume that $\inPos' = \prodPos$ and let $\varepsilon_1, \varepsilon_2 \geq 0$. By a $(\varepsilon_1, \varepsilon_2)$-interleaving in the homotopy category between $F^{\bullet}, G^{\bullet} \colon \inPos' \to \mathcal{C}$, we mean $\vecepsilon$-interleaving in the homotopy category, where $\vecepsilon = (\varepsilon_1, \varepsilon_2, \dots, \varepsilon_2)$. We use the notation $\varphi \colon F \simeq_{\varepsilon_1, \varepsilon_2} G \colon \psi$ to refer to such an interleaving.
\end{definition}
\begin{notation}
	In order not to overcrowd notation, we will often omit the prefix $\varepsilon$ (or $(\varepsilon_1, \varepsilon_2)$) from the word $\varepsilon$-interleaving when we write interleavings in the form $\varphi \colon F \simeq_{\varepsilon} G \colon \psi$. The interleaving parameters will always be clear from the subscript.
\end{notation}
\begin{remark}
It can be useful to allow for formal interleavings of degree $\varepsilon = \infty$ (not to be confused with the interleaving distances below being infinite), which correspond to equivalences of the (infinity categorical) colimit $\varinjlim F^{\bullet} \simeq \varinjlim G^{\bullet}$. We will not make use of this here, though.
\end{remark}
\begin{remark}\label{rem:properties_of_int}
	Two interleavings $\varphi \colon F^{\bullet} \simeq_{\vecepsilon} G^{\bullet} \colon \psi$ and $\varphi' \colon G^{\bullet} \simeq_{\vecepsilon\,'} J^{\bullet}\colon \psi'$ compose to an interleaving $(\varphi')^{+\vecepsilon}\circ \varphi \colon F^{\bullet} \simeq_{\vecepsilon+\vecepsilon \, '} J^{\bullet} \colon \psi^{+\vecepsilon \,'} \circ \psi'$ (see, for example, \cite{deSilvaMunchStefanou2018}). Observe also that any interleaving $\varphi \colon F^{\bullet} \simeq_{\vecepsilon} G^{\bullet} \colon \psi$ gives rise to an interleaving $s \circ \varphi \colon F^{\bullet} \simeq_{\vecepsilon \,''} G^{\bullet} \colon s \circ \psi$ for any $\vecepsilon \,'' \geq \vecepsilon$. Analogous statements hold for the coupled parameter versions defined above.
\end{remark}
\begin{example}
	 When $\inPos = \mathbb{R}^0 = \{0\}$, we can identify $\mathcal{C}^{\inPos} = \mathcal{C}$ and $(-)+\varepsilon \colon \inPos \to \inPos$ is the identity. Then, an $\varepsilon$-interleaving in $\ho( \mathcal{C})$ is the same as an isomorphism in $\ho( \mathcal{C})$. 
\end{example}
\begin{example}
    It will be helpful to decode the definition of a $\vecepsilon$-interleaving for the special case of the homotopy theory $\mathcal{C} = \iSpaces$. Let $X^{\bullet}$ and $Y^{\bullet}$ be $\inPos$-persistent spaces. Let us, furthermore, assume that both are filtered cell complexes. By \cref{prop:cof_replacement}, this can always be achieved, up to replacing $X^{\bullet}$ and $Y^{\bullet}$ by weakly equivalent persistent spaces. Then it follows from \cref{prop:compute_mapping_spaces} that an interleaving in the homotopy category between $X^{\bullet}$ and $Y^{\bullet}$ is given by the data of the persistent homotopy classes of persistent maps $\varphi \colon X^{\bullet} \to Y^{\bullet \flow {\vecepsilon}}$ and $\psi \colon Y^{\bullet} \to X^{\bullet \flow \vecepsilon}$ in $\Top^{\inPos}$ such that the diagrams 
	\begin{diagram}
	{X^{\bullet}} && {X^{\bullet \flow 2{\vecepsilon}}} & {Y^{\bullet}} && {Y^{\bullet \flow 2{\vecepsilon}}} \\
	& {Y^{\bullet \flow {\vecepsilon}}} &&& {X^{\bullet \flow {\vecepsilon}}}
	\arrow["s", from=1-1, to=1-3]
	\arrow["\varphi"', from=1-1, to=2-2]
	\arrow["s", from=1-4, to=1-6]
	\arrow["\psi"', from=1-4, to=2-5]
	\arrow["{\psi^{\flow {\vecepsilon} }}"', from=2-2, to=1-3]
	\arrow["{\varphi^{\flow {\vecepsilon} }}"', from=2-5, to=1-6]
\end{diagram}
	commute up to persistent homotopy.
\end{example}
\begin{construction}\label{con:inho_distance}
    It follows from \cref{rem:properties_of_int} that interleavings in the homotopy category give rise to a distance 
    \begin{align*}
              \dIntHo \colon  \iSpaces^{\inPos} \times \iSpaces^{\inPos} &\to [0, \infty] \\
              (F^{\bullet},G^{\bullet}) &\mapsto \inf\{ \varepsilon \geq 0 \mid F^{\bullet} \simeq_{\varepsilon} G^{\bullet} \}
    \end{align*}
    that fulfills the triangle inequality and takes the value $0$ on weakly equivalent objects. It is referred to as the \define{interleaving distance in the homotopy category} (see \cite{LanariScoccola2023Rectification}).
\end{construction}
\begin{remark}
	There are two alternative notions of homotopical interleaving and distances for $\iSpaces^\inPos$ that have been considered in the literature (see \cite{BlumbergLesnick2023HomotopyInterleaving,LanariScoccola2023Rectification}). The first is given by interleavings in $(\ho \iSpaces)^{\inPos}$. As it is defined on $(\ho \iSpaces)^{\inPos}$ rather than $\ho(\iSpaces^{\inPos})$, it misses crucial features of persistent homotopy types whenever $\inPos$ is not a subset of $\mathbb{R}^1$, i.e., in any multiparameter setting. It therefore does not seem to be a suitable candidate to be considered for a good notion of interleaving for persistent homotopy types (see also \cite{LanariScoccola2023Rectification}). The second notion, referred to as \define{homotopy interleavings} in \cite{BlumbergLesnick2023HomotopyInterleaving}, is more subtle and relevant. Roughly speaking, whereas interleavings in the homotopy category require the commutativity of the triangle up to some persistent homotopy, homotopy interleavings require that the relevant persistent homotopies also fulfill an infinite tower of higher coherence relations (see \cite{BlumbergLesnick2023HomotopyInterleaving} for more details). Thus, the resulting distance is generally larger and, in fact, is the largest distance bounded by the interleaving distance that is invariant under weak equivalences. The precise relationship between the two distances is still the topic of ongoing research. However, it is known that in the case of $1$-parameter persistence, the two distances are equivalent (see \cite{LanariScoccola2023Rectification}). Whether every result we obtain for interleavings in the homotopy category also holds for homotopy interleavings is an interesting question, which we do not address in this article.
\end{remark}
\section{The persistent Hausmann Theorem}
\label{sec:Hausmann}
In this section, we prove the following persistent version of Hausmann's Theorem.
\begin{recollection}
    Recall that a metric space $M$ is called \define{proper}, if every bounded closed subset in $M$ is compact. Also recall that every proper metric space is complete.
\end{recollection}
\begin{notation}
For the remainder of this section, let $\kappa \in \mathbb{R}$, and let $M$ be a proper $\cabK$-space. Furthermore we fix a Lipschitz function $f \colon M \to \RN$, with Lipschitz constant $\Lipf$, with respect to the $\infty$-norm on $\RN$.
Recall also the constant $\catRad$ defined in \cref{not:catRad}, as well as the Jung constant $\conShrM[\rho]$, defined in \cref{not:conJung}, depending on a parameter $\rho$ (in this section $\delta$).
\end{notation}
\begin{notation}
	 In this section, whenever we write $\real{\Rips(M)}$ we will mean the topological realization of the simplicial complex definition of the Rips complex.
\end{notation}
\begin{theorem}
\label{thm:persistent_hausmann}
	Let $\delta \leq \catRad$. Then there is an interleaving 
	\[
	\real{\Rips[\delta](\mbull)} \simeq_{\conShrM[\delta] \Lipf \delta} M^{\bullet}
	\]
	in the persistent homotopy category.
\end{theorem}
\begin{remark}
    We will only use properness at one point in this section, namely in \cref{lem:cont_of_karcher} to prove the continuity of Karcher means. In fact, one can prove a version of \cref{thm:persistent_hausmann} just under the assumption of completeness, at the cost of passing to $\delta < \catRad$, by using \cite[Prop. 24]{Yokota2016} in the proof of \cref{lem:cont_of_karcher} (for $\delta < \catRad$) instead. 
\end{remark}
\subsection{Metric thickenings}
The proof relies on the technique of metric thickenings, introduced in \cite{metric_reconstruction_via_optimal_transport}. We now recall the most relevant notions and results from \cite{metric_reconstruction_via_optimal_transport,Gillespie2024VietorisThickenings}.
\begin{recollection}\label{recol:metric_vietoris_rips} Let $\delta \geq 0$. The \textit{metric Vietoris-Rips thickening} of $M$, $\MRips(M)$, is the topological space defined as follows. The underlying set of $\MRips(M)$ is the same as the one of $\real{\Rips(M)}$. However, instead of equipping it with the weak topology coming from the geometric realization, the topology arises from a Wasserstein metric, defined as follows. We identify a point $\sum_{x \in M} \lambda_x x$ in the geometric realization of $\Rips(M)$ with a finite probability Borel measure on $M$, given by $\mu(S) = \sum_{x \in S} \lambda_x$. We then use the $1$-Wasserstein metric to define a topology on the space of such finite measures (see \cite{metric_reconstruction_via_optimal_transport}). \\
The identity on the level of sets defines a continuous bijection $\real{\Rips(M)} \to \MRips(M)$. Now, given a filtration function $f \colon M \to \RN$, we can turn $\MRips(M)$ into a persistent (filtered even) space, by filtering it by the metric thickenings $\MRips(M^{u})$, for $u \in \RN$. We denote the resulting persistent space by $\MRips(M^{\bullet})$. The continuous bijection $\real{\Rips(M)}\to  \MRips(M) $ then defines a comparison map of $\RN$-persistent homotopy types $\real{\Rips(\mbull)}\to \MRips(\mbull) $.
\end{recollection}
As an immediate corollary of \cite[Thm. 1]{Gillespie2024VietorisThickenings}, we obtain the following result.
\begin{proposition}\label{prop:metric_vietoris_rips_equ_rips} Let $\delta \geq 0$. 
	The canonical morphism $ \real{\Rips(\mbull)}\to \MRips(\mbull)$ is a weak equivalence in $\iSpaces^{\RN}$.
\end{proposition}
\begin{recollection}\label{rec:inclusion_of_space_metric_rips}
	The advantage of $\MRips(\mbull)$ over $\real{\Rips(\mbull)}$ is that the former comes with a canonical subspace inclusion $\iota \colon \mbull \hookrightarrow \MRips(\mbull)$, mapping $x \in M$ to the Dirac measure $\delta_x = 1 x$, which we will identify with $x$.
\end{recollection}
\subsection{A remark on the relevance of the choice of Wasserstein metric}
    We have stated that the topology of $\MRips(M)$ is induced by the $1$-Wasserstein distance above. It is useful to note, however, that the choice of $p=1$ has no impact on the resulting topology.
    \begin{notation}
    Given a measure $\mu \in \MRips(M)$, we write $\supp(\mu)$ for the support of $\mu$. In other words, if $\mu$ is given by a convex combination $\sum_{i=0}^{n} \lambda_i x_i$, with $\lambda_i >0$, then $\supp(\mu) = \{x_0, \dots,x_n\}$.
    \end{notation}
    \begin{notation}
        Given a metric space $X$ and $p \geq 1$, we denote by $d_p$ the $p$-Wasserstein distance on $\MRips(X)$.
    \end{notation}
    \begin{lemma}\label{lem:indep_of_dist}
        Given any metric space $X$, $\delta \geq 0$, and $p \geq 1$, the topology on $\MRips(X)$ induced by $d_p$ is independent of the choice of $p \geq 1$. 
    \end{lemma}
    \begin{proof}
    Let $1 \leq p<q$, as otherwise the statement is evident. 
    To see this, recall first the basic fact about Wasserstein distances that 
    \[
    d_p(\mu,\nu) \leq d_{q}(\mu,\nu)
    \]
    for $p \leq q$ for any probability measures on a metric space $M$. This shows that the $q$-topology is finer than the $p$-topology. Now, conversely, assume that $\mu_n \to \mu$ is a convergent sequence in $\MRips(X)$ in the $p$-Wasserstein topology. Suppose first that $\mu_n$ and $\mu$ are contained in a bounded subspace of $X$. Then a straightforward computation shows that
    \[
    d_{q}(\mu, \nu) \leq \diam(X)^{1-\frac{p}{q}} d_p(\mu, \nu)^{\frac{p}{q}},
    \]
    for any two probability measures $\mu$ and $\nu$ on $X$. Consequently, convergence $d_p(\mu_n, \mu) \to 0$ also implies $d_{q}(\mu_n, \mu) \to 0$.\\
    It remains to show that we can reduce to the bounded case. In fact, we show the following:
    For any $x \in \supp(\mu)$, there necessarily exists a sequence $x_n \in M$, with $x_n \in \supp(\mu_n)$, such that $d(x_n,x) \to 0$. 
    Note that, since we have assumed a global bound of $\delta$ on the diameter of the support of $\mu_n$, this implies the claim by the triangle inequality.
    Now, let us prove the remaining claim. Suppose, to the contrary and by restricting to an appropriate subsequence, that $d(y_n,x) \geq \varepsilon$, for all $n \geq 0$, and all $y_n \in \supp(\mu_n)$. For $n \geq 0$, fix a probability measure $\Theta_n$ on $X \times X$ with marginals $\mu_n$ and $\mu$. Then 
    \begin{align*}
        \int d(-,-)^p d \Theta_n &\geq  \int \mathbb{1}_{X \times \{x\}}d(-,-)^p d \Theta_n \\&
        = \int d(-,x)^p d \Theta_n(- \times \{x\}) \\ &\geq \int \varepsilon^p d \Theta_n(- \times \{x\}) \\
        &= \Theta_n(X \times \{x\}) \varepsilon^{p} \\ &= \mu(\{x\}) \varepsilon^p,
    \end{align*}
    and thus
    \[
    d_p(\mu_n,\mu) \geq \mu(\{x\})^{\frac{1}{p}} \varepsilon >0,
    \]
    in contradiction to the convergence assumption. 
       \end{proof}
\subsection{Karcher means and the proof of the persistent Hausmann Theorem}
In \cite{metric_reconstruction_via_optimal_transport}, a homotopy inverse to the inclusion $\iota \colon M \hookrightarrow \MRips[\delta](M)$ is constructed for the case of Riemannian manifolds through the technique of Karcher means, also referred to as centers of mass.
\begin{recollection}\cite{Yokota2016}
Recall that \define{the center of mass} (also referred to as Karcher mean or Fr\'echet mean) of a finite measure $\mu = \sum_{i = 0}^n \lambda_i x_i $ is defined as a global minimizer of the function $E_{\mu} \colon M \to \mathbb{R}_{\geq 0}$, $y \mapsto \int d(x,y)^2 d\mu(x) = \sum_{i=0}^n \lambda_i d(x_i,y)^2$.
\end{recollection}
\begin{notation}
	 Given a measure $\mu \in \MRips(M)$, suppose that the center of mass of $\mu$ exists and is unique. Then we denote it by $K(\mu)$.
\end{notation}
\begin{proposition}[{\cite[Thm B]{Yokota2016}}]\label{prop:inverse_hausmann}
   Let $0 \leq \delta \leq \catRad$. Then, given any $\mu \in \MRips(M)$, the center of mass $K(\mu)$ exists and is unique. Furthermore, given $x \in M$ and $0< r < \catRad$ such that $\supp(\mu) \subset \closedball[r](x)$, it holds that $K(\mu) \in \closedball[r](x).$
\end{proposition}
\begin{proof}
    We write $\mu = \sum_{i=0}^n \lambda_i x_i$, with $\lambda_i >0$, $\sum_{i=0}^n \lambda_i = 1$ and $x_i \in M$. 
    \cite{Yokota2016} states the result for $\catK$ spaces. Now, given any $x \in M$ and $0< r < \catRad$, the closed ball $\closedball[r](x)$ is a $\catK$ space (\cref{rec:facts_about_balls_in_bounded_curvature}). Hence, for any $x\in M$, we can apply \cite[Thm B]{Yokota2016} to obtain the existence and uniqueness of the center of mass on $\closedball[r](x)$, for measures supported in $\closedball[r](x)$. Now let $x \in \supp(\mu)$. By assumption $\supp(\mu) \subset \closedball[r](x)$, for any $r$ with $\diam(\supp(\mu)) < r < \delta$. Denote the resulting center of mass on $\closedball[r](x)$ of $\mu$ by $m$. An argument analogous to the one in the case of circumcenters in \cref{theo:ex_of_circumcenter} shows that this minimizer $m$ is also independent of the choice of $x = x_i$ and $r$.
	Observe that $E_{\mu}(m) <\delta^2$, since $E_{\mu}(x_i) < \delta^2$ for all $i \in \{0, \dots, n\}$.
	We now need to show that the center of mass computed on $\closedball[r](x)$ is, in fact, a global minimizer of $E_{\mu}$ on $M$, and uniquely so. So, let $y \in M$ with $E_{\mu}(y) \leq E_{\mu}(m)$. 
	Note that $E_{\mu}(y) < \delta^2$, since $E_{\mu}(m) = \sum_{i=0}^n \lambda_i d(x_i,m)^2 < \delta^2$. Hence, there exists an $i \in \{0, \dots, n\}$ such that $d(y,x_i) <\delta$. In particular, $y$ and $\supp(\mu)$ are contained in a closed ball of radius $\diam(\supp(\mu)) < \delta$ around $x_i$. Hence $y = m$, by the uniqueness on such balls.
\end{proof}
\begin{corollary}\label{cor:inverse_hausmann} 
	Let $\delta \leq \catRad$. Then the map $K \colon \MRips(M) \to M$, $\mu \mapsto K(\mu)$, is well-defined.
\end{corollary}
\begin{corollary}\label{lem:distances_to_karcher}
    Let $0 < \delta \leq \catRad$ and let $\mu \in \MRips(M)$. Then, for all $x \in \supp (\mu)$, it holds that $d(x, K(\mu)) <\delta$. Furthermore, there exists an $x \in \supp(\mu)$ such that $d(x,K(\mu)) < \conShrM[\delta] \delta$. 
\end{corollary}
\begin{proof}
	The first statement is immediate from \cref{prop:inverse_hausmann} by taking $r= \diam(\supp(\mu)) < \delta$. For the second statement, use \cref{theo:ex_of_circumcenter} and let $c$ be the circumcenter of $\supp(\mu)$. By assumption, we have $d(c,x) < \conShrM[\delta] \delta$, for all $x \in \supp(\mu)$. In particular, $E_{\mu}(c) < (\conShrM[\delta] \delta)^2$. As $K(\mu)$ is a minimizer of $E_{\mu}$, we also have $E_{\mu}(K(\mu)) < (\conShrM[\delta] \delta)^2$. Hence, there exists an $x \in \supp(\mu)$ such that $d(x,K(\mu)) < \conShrM[\delta] \delta$.
\end{proof}
We furthermore obtain the following relevant continuity claim.
\begin{corollary}\label{lem:cont_of_karcher}
Let $\delta \leq \catRad$. Then the map $K \colon \MRips(M) \to M$, $\mu \mapsto K(\mu)$, is continuous.
\end{corollary}
\begin{proof}
    By \cref{lem:indep_of_dist}, it follows that the topology on $\MRips(M)$ is induced by the $2$-Wasserstein distance. Observe that a center of mass of $\mu$ is precisely a minimizer of the $2$-Wasserstein distance to $M \subset \MRips(M)$ (using the canonical inclusion). It follows that, for any sequence $\mu_n$ in $\MRips(M)$ converging to $\mu \in \MRips(M)$, we also have 
    \[E_{\mu_n}(K(\mu_n)) =d_2( \mu_n, K(\mu_n))^2 =d_{2}(\mu_n,M)^2  \to d_2(\mu, M)^2 = d_2(\mu, K(\mu))^2 = E_{\mu}(K(\mu)) < \delta^2,\] for $n \to \infty$. Consequently, arguing as we did in the proof of \cref{prop:inverse_hausmann}, we obtain that $K(\mu_n)$ is contained in $\closedball[r](x)$, for some $x \in \supp(\mu)$, $r< \delta$ and $n$ sufficiently large. We now finish the proof using the assumption that $M$ is proper, i.e., that every bounded closed set is compact. It follows by compactness of $\closedball[r](x)$ that every subsequence of $K(\mu_n)$ has a convergent subsequence, denoted the same by abuse of notation. Let $x \in M$ be the limit of such a sequence.
    Again, expressing $E_{\mu}$ in terms of the $2$-Wasserstein distance, we obtain
    \[
    E_{\mu}(x) = \lim_{n \to \infty} E_{\mu_n} (K(\mu_n))  =E_{\mu} (K(\mu)).
    \]
    Hence, $x$ is also a minimizer of $E_{\mu}$ and thus agrees with $K(\mu)$, by \cref{prop:inverse_hausmann}. To summarize, every subsequence of $K(\mu_n)$ has, in turn, a subsequence converging to $K(\mu)$. Thus $K(\mu_n)$ converges to $K(\mu)$, as $n \to \infty$.
\end{proof}
We may now prove the following proposition, which is a generalization of \cite[Thm 4.2]{metric_reconstruction_via_optimal_transport} to the filtered setting. The proof is essentially analogous to the one in \cite{metric_reconstruction_via_optimal_transport}.
\begin{proposition} 
Let $0 < \delta \leq \catRad$. Then the morphism of persistent homotopy types \mbox{$\iota \colon \mbull \hookrightarrow \MRips(\mbull)$} defines part of a $\conShrM[\delta] \Lipf \delta$-interleaving in the homotopy category.
\end{proposition}
\begin{proof}
 We denote $L := \conShrM[\delta] \Lipf$ for brevity.
 In fact, we provide a so-called $(0,L\delta)$-interleaving, i.e., an interleaving of the form $\iota \colon \mbull \to \MRips(\mbull)$, $\MRips(\mbull) \to \mbull[{\bullet +L\delta}]$, from which a $L\delta$-interleaving is obtained by composing $\iota \colon \mbull \to \MRips(\mbull)$ with $s \colon \MRips(\mbull) \to \MRips(\mbull[{\bullet+L\delta}])$.
 Observe that the persistent spaces $\mbull$ and $\MRips(\mbull)$ are simply given by filtrations of the underlying spaces $M$ and $\MRips(M)$. Making use of this, we construct a homotopy inverse on the level of the underlying spaces and show that it (together with the relevant homotopies) is compatible with these filtrations. Consider the map $K \colon  \MRips(M) \to M$, $\mu \mapsto K(\mu)$ of \cref{prop:inverse_hausmann}. Observe that since $d(K(\mu),x) < \conShrM\delta$ (\cref{lem:distances_to_karcher}), for some $x$ in the support of $\mu$, it follows by the Lipschitz continuity of $f$ that $f(K(\mu)) < f(x) + \Lipf \conShrM \delta = f(x) + L\delta$. In particular, it follows that $K$ defines a morphism of persistent (filtered) spaces $\MRips(\mbull) \to \mbull[{\bullet + L\delta}]$, which we also denote by $K$. Evidently, $K \circ \iota = s \colon \mbull \to \mbull[{\bullet + L\delta}]$.
 To obtain that $\iota \circ K =s \colon \MRips (\mbull) \to \MRips (\mbull[{\bullet + L\delta}])$ in $\ho( \iSpaces^{\RN})$, we expose a persistent (filtered) homotopy $s \simeq \iota \circ K$, $H \colon \MRips (\mbull) \times [0,1] \to \MRips (\mbull[{\bullet + L\delta}]) $.
 $H$ is given by mapping $(\mu,t) \in \MRips (\mbull) \times [0,1]$ to the convex combination $tK(\mu) + (1-t)\mu$ in the convex space $\MRips[\delta]{\fmet}$. \cref{lem:distances_to_karcher} guarantees that $H(\mu,t) \in \MRips[\delta](M)$.
That $H(\mu,t) \in \MRips(\mbull[{u + L\delta}])$ whenever $\mu \in \MRips(\mbull[u])$ follows from the convexity of $\MRips(\mbull[{u + L\delta}]) \subset \MRips(M)$ and the fact that $K$ maps $\MRips(\mbull[u])$ into $\mbull[{u + L\delta}]$, for all $u \in \RN$, as already shown above. 
\end{proof}
Combining this result with \cref{prop:metric_vietoris_rips_equ_rips}, we obtain an interleaving in the homotopy category
$\mbull \simeq_{L\delta} \MRips[\delta](\mbull) \simeq \Rips (\mbull).$
This finishes the proof of \cref{thm:persistent_hausmann}. 
\section{
Stability of the bivariate function-Rips persistent homotopy type}\label{section:bivariate_stabiity}
To study the stability properties of the assignment $ \fmetapprx \mapsto \Rips[\delta](\mapprxbull)$, it is useful to consider the latter as a bivariate construction by also varying $\delta$.  \begin{notation}
 Given a persistent object of the form $F \colon \inPos \times  \mathbb{R}_{\geq 0} \to \cat[C]$, we will use the notation $F^{\bullet, \bullet}$ to indicate the fact that it is a functor in two variables. 
\end{notation}
\subsection{Correspondences of metric pairs}\label{subsec:correspondences}
We make use of the following parametrized version of correspondences, first introduced in \cite{CarlssonMemoli2010Multiparameter} in the $1$-parameter case.
	\begin{definition}[See \cite{CarlssonMemoli2010Multiparameter} for the $1$-parameter case]\label{def:filtered_correspondence}
	Let $\fmetapprx[0]$ and $\fmetapprx[1]$ be metric pairs over $\RN$ and $\epsmet,\epsfun \geq 0$. A \textit{$(\epsmet,\epsfun)$-correspondence} between $\fmetapprx[0]$ and $\fmetapprx[1]$ is a subset $\mathfrak{C} \subset \mathbb{M}_0 \times \mathbb{M}_1$ such that $ \pi_{\mathbb{M}_0}(\mathfrak{C}) = \mathbb{M}_0$ and $\pi_{\mathbb{M}_1}(\mathfrak{C}) = \mathbb{M}_1$, and such that for all $(x,y), (x',y') \in \mathfrak{C}$ it holds that 
	\begin{align*}
		|d_{\mathbb{M}_0}(x,x') - d_{\mathbb{M}_1}(y,y')| \leq \epsmet,
	\end{align*}
	and furthermore such that, for all $(x,y) \in \mathfrak{C}$, it holds that
	\begin{align*}
		 |\mathbb{f}_0(x) - \mathbb{f}_1(y)|_{\infty} \leq \epsfun.
	\end{align*}
	We will denote $(\epsmet,\epsfun)$-correspondences in the form $\mathfrak{C} \colon \fmetapprx[0] \approx_{\epsmet,\epsfun} \fmetapprx[1]$.
\end{definition}
\begin{remark}
    Observe that a $(0,0)$-correspondence of metric pairs $\fmetapprx[0]$ and $\fmetapprx[1]$ specifies the same data as an isometry $\metapprx_0 \cong \metapprx_1$ that commutes with $\fapprx_0$ and $\fapprx_1$. Hence, for all intents and purposes, metric pairs that admit a $(0,0)$-correspondence are the same.
\end{remark}
For the purpose of topological data analysis, stability results are typically stated in terms of a distance induced by certain comparison structures such as correspondences and interleavings. We note that expressing results in this manner is often a lossy process, as it omits \textit{how things are nearby} and only retains \textit{that they are nearby}. Furthermore, in the case of correspondences of function pairs, this requires a choice of how to combine the two parameters $\epsmet$ and $\epsfun$ into a single one, which loses additional information. In this article, we will usually stick to keeping the two parameters separate, as this provides finer information on how the metric and functional aspects influence the interleaving. Nevertheless, the following definition will be useful.
 \begin{construction}\label{con:corresp_distance}
     We denote by $\MetRN$ the class of all metric pairs $\fmetapprx$ over $\RN$. Fix any norm $\|\cdot \|_{\alpha}$ on $\mathbb{R}^2$.
	 Then the assignment
    \begin{align*}
              \dCor \colon \MetRN \times \MetRN &\to [0, \infty] \\
               (\fmetapprx[0],\fmetapprx[1]) &\mapsto \inf\{ \|(\epsmet,\epsfun)\|_{\alpha} \mid \epsmet, \epsfun \geq 0, \fmetapprx[0] \approx_{\epsmet,\epsfun} \fmetapprx[1] \}
    \end{align*}
      fulfills the triangle inequality and takes the value $0$ on pairs that are isometric through an isometry that commutes with the functions. We call the thus defined distance the \define{correspondence distance on metric pairs over $\RN$} (depending on $\|\cdot \|_{\alpha}$).
\end{construction}
\begin{remark}\label{rem:scale_normalization}
	Evidently, all possible choices of norm lead to equivalent distances.
	While the prominent choice in the literature is $\|\cdot \|_{\alpha} = \|\cdot \|_{\infty}$ (see \cite{CarlssonMemoli2010Multiparameter}),
	we also think that other choices of norm can be insightful. For instance, when one is studying a metric pair $\fmet$ with $f$ having a large Lipschitz constant $K>0$, it can be more sensible to use the norm $\|(x,y)\| = \| (x, \frac{1}{K}y) \|_{\infty}$ in order for the metric and functional aspects to renormalize to similar scales.
\end{remark}
\begin{notation}\label{not:reparam_dist}
	Given $K>0$, we write $\dCor[{\alpha_K}]$ for the correspondence distance on metric pairs over $\RN$ defined using the norm $\|(x,y)\|_{\alpha_K} := \| (x, \frac{1}{K}y) \|_{\infty}$.
\end{notation}
\begin{remark}
 When $N=0$, we can identify objects of $\MetRN$ with their underlying metric space. Then correspondences are just correspondences of metric spaces in the classical sense and $d_{\textnormal{Cor},\infty}$ is twice the Gromov-Hausdorff distance (see \cite[Ch. 7]{BBI2001}).
\end{remark}
\subsection{Statement of the bivariate stability result}
The additional flexibility of varying $\delta$ in the bivariate function-Rips complex will allow us to prove the following stability result. We remark that it is not conceptually new, and homological variants are found in several places (\cite{chazal2011scalar,blumberg2024stability}).
\begin{proposition}\label{prop:bivariate_interleaving}
	Let $\epsmet,\epsfun \geq 0$ and let $\fmetapprx[0]$ and $\fmetapprx[1]$ be metric pairs over $\RN$.
	Any $(\epsmet,\epsfun)$-correspondence $\mathfrak{C} \colon \fmetapprx[0] \approx_{\epsmet,\epsfun} \fmetapprx[1]$ induces a canonical $(\epsmet,\epsfun)$-interleaving 
	\[ \Rips[\bullet] (\mapprxbull_0) \simeq_{\epsmet,\epsfun} \Rips[\bullet] (\mapprxbull_1) \] in $\ho ( \iSpaces^{\prodPosR})$. 
\end{proposition}
In terms of the interleaving distance in the homotopy category and the $\dCor[\infty]$ distance, one obtains the following stability result as an immediate corollary.
\begin{corollary}
The assignment 
\begin{align*}
	\MetRN &\to \iSpaces^{\prodPosR} \\
	\fmetapprx &\mapsto \real{\Rips[\bullet](\mapprxbull)}
\end{align*} is $1$-Lipschitz with respect to $\dCor[\infty]$ and the interleaving distance in the homotopy category.
\end{corollary}
\subsection{The universal property of the bivariate Rips complex}
To prove \cref{prop:bivariate_interleaving}, we use \cref{rec:equv_model_simp,rec:equ_with_ss} and construct the interleaving using the simplicial set model for the bivariate Rips complex.
We will leverage the universal property of the Rips simplicial set here, which we state for the case of filtered simplicial sets $X^{\bullet,\bullet} \in \sSet^{\prodPosR}$, i.e., where all structure morphisms are inclusions (see \cref{obs:universal_property_rips_filtered_sset} below). It allows one to construct morphisms $X^{\bullet,\bullet} \to \Rips[\bullet](\mapprxbull)$ by defining a map with target $\mathbb{M}$ on the vertices of $X^{\bullet,\bullet}$ (see \cite{ChazalDeSilvaOudot2014} which uses the simplicial complex analog). 
We write $X^{\infty} = \bigcup_{(\delta,u) \in \prodPosR  } X^{\delta,u}$.
\begin{construction}\label{con:requirements_for_universal_prop}
	Let $X \colon \prodPosR \to \sSet$ be a filtered simplicial set and $\fmetapprx$ be a metric pair over $\RN$.
We denote by $S(X, \fmetapprx) \subset \cat[Set](X^{\infty}_0, \mathbb{M})$ the set of such maps $\varphi \colon X^{\infty}_0 \to \mathbb{M}$ that fulfill the condition that, for all $(\delta,u) \in \prodPosR$, it holds that: 
	\begin{enumerate}
		\item for all $x \in X^{\delta,u}_0$, we have $\fapprx(\varphi(x)) \leq u$;
		\item for all $1$-simplices $\sigma \in X^{\delta,u}_1$, with vertices $x_0$ and $x_1$, we have $d(\varphi(x_0), \varphi(x_1)) < \delta$.
	\end{enumerate}
    Associating to a morphism $X^{\bullet,\bullet} \to \Rips[\bullet](\mapprxbull)$ the induced simplicial map  $X^{\infty} \to  \Rips[\infty]\fmetapprx =\Rips[\infty](\mathbb{M})$ and evaluating on vertices to obtain a map $X^{\infty}_0 \to  \Rips[\infty](\mathbb{M})_0 = \mathbb{M}$, defines a map
    \[
    \eta \colon \sSet^{\prodPosR}(X^{\bullet,\bullet}, \Rips[\bullet](\mapprxbull)) \to S(X, \fmetapprx).
    \]
	\end{construction}
\begin{proposition}[Appendix~\ref{app:simplicial_cplx_vs_set}.]\label{obs:universal_property_rips_filtered_sset}
	The map $\eta$ of \cref{con:requirements_for_universal_prop} defines a natural bijection.
\end{proposition}
\begin{remark}
    Conceptually speaking, the injectivity part of \cref{obs:universal_property_rips_filtered_sset} means that a morphism $X^{\bullet,\bullet} \to \Rips[\bullet](\mapprxbull)$ is uniquely determined by its values on vertices (the values can be identified with elements of $\mathbb{M}$). The surjectivity and well-definedness state that a  map $X^{\infty}_0 \to \mathbb{M}$ extends if and only if the criteria of \cref{con:requirements_for_universal_prop} are verified. 
\end{remark}
\begin{remark}\label{rem:uni_rips_mod_version}
	We will often need to use a modified version of \cref{obs:universal_property_rips_filtered_sset}, where instead of $\Rips[\bullet](\mapprxbull)$, we consider a reparametrized version in which $\Rips[\bullet](\mapprxbull) \colon \prodPosR \to \sSet$ is replaced by a composition $\prodPosR \xrightarrow{(\delta',u')} \prodPosR \xrightarrow{\Rips[\bullet](\mapprxbull)} \sSet$ for poset maps $\delta',u'$. Then, the analogous statement holds, replacing in \cref{con:requirements_for_universal_prop} $u$ with $u'(\delta,u)$ and $\delta$ with $\delta'(\delta,u)$ in the defining inequalities for $S(X^{\bullet,\bullet}, \fmetapprx)$. Furthermore, we will often encounter the modified case where we work with filtered simplicial sets $X^{\bullet,\bullet}$ over $ \mathbb{I} \times \RN$, with $\mathbb{I} = [0,\rho]$ or $\mathbb{I}=\{\rho\}$, and we restrict $\Rips[\bullet](\mapprxbull)$ to $\mathbb{I} \times \RN$. In this case, again the essentially same statement holds, with the only change being the replacement of $X^{\infty}$ by $X^{\rho,\infty} = \bigcup_{u \in \RN} X^{\rho,u}$ and the defining condition for $S(X^{\bullet,\bullet}, \fmetapprx)$ needing to be verified only for pairs $(\delta, u) \in \mathbb{I} \times \RN$. 
 \end{remark}
 \subsection{Proof of \texorpdfstring{\cref{prop:bivariate_interleaving}}{the bivariate interleaving result}}
Next, let us construct the simplicial maps that define the interleaving of \cref{prop:bivariate_interleaving}: 
\begin{construction}\label{cons:filtered_interleaving}
	Given such a $(\epsmet,\epsfun)$-correspondence $\mathfrak{C}$, we define maps 
	\begin{align*}
	\varphi \colon \Rips[\bullet](\mapprxbull_0) \to \Rips[\bullet+\epsmet](\mapprxbull[\bullet + \epsfun]_1) \quad{\textnormal{ and }}\quad \psi \colon \Rips[\bullet](\mapprxbull_1) \to \Rips[\bullet+\epsmet](\mapprxbull[\bullet + \epsfun]_0)
	\end{align*}
	by applying \cref{obs:universal_property_rips_filtered_sset} as follows. Observe that for $X^{\bullet,\bullet}=\Rips[\bullet](\mapprxbull_0)$, we have $X^{\infty}_0= \mathbb{M}_0$. Then $\varphi$ is defined under \cref{obs:universal_property_rips_filtered_sset} by choosing for each $x \in \mathbb{M}_0$ an element $y \in \mathbb{M}_1$ with $(x,y) \in \mathfrak{C}$ and defining $\varphi(x) := y$. The map $\psi$ is defined analogously. A priori, these maps depend on the choices of elements in the fibers of the correspondence. However, up to homotopy these choices are irrelevant.
\end{construction}
\begin{proof}[Proof of \cref{prop:bivariate_interleaving}]
	The conditions on $\mathfrak{C}$ ensure exactly that the (shifted) requirements of \cref{obs:universal_property_rips_filtered_sset,rem:uni_rips_mod_version} are fulfilled.
	We write $\varepsilon:= (\epsmet,\epsfun)$ for brevity.
	It remains to see that $\varphi$ and $\psi$ define an $(\epsmet,\epsfun)$-interleaving in the homotopy category. We show that $\psi^{+ (\epsmet,\epsfun)} \circ \varphi \simeq s$; the other equality being shown analogously. To this end, we need to construct a homotopy
	$
	H \colon \Rips[\bullet](\mapprxbull_0) \times \Delta^1 \to \Rips[\bullet + 2\epsmet](\mapprxbull[\bullet + 2\epsfun]_0)$
	between $\psi^{+\varepsilon} \circ \varphi$ and $s$. We again use \cref{obs:universal_property_rips_filtered_sset}. In this case, it implies that the persistent simplicial maps $\psi^{+\varepsilon} \circ \varphi$ and $s$ are elementarily homotopic, if for every pair $x,x' \in \mathbb{M}_0$ it holds that 
	$ d(\psi^{+\varepsilon}(\varphi(x)), s(x')) \leq d(x,x') + 2\epsmet.$ Indeed, the inequality
	\[
	d(\psi^{+\varepsilon}(\varphi(x)), s(x')) = d(\psi^{+\varepsilon}(\varphi(x)), x') \leq d(\varphi(x), \varphi(x')) + \epsmet \leq d(x,x') + 2\epsmet,
	\]
     holds by definition of $\varphi$ and $\psi$ in terms of $\mathfrak{C}$. Finally, let us show independence of the choices made in the construction of $\varphi$ (or $\psi$) up to persistent homotopy. Observe that for any two choices of $\varphi$, say $\varphi_0$ and $\varphi_1$, we can again apply \cref{obs:universal_property_rips_filtered_sset} to construct a homotopy, and it suffices to show that 
     \[
     d(\varphi_0(x), \varphi_1(x')) \leq d(x,x') + \epsmet,
     \] 
     for all $x,x' \in \metapprx_0$.
     By construction we have $(x, \varphi_0(x)), (x', \varphi_1(x')) \in \mathfrak{C}$. Hence, the latter inequality holds by the definition of $(\epsmet,\epsfun)$-correspondences.
\end{proof}
\section{The shrinking trick}\label{section:shrinking_trick}
Throughout this section, we fix some homotopy theory $\mathcal{C}$. By an interleaving, we will always mean interleaving in the homotopy category. 
\begin{notation}
    When fixing either parameter of a persistent object $F^{\bullet,\bullet} \colon \prodPos \to \mathcal{C}$, we will use the notation $F^{\delta,\bullet}$ or $F^{\bullet,u}$, to denote the resulting persistent objects on $\inPos$, or $\mathbb{R}_{\geq 0}$ respectively. At times, we will also use the notation $\bullet_1$ and $\bullet_2$ to indicate the first and second variable, respectively. For example, $F^{\bullet_1, \bullet_2\flow {\bullet_1}}$ indicates the persistent object given by precomposing $F$ with the endofunctor $(\delta,u) \mapsto (\delta,u + \delta)$ of $\prodPos$. 
\end{notation}
\begin{remark}\label{block:fixing_delta_not_stable}
    Suppose now we are given $\delta \geq 0$, as well as two persistent objects $F^{\bullet,\bullet},G^{\bullet,\bullet} \colon \prodPos \to \mathcal{C}$, together with an $(\varepsilon_1, \varepsilon_2)$-interleaving $\varphi \colon  F^{\bullet,\bullet} \simeq_{\varepsilon_1, \varepsilon_2} G^{\bullet,\bullet}\colon \psi$, for some $\varepsilon_1,\varepsilon_2 \geq 0$. We want to deduce an interleaving
$
F^{\delta,\bullet} \simeq_{\varepsilon'} G^{\delta,\bullet},$
for some $\varepsilon'$ depending on $\varepsilon_1, \varepsilon_2$ and $\delta$. The issue at hand is, of course, that the interleaving morphisms $\varphi$ and $\psi$ only procure morphisms $
F^{\delta,\bullet} \to G^{\delta + \varepsilon_1,\bullet \flow {\varepsilon_2}} \quad \text{ and } \quad G^{\delta,\bullet} \to F^{\delta + \varepsilon_1,\bullet \flow {\varepsilon_2}}.$
\end{remark}
\subsection{Shrinking transformations}
To amend the difficulty described in \cref{block:fixing_delta_not_stable}, we introduce an additional structure on $G$, which we call a shrinking transformation. This transformation allows us to decrease the $\delta$-part of the persistence parameter of $G$, at the cost of increasing the $u$-part. 
\begin{notation}\label{notation:C_S}
For the remainder of this section, fix real constants $0 \leq \conShr \leq 1$, $\conCos \geq 0$ and $\rho > 0$.
\end{notation}
\begin{definition}\label{def:back-prop}
    A \textit{shrinking transformation} for $G \colon \prodPos \to \mathcal{C}$ 
	is a morphism 
    \[ \tau \colon G^{\bullet_1, \bullet_2}|_{\prodPosrho} \to G^{\conShr\bullet_1, \bullet_2 \flow \conCos \bullet_1}|_{\prodPosrho} \]
     in $\ho(\mathcal{C}^{\prodPosrho})$ such that the following diagram in $\ho(\mathcal{C}^{\prodPosrho})$ commutes:  
\begin{diagram}\label{diag:back-prop}
	{}& {G^{\bullet_1,\bullet_2+\conCos \bullet_1}|_{\prodPosrho}} \\
	{G^{\conShr\bullet_1,\bullet_2 + \conCos \bullet_1}|_{\prodPosrho}} \\
	& {G^{\bullet_1,\bullet_2}|_{\prodPosrho}}
	\arrow["s", from=2-1, to=1-2]
	\arrow["s"', from=3-2, to=1-2]
	\arrow["\tau", crossing over, from=3-2, to=2-1]
\end{diagram}
\end{definition}
\begin{definition}
Now let $\delta \in [0, \rho]$, $0 \leq \varepsilon_1$ be such that $\varepsilon_1 \leq \min \{\rho- \delta,\frac{1- \conShr}{1+\conShr} \delta\}$, and $\varepsilon_2 \geq 0$.
As we assumed that $\delta + \varepsilon_1 \leq \rho$, the morphism $\tau$ restricts to a well-defined morphism $\tau \colon G^{\delta + \varepsilon_1,\bullet \flow {\varepsilon_2}} \to G^{\conShr(\delta + \varepsilon_1),\bullet \flow {\varepsilon_2} + \conCos(\delta + {\varepsilon_1})}$, denoted by the same name by abuse of notation. We say that an $(\varepsilon_1,\varepsilon_2)$-interleaving $\varphi \colon F^{\bullet,\bullet} \simeq_{\varepsilon_1,\varepsilon_2} G^{\bullet,\bullet} \colon \psi$ is \textit{$\delta$-compatible} with $\tau$ if the following diagram in $\ho(\mathcal{C}^{\inPos})$ commutes:
\begin{diagram}[column sep = tiny]\label{diag:compat}
	{}& {F^{\conShr{(\delta + \varepsilon_1)} +\varepsilon_1, \bullet\flow 2\varepsilon_2+\conCos(\delta+ \varepsilon_1)}} & {F^{\delta,\bullet\flow 2 \varepsilon_2 + \conCos(\delta + \varepsilon_1) }} \\
	{G^{\conShr(\delta + \varepsilon_1),\bullet + \varepsilon_2 + \conCos(\delta + \varepsilon_1)}} \\
	&&& {G^{\delta + \varepsilon_1,\bullet+ \varepsilon_2}} \\
	&& {F^{\delta,\bullet}}
	\arrow["s", from=1-2, to=1-3]
	\arrow["\psi"{description}, from=2-1, to=1-2]
	\arrow["\tau"{description}, from=3-4, to=2-1]
	\arrow["s"{description, pos=0.6}, from=4-3, to=1-3, crossing over]
	\arrow["\varphi"', from=4-3, to=3-4]
\end{diagram}
\end{definition}
\subsection{Verticalized interleavings from shrinking transformations}
The following result then allows for the transformation of bivariate into univariate interleavings, through the use of a shrinking transformation.
    \begin{theorem}\label{lem:munchkin_lemma}
    Let $G \colon \prodPos \to \mathcal{C}$ be a persistent object in $\mathcal{C}$, equipped with a shrinking transformation $\tau$ on $G$. Let $\varepsilon_1, \delta \geq 0$ be such that $\varepsilon_1 \leq \min \{\rho- \delta,\frac{1- \conShr}{1+\conShr} \delta\}$ and let $\varepsilon_2 \geq 0$.  \\
	Let $F \colon \prodPos \to \mathcal{C}$ be another persistent object, together with an $(\varepsilon_1,\varepsilon_2)$-interleaving 
	$\varphi \colon F \simeq_{\varepsilon_1,\varepsilon_2} G \colon \psi$ that is $\delta$-compatible with $\tau$. 
	Then there is an interleaving
\[
\varphi' \colon F^{\delta, \bullet} \simeq_{\varepsilon_2 + \conCos \varepsilon_1 + \conCos\delta} G^{\delta,\bullet} \colon \psi',\]
explicitly given by the compositions
\begin{align*}
	&\varphi'\colon 
	F^{\delta,\bullet} 
	\xrightarrow{\varphi} 
	G^{\delta + \varepsilon_1,\,\bullet \flow {\varepsilon_2}} 
	\xrightarrow{\tau} 
	G^{\conShr(\delta + \varepsilon_1),\,\bullet \flow \varepsilon_2 + \conCos \varepsilon_1 + \conCos\delta} 
	\xrightarrow{s} 
	G^{\delta,\,\bullet \flow \varepsilon_2 + \conCos \varepsilon_1 + \conCos\delta}, \\
	&\psi' \colon 
	G^{\delta,\bullet} 
	\xrightarrow{\tau} 
	G^{\conShr\delta,\,\bullet \flow \conCos\delta} 
	\xrightarrow{\psi} 
	F^{\conShr\delta + \varepsilon_1,\,\bullet \flow \conCos\delta + \varepsilon_2} 
	\xrightarrow{s} 
	F^{\delta,\,\bullet \flow \varepsilon_2 + \conCos \varepsilon_1 + \conCos\delta}.
\end{align*}
\end{theorem}
\begin{proof}
	The fact that all morphisms $s$ and $\tau$ in the definition of $\varphi'$ and $\psi'$ are well-defined follows from the inequalities $
		\delta + \varepsilon_1 \leq \rho$ and $
		\conShr(\delta+ \varepsilon_1) \leq \conShr\delta + \varepsilon_1 \leq \conShr(\delta + \varepsilon_1) +\varepsilon_1 \leq \delta$, 
	which hold by the assumption on $\varepsilon_1$ and $\delta$.
	We now need to verify the defining commutativity conditions of an interleaving. To simplify notation, we define $\varepsilon' := \varepsilon_2 + \conCos \varepsilon_1 + \conCos\delta$ and $\bullet_{\gamma} := \bullet_2 + \varepsilon_2 + \conCos(\delta + \varepsilon_1) + \conCos\conShr(\delta + \varepsilon_1).$
	To simplify notation even further, we will omit $F$ and $G$ from the notation and only spell out the superscript. Whether $F$ or $G$ is meant will be uniquely determined by the specified morphisms. We also omit any shifting notation from the interleaving morphisms.  Now, consider the first composition $\psi' \circ \varphi'$. To verify that $\psi' \circ \varphi' = s$ it suffices to verify the commutativity of the following diagram.
{
\begin{adjustbox}{scale = 0.75,center}
\begin{tikzcd}
	&& {\delta, \bullet \flow 2\varepsilon'} \\
	& {\conShr\delta + \varepsilon_1, \bullet \flow \varepsilon' + \conCos\delta + \varepsilon_2} & {\conShr(\delta + \varepsilon_1) + \varepsilon_1, \bullet_{\gamma} \flow \varepsilon_2} \\
	& {\conShr\delta, \bullet \flow \varepsilon' + \conCos\delta} & {\conShr(\delta + \varepsilon_1), \bullet_{\gamma}} \\
	{\delta, \bullet \flow \varepsilon'} & {\conShr^2(\delta + \varepsilon_1), \bullet_{\gamma}} && {\delta, \bullet \flow 2\varepsilon_2 + \conCos(\delta + \varepsilon_1)} \\
	& {\conShr(\delta + \varepsilon_1), \bullet + \varepsilon_2 + \conCos(\delta + \varepsilon_1)} & {\conShr(\delta + \varepsilon_1) + \varepsilon_1, \bullet + 2\varepsilon_2 + \conCos(\delta + \varepsilon_1)} \\
	& {\delta + \varepsilon_1, \bullet \flow \varepsilon_2} \\
	&& {\delta, \bullet}
	\arrow["s"{description}, from=2-2, to=1-3]
	\arrow["s"{description}, from=2-3, to=1-3]
	\arrow["\psi"{description}, from=3-2, to=2-2]
	\arrow["s"{description}, from=4-2, to=3-2]
	\arrow["\psi"{description}, from=3-3, to=2-3]
	\arrow["{\psi'}", curve={height=-90pt}, from=4-1, to=1-3]
	\arrow["\tau"{description}, from=4-1, to=3-2]
	\arrow["s"{description}, from=4-2, to=3-3]
    \arrow["s"', curve={height=12pt}, from=4-4, to=1-3]
	\arrow[""{name=0, anchor=center, inner sep=0}, "s"{description}, curve={height=6pt}, from=5-2, to=3-3]
	\arrow["s"{description}, from=5-2, to=4-1]
	\arrow[""{name=1, anchor=center, inner sep=0}, "\tau"{description}, from=5-2, to=4-2]
	\arrow[""{name=2, anchor=center, inner sep=0}, "\psi"{description}, from=5-2, to=5-3]
	\arrow["s"{description}, from=5-3, to=4-4]
	\arrow["\tau"{description}, from=6-2, to=5-2]
	\arrow["{\varphi'}", curve={height=-70pt}, from=7-3, to=4-1]
	 \arrow["s"{description}, curve={height=12pt}, from=7-3, to=4-4]
	\arrow[""{name=3, anchor=center, inner sep=0}, "\varphi"{description}, from=7-3, to=6-2]
	\arrow["{({\textnormal{\ref{diag:back-prop}}})}"{description, pos=0.7}, shift left=2, draw=none, from=1, to=0]
	\arrow["{({\textnormal{\ref{diag:compat}}})}"{description}, draw=none, from=3, to=2]
\end{tikzcd}
\end{adjustbox}
}
Observe that for every $s$ arrow to be well-defined, we require the inequality $\conShr( \delta + \varepsilon_1)  + \varepsilon_1\leq \delta$, which we have already seen above. Observe that the canonical shift morphisms $s$ commute with essentially every other morphism in sight (in the appropriate contextual sense).
The commutativities of the cells in the diagram follow from the universal commutativity of $s$ or from \cref{diag:back-prop,diag:compat}. This proves the first interleaving equality. \\
The second interleaving equality follows by chasing the following diagram.\\
{
\begin{adjustbox}{scale = 0.7, center}
\begin{tikzcd}
	&& {\delta, \bullet + 2\varepsilon'} \\
	&& {\conShr(\delta + \varepsilon_1), \bullet + 2\varepsilon' } \\
	& {\delta + \varepsilon_1, \bullet + \varepsilon' + \varepsilon_2} \\
	{\delta, \bullet + \varepsilon'} && {\conShr(\conShr\delta + 2\varepsilon_1), \bullet + \conCos\delta + 2\varepsilon_2 + \conCos(\conShr\delta + 2\varepsilon_1)} & {\conShr^2\delta, \bullet + \conCos\delta + \conCos\conShr\delta} & {\conShr\delta, \bullet + \conCos\delta + \conCos\conShr\delta} \\
	& {\conShr\delta+ \varepsilon_1, \bullet + \conCos\delta + \varepsilon_2} & {\conShr\delta + 2\varepsilon_1, \bullet + \conCos\delta + 2\varepsilon_2} && 
    \\
	&& 
    {\conShr\delta, \bullet + \conCos\delta} 
    \\
	&& {\delta, \bullet}
	\arrow["s"{description}, from=2-3, to=1-3]
	\arrow["\tau"{description}, crossing over, from=3-2, to=2-3]
	\arrow["{\varphi'}", curve={height=-30pt}, from=4-1, to=1-3]
	\arrow["\varphi"{description}, from=4-1, to=3-2]
	\arrow["s"{description}, from=4-3, to=2-3]
	\arrow["s"{description}, from=4-4, to=4-3]
	\arrow["s"{description}, from=4-4, to=4-5]
	\arrow["s"{description}, curve={height=12pt}, from=4-5, to=1-3]
	\arrow["s"{description}, from=5-2, to=4-1]
	\arrow["\varphi"{description}, from=5-2, to=5-3]
	\arrow["s"{description}, curve={height=-18pt}, from=5-3, to=3-2]
	\arrow["\tau"{description}, from=5-3, to=4-3]
	\arrow["\tau"{description}, from=6-3, to=4-4]
	\arrow["s"{description}, from=6-3, to=4-5]
	\arrow["\psi"{description}, from=6-3, to=5-2]
	\arrow["s"{description}, from=6-3, to=5-3]
	\arrow["s"{description}, from=6-3, to=4-5]
	\arrow["{\psi'}", curve={height=-40pt}, from=7-3, to=4-1]
	\arrow["s \circ \tau"{description}, curve={height=12pt}, from=7-3, to=4-5]
	\arrow["\tau"{description}, from=7-3, to=6-3]
\end{tikzcd}
\end{adjustbox}
}
Note that none of the cells in this diagram require the $\delta$-compatibility condition.
Instead, one only uses the naturality properties of $s$, one of the two interleaving equalities, and the defining property of the shrinking transformation.
\end{proof}
\section{From pseudo-barycenters to shrinking transformations}\label{section:shrinking_for_rips}
We now want to apply \cref{lem:munchkin_lemma} together with \cref{prop:bivariate_interleaving} to obtain inference results on function-Rips persistent homotopy types. To this end, we need to establish the existence of a shrinking transformation for the bivariate function-Rips complex.
\subsection{Illustration of the technique}
Let us first explain the conceptual idea behind the construction of shrinking transformations for the bivariate function-Rips complex.
	For the sake of simplicity, let us first discuss the case $\RN = \mathbb{R}^0$, i.e., the case when $f$ defines a trivial filtration, $\prodPosRhoRN = [0, \rho]$, and the cost parameter $C$ is irrelevant. Again, let $M$ be a complete $\cabK$ space.
    We are looking to construct a morphism 
    \[
    \Rips[\bullet](\mbull[])|_{\rhoPos} \to \Rips[\conShr\bullet](\mbull[])|_{\rhoPos}
    \]
    in $\ho (\sSetK^{\rhoPos}) \simeq \ho (\iSpaces^{\rhoPos})$. 
	It is a general paradigm in simplicial approaches to homotopy theory - which appeared in its first incarnation in the classical simplicial approximation theorem - that a homotopy class between two simplicial objects can be presented by a simplicial map by subdividing the source object sufficiently often. 
	\begin{recollection}\label{ex:barycentric_subdiv}
		Recall that the category of simplicial sets admits an endofunctor $\sd_b \colon \sSet \to \sSet$ called the barycentric subdivision functor (see \cite{Kan1957}). For the purpose of this article, it will suffice to recall the following key properties of $\sd_b$:
	\begin{itemize}
		\item The functor $\sd_b$ preserves colimits and inclusions;
		\item Given a simplicial set $X$, the vertices of $\sd_b X$ correspond one-to-one to the non-degenerate simplices of $X$. In particular, the vertices of $\sd_b \Rips[\delta](M)$ correspond to sequences $(x_0, \dots, x_n)$ of elements in $M$ (without consecutive repetitions) of pairwise distance smaller than $\delta$;
		\item Given two vertices $x,y \in \sd_b X$, there is a $1$-simplex from $x$ to $y$ if and only if the corresponding non-degenerate simplex of $X$ corresponding to $x$ is a face of the non-degenerate simplex of $X$ corresponding to $y$. In particular, there is a $1$-simplex from $(x_0, \dots, x_n)$ to $(y_0, \dots, y_m)$ in $\sd_b \Rips[\delta](M)$ if and only if the sequence $(x_0, \dots, x_n)$ is a subsequence of the sequence $(y_0, \dots, y_m)$;
		\item There is a natural weak equivalence $\lambda_b \colon \sd_b \xRightarrow{\simeq} 1_{\sSet}$ called the last vertex map (see \cref{fig:last-vertex_map}). In the special case of $\sd_b \Rips[\delta](M)$, the last vertex map sends a vertex corresponding to a sequence $(x_0, \dots, x_n)$ to the last element $x_n$ of the sequence.
	\end{itemize}
	\end{recollection}
	\begin{figure}[tb]
	\centering
	\includegraphics[width=0.5\textwidth]{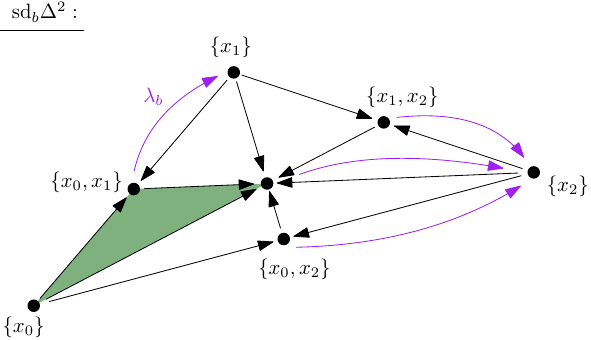}
	\caption{Illustration of the last vertex map $\lambda_b$ on $\sd_b \Delta^2$. The purple arrows indicate where the vertices are mapped by~$\lambda_b$. All but the green simplex
	are collapsed to lower-dimensional simplices.}
	\label{fig:last-vertex_map}
	\end{figure}
	\begin{block}\label{block:description_of_simple_case}
	Using the natural weak equivalence $\lambda \colon \sd_b \xRightarrow{\simeq} 1_{\sSet}$ we obtain an isomorphism
    \(
    \Rips[\bullet](\mbull[])|_{\rhoPos} \simeq \sd_b \Rips[\bullet](\mbull[])|_{\rhoPos},
    \)
	after passing to the persistent homotopy category.
    Hence, to construct a shrinking transformation, it suffices to expose a persistent simplicial map 
	\[
	\sd_b \Rips[\bullet](\mbull[])|_{\rhoPos} \to \Rips[\conShr\bullet](\mbull[])|_{\rhoPos}.
	\]
	We can now apply \cref{obs:universal_property_rips_filtered_sset,rem:uni_rips_mod_version} to see that such a morphism is uniquely determined by a map of sets 
	\[
	\Theta \colon \sd_b \Rips[\rho](M)_0 \to M
	\]
	fulfilling the property that whenever $x,y \in \sd_b \Rips[\delta](M)_0$, with $\delta \leq \rho$ are connected by a $1$-simplex from $x$ to $y$, then we need that 
	\(
	d(\Theta(x), \Theta(y)) < \conShr \delta.
	\)
	In other words, we map the vertices of $\sd_b \Rips[\delta](M)$ to points in $M$ in a manner such that distances between edge-connected vertices are shrunk by a factor of $\conShr$, compared to the distances of edge-connected vertices in $\Rips[\delta](M)$.
    \end{block}
    \begin{block}
	Now, to construct such a map $\Theta$, it helps to first observe that the vertices of $\sd_b \Rips[\delta](M)$ correspond to sequences $(x_0, \dots, x_n)$ of elements in $M$ of pairwise distance smaller than $\delta$. An edge relation corresponds to a subsequence relation $(x_0, \dots, x_n) \subset (y_0, \dots, y_m)$. Hence, to construct $\Theta$, it suffices to assign to each such sequence a point in $M$ in a manner that whenever one sequence is a subsequence of another then the assigned points are at distance smaller than $\conShr \delta$. \\
	Under appropriate geometric assumptions, such a map $\Theta$ can be constructed by assigning to each sequence $(x_0, \dots, x_n)$ the center of a minimal enclosing ball of the sequence - the circumcenter (see \cref{def:chebcenter,fig:pseudo-barycenter_map}).
	Modulo the difference in simplicial complex and simplicial set language, this circumcenter construction was in fact the decisive argument in the proof of Latschev's theorem in \cite{Majhi2024_DCG_DemystifyingLatschev}.
\end{block}
\begin{figure}[tb]
\centering
\includegraphics[width=0.45\textwidth]{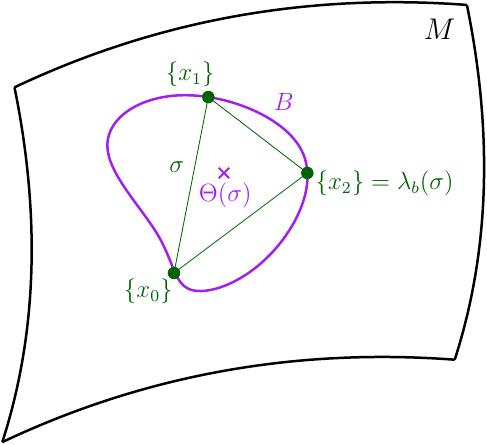}
\caption{Illustration of the pseudo-barycenter map $\Theta$.}
\label{fig:pseudo-barycenter_map}
\end{figure}
\subsection{Shrinking transformations from pseudo-barycenter maps}
	For our purposes, it will turn out to be useful to extract the abstract essence of the argument constructed in the previous subsection. This general approach leads to the notion of a \define{pseudo-barycenter map}, which is the main technical tool for constructing shrinking transformations for the bivariate function-Rips complex. As most of our arguments are not specific to the barycentric subdivision functor, we make the following convention.
    \begin{definition}\label{def:subdivision_functor}
	By a \define{subdivision functor}, we mean a pair $(\sd, \lambda)$ consisting of a functor $\sd \colon \sSet \to \sSet$ that preserves colimits and inclusions, together with a natural weak equivalence $\lambda \colon \sd \xRightarrow{\simeq} 1_{\sSet}$.
    \end{definition}
    We now provide a general version of the circumcenter construction of the previous subsection.
	\begin{notation}
	For the remainder of this section, we fix a metric space $M$, together with a Lipschitz function $f \colon M \to \RN$ with respect to the $\infty$-norm on $\RN$. We denote the associated Lipschitz constant by $\Lipf$. 
	$\fmetapprx$ will denote another metric pair with no assumptions on $\fapprx$.
	$M$ will not be assumed to be $\cabK$ or complete, unless stated otherwise later on.
	\end{notation}
\begin{definition}\label{def:pseudo_bar}

	 By a \textit{pseudo-barycenter map} with respect to a subdivision functor $(\sd, \lambda)$ and parameters $\rhoTheta>0, 0 < \conShrTheta \leq 1$ and $\conCosTheta \geq 0$, we mean a map
	$\Theta \colon \sd \Rips[\rhoTheta](M)_0 \to M$, 
	such that
    for all $0 < \delta \leq \rhoTheta$:
\newlist{pseudoprop}{enumerate}{2}
\setlist[pseudoprop,1]{
label=(\alph*),
 ref=(\alph*)
}
\setlist[pseudoprop,2]{
  label=(\alph{pseudopropi}\textsubscript{\roman*}),
  ref=(\alph{pseudopropi}\textsubscript{\roman*}),
  widest=(m\textsubscript{viii}),
  align=left
}
	\begin{pseudoprop}
		\item \label{property:ThetaXLambdaX}For every $x \in \sd \Rips[\delta](M)_0$, we have $d(\lambda(x),\Theta(x)) \leq \conCosTheta  \delta$; 
		\item \label{property:XY}	For  
        $x,y\in \sd \Rips[\delta](M)_0$, if
        there is a $1$-simplex from $x$ to $y$ in $ \sd \Rips[\delta](M)$, 
        then \begin{pseudoprop}
			\item \label{property:lambdaXThetaY} $d(\lambda(x),\Theta(y)) <	\conShrTheta \delta$, 
			\item \label{property:ThetaXThetaY} $d(\Theta(x), \Theta(y))<\conShrTheta \delta$.
		\end{pseudoprop}
\end{pseudoprop}
\end{definition} 
    The additional conditions in the definition of a pseudo-barycenter map (compared to the above discussion) are needed to ensure that the resulting morphism does indeed define a shrinking transformation. 
Let us now explain how to construct a shrinking transformation for the bivariate Rips persistent homotopy type from a pseudo-barycenter map. 
\begin{notation}
For the remainder of this subsection, we fix a subdivision functor $(\sd,\lambda)$, as well as a pseudo-barycenter map $\Theta \colon \sd \Rips[\rho](M)_0 \to M$.
\end{notation}
\begin{construction}\label{con:shrinking_transform}
We use \cref{rec:equ_with_ss,rec:equv_model_simp} and construct the shrinking transformation in the homotopy theory of persistent simplicial sets. Let $\Theta \colon \sd \Rips[\rhoTheta](M)_0 \to M$ be a pseudo-barycenter map with respect to $\sd$ and $\lambda$, and parameters as above. 
Denote $\rho:= \rhoTheta$, and $\conShr := \conShrTheta$, $\conCos := \Lipf \conCosTheta$.
Note that since $\sd$ preserves colimits, we have $\sd (\Rips[\bullet](\mbull))^{\rho,\infty}_0=\sd(\Rips[\rho](M))_0$. Now, under this identification, we can apply \cref{obs:universal_property_rips_filtered_sset,rem:uni_rips_mod_version} to uniquely extend $\Theta$ to a morphism $\tau' \colon \sd( \Rips[\bullet_1](\mbull[\bullet_2]))|_{\prodPosRhoRN} \to \Rips[\conShr \bullet_1]( \mbull[ \bullet_2 + \conCos \bullet_1])|_{\prodPosRhoRN}$ (see the proof of \cref{thm:back_propagation_rips}). Finally, in the persistent homotopy category $\ho(\sSetK^{\prodPosRhoRN})$, we can invert the weak equivalence $\lambda \colon \sd \Rips[\bullet](\mbull) \simeq \Rips[\bullet](\mbull)$ and define $\tau_{\Theta}$ as the following composition:
\[
\tau_{\Theta} \colon \Rips[\bullet_1](\mbull[\bullet_2])|_{\prodPosRhoRN} \xrightarrow{\lambda^{-1}} \sd(\Rips[\bullet_1](\mbull[\bullet_2])|_{\prodPosRhoRN}) \xrightarrow{\tau'} \Rips[\conShr \bullet_1]( \mbull[ \bullet_2 + \conCos \bullet_1])|_{\prodPosRhoRN}.
\]
\end{construction}
Using this construction, we can state the main result of this section.
\begin{theorem}\label{thm:back_propagation_rips}
	Given a pseudo-barycenter map $\Theta \colon \sd \Rips[\rho](M)_0 \to M$, the induced morphism
	\[
	\tau_{\Theta} \colon \Rips[\bullet_1](\mbull[\bullet_2])|_{\prodPosRhoRN} \to \Rips[\conShr\bullet_1](\mbull[\bullet_2 \flow \conCos\bullet_1])|_{\prodPosRhoRN}\]
	of \cref{con:shrinking_transform} is a shrinking transformation. \\ Furthermore, $\tau_{\Theta}$ is $\delta$-compatible with every interleaving $ \Rips[\bullet](\mapprxbull) \simeq_{\epsmet,\epsfun} \Rips[\bullet] (\mbull)$ arising from a correspondence $\fmetapprx \approx_{\epsmet,\epsfun} \fmet$, where $\delta \leq \rho$ and $\epsmet \leq \min\{\rho -\delta, \frac{1-\conShr}{1+\conShr} \delta\}$.
\end{theorem}
\begin{proof} 
	Throughout the proof, we are going to make use of the modified version of \cref{obs:universal_property_rips_filtered_sset} in \cref{rem:uni_rips_mod_version}.
	To simplify notation, we will write $\Rips[\bullet,\bullet]$ for $\Rips[\bullet](\mbull)|_{\prodPosRhoRN}$. We first show that $\tau'$, as defined in the previous construction, is indeed a well-defined morphism of persistent simplicial sets.
	 Observe that, compared to \cref{obs:universal_property_rips_filtered_sset}, the target persistent simplicial set has been reparametrized under the parameter change
    $\bullet'_1 = \conShr\bullet_1$, $ \bullet'_2 = \bullet_2 \flow \conCos\bullet_1.$
	The conditions that we need to verify for $\tau'$ to define a well-defined morphism (\cref{rem:uni_rips_mod_version}) are explicitly given as follows. Let $\delta \in [0,\rho]$ and $u \in \RN$. We need to verify the following: 
	\begin{enumerate}
		\item Let $\sigma \in \sd \Rips[\delta](M)_1$ be a $1$-simplex with vertices $x,y \in \sd \Rips[\delta](M)_0$. Then the inequality
		$	d(\Theta(x), \Theta(y)) < \conShr \delta$
		holds. This condition is assumed by \cref{property:ThetaXThetaY}.
		\item For $x \in \sd \Rips[u,\delta]_0$, it holds that $f(\Theta(x)) \leq u + \conCos \delta.$ This is a consequence of the inequalities
        \[
f(\lambda(x)) \leq u \quad \textnormal{and} \quad |f(\lambda(x)) - f(\Theta(x))|_{\infty} \leq \Lipf d(\lambda(x), \Theta(x)) \leq \Lipf \conCosTheta \delta = \conCos \delta.
        \]
        The first of these holds as $\lambda(x) \in \Rips[u,\delta]_0$. The second holds
        by \cref{property:ThetaXLambdaX}, which implies that $d(\lambda(x), \Theta(x)) \leq \conCosTheta \delta$, together with the assumption that $f \colon M \to \RN$ is $\Lipf$-Lipschitz.
	\end{enumerate}
	Next, let us verify the shrinking transformation condition. We need to show that the diagram 
\begin{diagram}
	{\Rips[\bullet_1, \bullet_2]} & {\sd \Rips[\bullet_1,\bullet_2]} & {\Rips[\conShr\bullet_1, \bullet_2 \flow \conCos\bullet_1]} \\
	&& {\Rips[\bullet_1, \bullet_2 \flow \conCos\bullet_1]}
	\arrow["{\lambda^{-1}}", from=1-1, to=1-2]
	\arrow["s"', curve={height=12pt}, from=1-1, to=2-3]
	\arrow["{\tau'}", from=1-2, to=1-3]
	\arrow["s", from=1-3, to=2-3]
\end{diagram}
	commutes. By precomposing with $\lambda$, this is equivalent to showing that the diagram 
\begin{diagram}
	{\sd \Rips[\bullet_1,\bullet_2]} & {\Rips[\conShr\bullet_1, \bullet_2 \flow \conCos\bullet_1]} \\
	{\Rips[\bullet_1,\bullet_2]} & {\Rips[\bullet_1, \bullet_2 \flow \conCos\bullet_1]}
	\arrow["{\tau'}", from=1-1, to=1-2]
	\arrow["\lambda"', from=1-1, to=2-1]
	\arrow["s", from=1-2, to=2-2]
	\arrow["s"', from=2-1, to=2-2]
\end{diagram}
	commutes. In fact, it turns out that the two morphisms of persistent simplicial sets $s \circ \tau'$ and $s \circ \lambda$ are elementarily homotopic. To see this, consider the map
	\begin{align*}
	H \colon (\sd \Rips[\bullet_1,\bullet_2] \times \Delta^1)_0 = (\sd \Rips[\bullet_1,\bullet_2])_0 \times \{0,1\} &\to M; \quad 
	(x,i) \mapsto \begin{cases}
	\lambda(x), & i=0,\\
    \Theta(x), & i=1,
	\end{cases}
	\end{align*}
    and again apply \cref{obs:universal_property_rips_filtered_sset,rem:uni_rips_mod_version}.
	As we already know that $H_0$ and $H_1$ fulfill the conditions of \cref{obs:universal_property_rips_filtered_sset}, it only remains to verify the condition on $1$-simplices $\sigma  \in (\sd \Rips[\bullet_1,\bullet_2] \times \Delta^1)_1$ in \cref{obs:universal_property_rips_filtered_sset} in the case where the vertices $x$ and $y$ are in $\sd \Rips[\delta](M)_0 \times \{0\}$ and $\sd \Rips[\delta](M)_0 \times \{1\}$ respectively. Observe that, by the definition of the simplicial product, we can identify $\sigma$ with a simplex $\sigma' \in \sd \Rips[\delta](M)_1$, whose vertices we will also denote by $x$ and $y$ by abuse of notation.
	Then, the condition explicitly states that 
	$d( \lambda(x),\Theta(y)) < \delta$
	which follows from \cref{property:lambdaXThetaY} in the definition of a pseudo-barycenter map. This shows that $s \circ \tau'$ and $s \circ \lambda$ are elementarily homotopic; thus, the shrinking transformation condition holds. It remains to show the compatibility with interleavings arising from a filtered correspondence. To this end, let $\mathfrak{C}$ be an $(\epsmet,\epsfun)$-correspondence between $\fmet$ and $\fmetapprx$, and let $\varphi$ and $\psi$ be the induced $(\epsmet,\epsfun)$-interleaving maps. Furthermore, let $\epsmet \leq \min\{\frac{1-\conShr}{1+\conShr} \delta, \rho - \delta\}$. We write $\mathcal{S}^{\bullet_1,\bullet_2}$ for $\Rips[\bullet_1](\mapprxbull[\bullet_2])|_{\prodPosRhoRN}$ and $\mathcal{S}^{\bullet_1}$ for $\Rips[\bullet_1](\mathbb{M})$. We now need to verify the commutativity of the outer diagram in
	\vspace{0.5cm} \newline  \noindent
\begin{adjustbox}{scale = 0.8,center}
\begin{tikzcd}[column sep = tiny]
	& {\Rips[\delta+\epsmet, \bullet_2 \flow \epsfun]} & {\sd \Rips[\delta+\epsmet, \bullet_2 \flow \epsfun]} & {\Rips[\conShr(\delta+ \epsmet), \bullet_2 \flow \epsfun + \conCos(\delta+ \epsmet)]} \\
	{\mathcal{S}^{\delta,\bullet_2}} && {\sd \mathcal{S}^{ \delta, \bullet_2}} && {\mathcal{S}^{\conShr(\delta+ \epsmet) + \epsmet,\bullet_2 \flow2\epsfun + \conCos(\delta+ \epsmet)}} \\
	&& {\mathcal{S}^{\delta,\bullet_2\flow2\epsfun + \conCos(\delta+ \epsmet)}}
	\arrow["{\lambda^{-1}}", color={rgb,255:red,214;green,92;blue,92}, from=1-2, to=1-3]
	\arrow["\tau'", color={rgb,255:red,214;green,92;blue,92}, from=1-3, to=1-4]
	\arrow["\psi", color={rgb,255:red,214;green,92;blue,92}, from=1-4, to=2-5]
	\arrow["\varphi", color={rgb,255:red,214;green,92;blue,92}, from=2-1, to=1-2]
	\arrow["{\lambda^{-1}}", from=2-1, to=2-3]
	\arrow["s"', color={rgb,255:red,214;green,92;blue,92}, from=2-1, to=3-3]
	\arrow["{\sd \varphi}", from=2-3, to=1-3]
	\arrow["{s \circ \lambda}"{description}, from=2-3, to=3-3]
	\arrow["s", color={rgb,255:red,214;green,92;blue,92}, from=2-5, to=3-3]
\end{tikzcd}
\end{adjustbox}
	which we have marked in red. Observe that the left square commutes by the naturality of $\lambda$, and that the lower left triangle commutes by definition. It thus suffices to show that the remaining cell to the right commutes. To this end, we show that the two morphisms of persistent simplicial sets $s \circ \psi \circ \tau' \circ \sd \varphi$ and $s \circ \lambda$ are elementarily homotopic. To see this, we again use \cref{obs:universal_property_rips_filtered_sset,rem:uni_rips_mod_version}, but this time applied to $\sd \mathcal{S}^{\delta,\bullet_2} \times \Delta^1$. Arguing exactly as above, we define a homotopy by extending the map 
	\begin{align*}
	H' \colon (\sd \mathcal{S}^{\delta,\infty} \times \Delta^1)_0 = (\sd \mathcal{S}^{\delta})_0 \times \{0,1\} &\to \mapprxbull[]; \quad 
	(x,i) \mapsto \begin{cases}
	\lambda(x), & i=0.\\
    (\psi \circ \tau'\circ \sd \varphi)(x), & i=1;
	\end{cases}
	\end{align*}
	By the same arguments as above, we can reduce to proving the following condition on $1$-simplices $\sigma  \in (\sd \mathcal{S}^{\delta,\bullet_2} \times \Delta^1)_1$ in \cref{obs:universal_property_rips_filtered_sset} in the case where the vertices $x$ and $y$ are in $\sd (\mathcal{S}^{\delta,\bullet_2})_0 \times \{0\}$ and $\sd(\mathcal{S}^{\delta,\bullet_2})_0 \times \{1\}$, respectively: Given such a $\sigma$, it holds that 
	\[
	d(\lambda(x),(\psi \circ \tau' \circ \sd \varphi)(y)) < \delta.
	\]
	To see this, we first apply the defining property of a correspondence to obtain 
	\[
	d(\lambda(x),(\psi \circ \tau' \circ \sd \varphi)(y)) \leq \epsmet + d( \varphi(\lambda(x)),(\tau' \circ \sd \varphi)(y)) = \epsmet + d(\varphi(\lambda(x)),\Theta((\sd \varphi)(y))).
	\]
	By naturality of $\lambda$, we have $\varphi(\lambda(x)) = \lambda((\sd \varphi)(x))$. Hence, we only need to show that 
	$
	d(\lambda( (\sd \varphi)(x)),\Theta((\sd \varphi)(y)) ) < \delta - \epsmet.$
	By assumption, there exists a $1$-simplex $\sigma' \in \sd \mathcal{S}^{\delta,\bullet_2}_1$ from $x$ to $y$. Consequently, $(\sd \varphi)(\sigma') \in \sd \mathcal{R}^{\delta+\epsmet, \bullet_2 \flow \epsfun}_1$ is a $1$-simplex from $(\sd \varphi)(x)$ to $(\sd \varphi)(y)$ with $\delta + \epsmet \leq \rho$. By \cref{property:lambdaXThetaY}, we thus have 
	\[
	d(\lambda((\sd \varphi)(x)),\Theta((\sd \varphi)(y))) < \conShr(\delta + \epsmet).
	\]
	Finally, the assumption $\epsmet \leq \frac{1-\conShr}{1+\conShr} \delta$ implies that
	$\conShr(\delta + \epsmet) \leq \delta - \epsmet,$
	which concludes the proof.
\end{proof}
\subsection{The approximation theorem for the function-Rips homotopy type}
We can now combine the main results of the previous sections to prove the following result. We note that it is strictly weaker than the main stability result for function-Rips complexes in \cref{thm:local_stability_function_rips}. We nevertheless state it here, since it provides a conceptual intermediary step, and the techniques used here will be used again in the proof of \cref{thm:local_stability_function_rips}. 
\begin{notation}
For the remainder of this subsection, we fix a complete $\cabK$-space $M$, for some $\kappa \in \mathbb{R}$, together with a Lipschitz function $f \colon M \to \RN$ with Lipschitz constant $\Lipf$. We furthermore fix another metric pair $\fmetapprx$ with no assumptions on $\fapprx$.
\end{notation}
\begin{theorem}\label{thm:weaklocalStabFull} 
Let $0 <\delta \leq \catRad$. Furthermore, let $\epsmet,\epsfun \geq 0$ be such that 
\[\epsmet \leq \catRad - \delta \quad \text{and} \quad \epsmet \leq \frac{1-\conShrM[\delta +\epsmet]}{1+\conShrM[\delta +\epsmet]}\delta. \]
Then any correspondence $\fmetapprx \approx_{\epsmet,\epsfun} \fmet$ induces an interleaving \[\Rips[\delta](\mapprxbull) \simeq _{\Lipf \conShrM(\delta + \epsmet) + \epsfun} \Rips[\delta](\mbull) \] in the persistent homotopy category.
\end{theorem}
\begin{remark}
We will additionally assume that $0<\conShrM$, i.e., that $M$ is non-discrete (see \cref{rem:discrete_case}), and that $\conShrM$ can serve as the shrinking constant of a pseudo-barycenter map. We note that the the discrete case is much simpler and is addressed in \cref{subsec:proof_of_main_result}. We furthermore assume that $\conShrM < 1$, which is the case whenever $\rho < \catRad$ (see \cref{rem:explicit_numbers}). Note that when $\conShrM =1$, then $\epsmet = 0$, and this case is also much simpler (see also \cref{subsec:proof_of_main_result}).
\end{remark}
To prove this result let us construct a pseudo-barycenter map for $M$ with respect to the barycentric subdivision functor $\sd_b$ and the last vertex map $\lambda_b$. 
\begin{example}\label{ex:existence_of_pseudo_bar_classical}
Let $\rho \leq \catRad$. We can then consider the map $\Theta$ that sends a vertex $\sigma=(x_0, \dots, x_n) \in \sd_b (\Rips[\rho](M))$ to the circumcenter of $\{x_0, \dots, x_n\} \subset M$, as illustrated in \cref{fig:pseudo-barycenter_map}. 
\end{example}
\begin{proposition}\label{prop:existence_of_pseudo_bar_classical}
	Let $\rhoTheta = \rho  \leq \catRad$ and $\conCosTheta=\conShrTheta=\conShrM$ where we assume that $0<\conShrM<1$. Then the map $\Theta$ from \cref{ex:existence_of_pseudo_bar_classical} is a pseudo-barycenter map with respect to the barycentric subdivision functor $\sd_b$, the last vertex map $\lambda_b$ and parameters as above.
\end{proposition}
\begin{proof}
	Note, first, that when $\conCosTheta=\conShrTheta$, \cref{property:ThetaXLambdaX} is just a special case of \cref{property:lambdaXThetaY}, using the degenerate $1$-simplex. Hence, it suffices to verify the second condition in \cref{property:XY}.
	It is immediate from \cref{lem:distance_bounds_circum} that $\Theta$ is well-defined and has the property that whenever $x,y \in \sd_b \Rips[\delta](M)_0$ are connected by a $1$-simplex, then $d(\Theta(x), \Theta(y)) < \conShrM \delta$, which ensures \cref{property:ThetaXThetaY}. Note, furthermore, that for $(y_0, \dots, y_n) \in \sd_b \Rips[\delta](M)_0$ we have $\lambda_b(y_0, \dots, y_n) = y_n = \Theta((y_n))$. Hence, \cref{property:lambdaXThetaY} follows from \cref{property:ThetaXThetaY} by setting $x = (y_n)$.
\end{proof}
\begin{remark}\label{rem:better_subdivision}
	The proof of \cref{prop:existence_of_pseudo_bar_classical} indicates that, in this special case, the definition of a pseudo-barycenter map is highly redundant. However, the full strength of the definition is needed to construct pseudo-barycenter maps that enable us to get the constant $\conCos$ arbitrarily close to $0$ at the cost of increasing $\conShr$ towards $1$. Doing so allows us to procure stability results for function-Rips complexes in \cref{sec:local_stability_function_rips}.
\end{remark}
Given this result, we can now provide the proof of \cref{thm:weaklocalStabFull}.
\begin{proof}[Proof of \cref{thm:weaklocalStabFull}]
Set $\rho = \delta + \epsmet$.
By \cref{prop:bivariate_interleaving}, the correspondence $\fmetapprx \approx_{\epsmet,\epsfun} \fmet$ gives rise to an interleaving $\varphi \colon \Rips[\bullet](\mapprxbull) \simeq_{\epsmet,\epsfun} \Rips[\bullet] (\mbull) \colon \psi$ in the homotopy category. By \cref{prop:existence_of_pseudo_bar_classical}, $\fmet$ admits a pseudo-barycenter map $\Theta$ with respect to $\rho \leq \catRad$ and $\conCosTheta= \conShrTheta = \conShrM$. 
By \cref{thm:back_propagation_rips}, $\Theta$  induces a shrinking transformation $\tau$ compatible with $\varphi$ and $\psi$, with parameter $\rho = \rhoTheta$, $\conCos= \Lipf \conCosTheta$ and $\conShr = \conShrTheta$. Applying \cref{lem:munchkin_lemma}, we obtain an interleaving $\Rips[\delta](\mapprxbull) \simeq_{ \epsfun +\conCos(\epsmet + \delta)} \Rips[\delta] (\mbull)$.
\end{proof}
\section{Perturbative stability at spaces of bounded curvature}\label{sec:local_stability_function_rips}
Note that \cref{thm:weaklocalStabFull} guarantees a bounded (homotopical) interleaving distance between $\Rips[\delta](\mapprxbull)$ and $\Rips[\delta](\mbull)$. It does, however, not guarantee that this interleaving distance goes to $0$ when $\epsmet$ and $\epsfun$ go to $0$; in this case, the remaining $\delta$-term dominates the interleaving distance.
In this sense, \cref{thm:weaklocalStabFull} does not provide a pointwise continuity or perturbative stability result for the assignment $\fmetapprx \mapsto \Rips[\delta](\mapprxbull)$. Note first that, globally, no such result can ever be expected. Already in the case $N=0$, the persistent homotopy category $\ho (\iSpaces^{\RN}) = \ho(\iSpaces)$ is discrete in the sense that any two distinct homotopy types have infinite interleaving distance, where $\MetRN$ is the (large) space of all metric spaces equipped with the topology induced by the Gromov-Hausdorff distance. Hence, for the assignment to be globally continuous, it would have to be constant on the path-connected component given by bounded metric spaces (see \cite{IvanovNikolaevaTuzhilin2016GHStrictlyIntrinsic}), which it is evidently not. However, one can nevertheless ask the question of \textit{where the assignment is continuous or stable}. For many intents and purposes, for example, when trying to infer information about a geometrically well-behaved space $\fmet$ from an approximation $\fmetapprx$, this suffices (compare \cite{andre2025estimating}).
Let us now discuss how to improve \cref{thm:weaklocalStabFull} to a stability-at-a-point result. In this section, we prove the following.
\begin{notation}
For the remainder of this section we fix a complete $\cabK$-space $M$, for some $\kappa \in \mathbb{R}$, together with a Lipschitz function $f \colon M \to \RN$, with Lipschitz constant $\Lipf$ with respect to the $\infty$-norm on $\RN$. We furthermore fix another metric pair $\fmetapprx$ over $\RN$, with no further assumptions on $\fapprx$ and $\metapprx$. 
\end{notation} 
\begin{notation}
	Given $\rho \leq \catRad$, we denote $\lipConstRho:= \frac{2\conShrM}{1-\conShrM}$. When $\conShrM =1$, we treat the expression $\frac{2\conShrM}{1- \conShrM}$ as $0$ and formally specify $\lipConstRho 0 = 0$.
\end{notation}
\begin{remark}\label{rem:discrete_later}
     We will furthermore assume that $\conShrM >0$, i.e., that $M$ is not discrete (see \cref{rem:discrete_case}). We note, however, that the statements below remain true even in this case if one additionally assumes that $\epsmet < \delta$, which is automatic when $\conShrM > 0$. Unsurprisingly, this discrete case is much simpler and a separate proof is discussed in \cref{appendix:discrete_case}.
\end{remark}
\begin{theorem}\label{thm:local_stability_function_rips}
	Let $0 < \delta \leq \catRad$. 
     Fix any $\rho$ with $\delta \leq\rho \leq \catRad$. Let $\epsmet \geq 0$ be such that
	 \[
	 \epsmet \leq \rho - \delta \quad \text{and} \quad \epsmet \leq \frac{1 - \conShrM}{1+ \conShrM} \delta,
	\]
	and let $\epsfun \geq 0$.
	Then any correspondence $\fmetapprx \approx_{\epsmet, \epsfun} \fmet$ induces an interleaving  \[
	\Rips[\delta](\mapprxbull) \simeq _{\Lipf \lipConstRho \epsmet + \epsfun} \Rips[\delta](\mbull) \] in the persistent homotopy category.
\end{theorem}
Using \cref{thm:local_stability_function_rips}, we can now provide the following persistent Latschev theorem, by combining \cref{thm:persistent_hausmann} and \cref{thm:local_stability_function_rips}.
\begin{theorem}\label{thm:strengthened_persistent_latschev}
Let $0 <\delta \leq \catRad$. Fix any $\rho$ with $\delta \leq\rho \leq \catRad$. Let $\epsmet \geq 0$ be such that
	 \[
	 \epsmet \leq \rho - \delta \quad \text{and} \quad \epsmet \leq \frac{1 - \conShrM}{1+ \conShrM} \delta,
	\]
	and let $\epsfun \geq 0$.
    Then any correspondence $\fmetapprx \approx_{\epsmet,\epsfun} \fmet$ induces an interleaving 
     \[\Rips[\delta](\mapprxbull) \simeq _{\Lipf(\conShrM[\delta] \delta + \lipConstRho\epsmet) + \epsfun} \mbull\]  
     in the persistent homotopy category.
\end{theorem}
\subsection{A Lipschitz-at-a-point interpretation of \cref{thm:local_stability_function_rips}}\label{subsec:lipschitz_at_a_point}
So far, we have started with a fixed $\cabK$-space $M$ and asked for which parameters $\delta$ the function-Rips complex construction is well-behaved at $M$. It is insightful to turn this question on its head. Namely, we can fix a $\delta$ and ask what the continuity properties of the assignment $\fmetapprx \mapsto \Rips[\delta](\mapprxbull)$ are. As already discussed in \cref{sec:local_stability_function_rips}, the attempt to express \cref{thm:local_stability_function_rips} merely as a Lipschitz continuity statement is necessarily lossy. Amongst other reasons, this is due to the necessary coupling of distance and functional distortion. To retain a certain degree of flexibility, we will use the rescaled distances of \cref{not:reparam_dist} that at least allows for the two distortions to be treated at different scales.
\begin{notation}
For the remainder of this subsection, let $K >0$.
We equip the class of metric spaces over $\RN$, $\MetRN$, with the rescaled correspondence distance, $\dCor[\alpha_K]$(see \cref{not:reparam_dist}) induced by the norm $(x,y) \mapsto \|(x,\frac{1}{K}y)\|_{\infty}$. Furthermore, recall the interleaving in the homotopy category distance, $\dIntHo$, on $\iSpaces^{\RN}$ (\cref{con:inho_distance}).
\end{notation}
As already discussed, the assignment 
\begin{align*}
\MetRN &\to \iSpaces^{\RN}\\
\fmetapprx &\mapsto \Rips(\mapprxbull)
\end{align*} 
is generally not continuous with respect to these distances. The following definition captures the idea of a map being Lipschitz only at certain points.\footnote{There does not seem to be a standard source or nomenclature for this notion, even though it appears at several points in the literature under mild variations (\cite{Cheeger1999,DurandCartagenaJaramillo2010}). Compare also \cite{MaderWaas2024} for related results for persistent stratified homotopy types.}
\begin{definition}
	Let $L>0$.
    Let $\varphi \colon X \to Y$ be a map between metric spaces and let $L >0$. Fix $x \in X$ and a neighborhood $U$ of $x$. We say that $\varphi$ is \define{$L$-Lipschitz at $x$ on $U$} if for every $x' \in U$, we have \[ d(\varphi(x), \varphi(x')) \leq L d(x,x').\] 
	There is also the following infinitesimal version of this condition.
	Given $x \in X$, the map $\varphi$ is called \define{$L$-Lipschitz at $x$} if one of the following two equivalent conditions holds:
\begin{itemize}
	\item The inequality \[\limsup_{x_n \to x} \frac{d(\varphi(x), \varphi(x_n))}{d_X(x,x_n)} \leq L\] holds for every sequence $(x_n)_{n \in \mathbb{N}}$, $x_n \neq x$, converging to $x$.
	\item For every $L'>L$, there exists a neighborhood $U$ of $x$ such that $\varphi$ is $L'$-Lipschitz at $x$ on $U$.
\end{itemize}
	In other words, at a point $x$, convergence behavior of $\varphi(x')$ to $\varphi(x)$ is no worse than $L$ times the convergence behavior of $x'$ to $x$. 
\end{definition}
\begin{notation}
For the remainder of this section, fix $K>0$ and $L >0$.
Let us denote by $\cat[CBA](\kappa,L)_{\delta} \subset \MetRN$ the subspace of those $\fmet$ which are such that $f$ is $L$-Lipschitz and that $M$ is a complete $\cabK$ space, for some $\kappa \in \mathbb{R}$, with $\delta < \catRad$. If $\delta < \frac{\varpi_{\kappa}}{2}$, then $\cat[CBA](\kappa,L)_{\delta}$ contains all pairs $\fmet$ for which $M$ is a complete $\catK$ space and $f$ is $L$-Lipschitz. More than that, it even suffices that $M$ is a compact $\cabK$ space with convexity radius larger than $\delta$.
\end{notation}
With this language, one obtains the following immediate consequence of \cref{thm:local_stability_function_rips}.
\begin{notation}
	Given $\rho$ with $0 \leq \rho \leq \catRad$, we denote  $r^{\rho}_M := \min\{\rho - \delta, \frac{1- \conShrM}{1+ \conShrM} \delta\}$. Furthermore, we denote 
	\[
	\flexLipConstM[\rho] := L\frac{2 \conShrM}{1- \conShrM} + K.
	\]
	In the special case where $K = L$, this simplifies to 
	\[
	\flexLipConstM[\rho] := L\frac{2 \conShrM}{1- \conShrM} + L = L \frac{1+\conShrM}{1-\conShrM}.
	\]
\end{notation}
\begin{theorem}\label{thm:pointwise_lipsch_continuity}
 Let $0 <\delta < \catRad$. Let $\fmet \in \cat[CBA](\kappa,L)_{\delta}$. Fix any $\rho$ with $\delta < \rho \leq \catRad$ and assume $\rho < \frac{\varpi_\kappa}{2}$. 
 Then the assignment
\begin{align*}
	 \MetRN &\to \iSpaces^{\RN}\\
\fmetapprx &\mapsto \Rips(\mapprxbull)
\end{align*}
is $\flexLipConstM[\rho]$-Lipschitz at $\fmet$ on the open ball of radius $r^{\rho}_M$ around $\fmet$.
\end{theorem}
Using the infinitesimal version of being Lipschitz at a point, we can furthermore provide the following interpretation of \cref{thm:local_stability_function_rips}.
\begin{notation}
    Given $\kappa \in \mathbb{R}$ and $\delta \in [0, \frac{\varpi_{\kappa}}{2})$ recall the model Jung's constant $\conShrK[\delta] \in [\frac{1}{\sqrt{2}}, 1)$ from \cref{not:conJungGlob}. By \cref{cor:global_bound}, we have $\conShrM[\delta] \leq \conShrK[\delta] < 1$. 
	Similarly to the case of a fixed $M$, we denote
    \[\flexLipConstkappa[\delta] := L\frac{2 \conShrK[\delta]}{1- \conShrK[\delta]} + K. \]
\end{notation}

\begin{corollary}
	Let $\kappa \in \mathbb{R}$ and $0 \leq \delta < \frac{\varpi_{\kappa}}{2}$. Then the assignment
	\begin{align*}
	 \MetRN &\to \iSpaces^{\RN}\\
	\fmetapprx &\mapsto \Rips(\mapprxbull)
	\end{align*}
is $\flexLipConstkappa[\delta]$-Lipschitz at every $\fmet \in \cat[CBA](\kappa,L)_{\delta}$.
\end{corollary}
\begin{example}
	When $K = L =1$ and $\delta \leq \frac{\varpi_{\kappa}}{4}$, we have $\conShrK[\delta] < \frac{3}{4}$, and thus 
	\[
	\flexLipConstkappa[\delta] < 1 \cdot \frac{2 \cdot \frac{3}{4}}{1 - \frac{3}{4}} + 1 = 7
	\]
	independently of $\kappa$. 
\end{example}
\medskip
To summarize, while the assignment $\fmetapprx \mapsto \Rips(\mapprxbull)$ is evidently not globally Lipschitz, it is nevertheless Lipschitz continuous at every pair in $\cat[CBA](\kappa)_{\delta}$, with a global constant depending only on $\kappa$ and $\delta$.
\subsection{Idea behind the proof of \cref{thm:local_stability_function_rips}}
Let us now discuss the main idea of the proof of \cref{thm:local_stability_function_rips}.
Observe that \cref{lem:munchkin_lemma} gives us an interleaving \[\Rips[\delta](\mapprxbull) \simeq_{\Lipf C\epsmet + \Lipf C\delta + \epsfun} \Rips[\delta](\mbull),\] for $C = \conShrM$. If we do not want the term $\Lipf C\delta$ to dominate for small $\epsmet$ and $\epsfun$, we need to be able to make $C$ arbitrarily small. Now, in the definition of a pseudo-barycenter map, the constant $C$ controls the distance between $\lambda(x)$ and $\Theta(x)$. Hence, to make $C$ small, we need to construct pseudo-barycenter maps $\Theta$ that are \textit{close to the last vertex map $\lambda$}. 
\begin{remark}\label{block:small_scale_Theta_t}
At sufficiently small scale, there is an evident way of achieving this systematically. Suppose that $C =S < 1$ and that $\delta$ is lesser than or equal to the convexity radius of $M$. Then we can modify a pseudo-barycenter map $\Theta$ as follows. Given $x \in \sd \Rips[\delta](M)_0$, consider the constant speed geodesic $\gamma_x \colon [0,1] \to M$ from $\lambda(x)$ to $\Theta(x)$, and define $\Theta_{t}(x) := \gamma_x(t)$ for $t \in [0,1]$. Then, by construction, we have \[
d(\lambda(x), \Theta_t(x)) \leq tC \delta.\] The difficulty here is to verify that \cref{property:XY} still holds for appropriately modified $S_t$. 
\end{remark}
\subsection{Convexity of the distance function}
Let us first focus on \cref{property:lambdaXThetaY}. In this case, we need to understand how the distance between $\lambda(x)$ and $\Theta_t(y)$ behaves as a function of $t$ and the distances $d(\lambda(x), \lambda(y))<\delta$ and $d(\lambda(x), \Theta(y))< S\delta$. 
In fact, we will prove the following statement.
\begin{lemma}\label{lem:dist_lambdaX_ThetaTY} Let $x,y_0,y_1 \in M$ with $d(x,y_0) < \delta \leq \catRad$ and $d(x,y_1) < S\delta$, and let $\gamma \colon [0,1] \to M$ be a constant speed geodesic from $y_0$ to $y_1$. Then, for every $t \in [0,1]$, it holds that
\[
d(x, \gamma(t)) < (1-t)\delta + t S \delta = ((1-t) + t\conShr)\delta.
\]
\end{lemma}
This type of convexity statement is captured by the following notion of convexity in geodesic spaces, i.e., spaces in which every two points are connected by a geodesic.
\begin{recollection}
	Let $X$ be a geodesic space. A function $\varphi \colon X \to \mathbb{R}$ is called convex if for every constant speed geodesic $\gamma \colon [0,1] \to X$ the function $\varphi \circ \gamma \colon [0,1] \to \mathbb{R}$ satisfies the inequality
	\[
	\varphi(\gamma(t)) \leq (1-t)\varphi(\gamma(0)) + t \varphi(\gamma(1)),
	\]
	for every $t \in [0,1]$.
\end{recollection}
In fact, in a $\catK$ space of sufficiently small radius, the distance function to a fixed point is convex. 
\begin{proposition}\cite[Ch. 2 Exercise (1)]{BridsonHaefliger1999}\label{prop:distance_to_point_conv}
	Let $x \in M$ and let $r < \catRad$. Then, the restricted distance function 
\begin{align*}
d_x \colon \closedball[r](x) &\to \mathbb{R};\\
y &\mapsto d(x,y)
\end{align*}
is convex.
\end{proposition}
This was stated as \cite[Ch. 2 Exercise (1)]{BridsonHaefliger1999} in slightly different form. For the convenience of the reader, we provide a proof of this statement in \cref{appendix:elem_trig_proof}.
\cref{lem:dist_lambdaX_ThetaTY} now follows immediately from \cref{prop:distance_to_point_conv}.
\subsection{Distances of points on geodesics}
To verify the second condition in \cref{property:XY} for $\Theta_t$, we need to understand how the geodesic based modification of $\Theta$ affects the distances in \cref{property:ThetaXThetaY}. To this end, we first need to recall the following statement about the (local) convexity of the squared distance function in a $\cabK$ space. It was first proven in \cite{Ohta2007Convexities} in a slightly weaker form and then generalized in \cite[Lemma 5]{Yokota2016}. We state the statement for $\cabK$ spaces here, making use of \cref{rec:facts_about_balls_in_bounded_curvature}.
\begin{theorem}\cite{Yokota2016}\label{thm:Ohta2007Convexities}
Let $M$ be a $\cabK$ space and $r < \catRad$. Then, for every $x \in M$, there exists an $\varepsilon >0$, only depending on $r$ and $\kappa$, such that the squared distance function
\begin{align*}
d_x^2 \colon M &\to \mathbb{R};\\
y &\mapsto d(x,y)^2
\end{align*}
has the following property. Let $y_0,y_1 \in \closedball[r](x)$ and let $\gamma \colon [0,1] \to \closedball[r](x)$ be the unique constant speed geodesic from $y_0$ to $y_1$. Then it holds that
\[
d_x^2(\gamma(t)) \leq (1-t) d_x^2(y_0) + t d_x^2(y_1) - t(1-t)\varepsilon d(y_0, y_1)^2.
\]
In particular, $d_x^2$ is convex on $\closedball[r](x)$.
\end{theorem}
We will employ this result to obtain the following lemma.
\begin{lemma}\label{lem:geodes_dist}
Let $M$ be a $\cabK$ space. Let $x_0,x_1,y_0,y_1 \in M$ be points of pairwise distance strictly smaller than $\catRad$, and let $\gamma, \eta \colon [0,1] \to M$ be constant speed geodesics from $x_0$ to $x_1$ and from $y_0$ to $y_1$, respectively. Then, for every $t \in [0,1]$, it holds that
\[
d(\gamma(t),\eta(t))^2 \leq ((1-t)d(x_0,y_0))^2 + (td(x_1,y_1))^2 + t(1-t)((d(x_0,y_1))^2 + (d(x_1,y_0))^2). 
\]
\end{lemma}
\begin{proof}[Proof of \cref{lem:geodes_dist}]
	For $x \in M$, consider the function
	\begin{align*}
	 d_{x}^2 \colon M &\to \mathbb{R}\\
	 y &\mapsto d(x, y)^2.
	\end{align*}
    Fix any $r < \catRad$ greater than the pairwise distances of the points $x_0,x_1,y_0,y_1$. Now let $x \in \{y_0,y_1\}$. As $x_0,x_1 \in \closedball[r](x)$, it follows from \cref{thm:Ohta2007Convexities} that
    \begin{align}\label{equ:squared_formula_1}
        d(\gamma(t),y_0)^2 &\leq (1-t)d(x_0,y_0)^2 + t d(x_1,y_0)^2,\\\label{equ:squared_formula_2}
	   d(\gamma(t),y_1)^2 &\leq (1-t)d(x_0,y_1)^2 + t d(x_1,y_1)^2,
    \end{align}
    for all $t \in [0,1]$.
    Next, let $x = \gamma(s)$, for some fixed $s \in [0,1]$. Note that $d(\gamma(s),y_0) \leq r$ and $d(\gamma(s),y_1) \leq r$, using the convexity of balls of radius $r <\catRad$. We may thus apply \cref{thm:Ohta2007Convexities} again to obtain
	\begin{equation*}
	    d(\gamma(s),\eta(t))^2 \leq (1-t)d(\gamma(s),y_0)^2 + t d(\gamma(s),y_1)^2.
	\end{equation*}
    Now set $s = t$, to obtain 
    \begin{equation}\label{equ:squared_formula_3}
	    d(\gamma(t),\eta(t))^2 \leq (1-t)d(\gamma(t),y_0)^2 + t d(\gamma(t),y_1)^2,
	\end{equation}
    for any $t \in [0,1]$. 
	Plugging \cref{equ:squared_formula_1,equ:squared_formula_2} into \cref{equ:squared_formula_3}, we obtain
\begin{align*}
	d(\gamma(t),\eta(t))^2 &\leq (1-t)((1-t)d(x_0,y_0)^2 + t d(x_1,y_0)^2) + t((1-t)d(x_0,y_1)^2 + t d(x_1,y_1)^2) \\
	& = ((1-t)d(x_0,y_0))^2 + (td(x_1,y_1))^2 + t(1-t)(d(x_0,y_1)^2 + d(x_1,y_0)^2)
\end{align*}
as claimed.
\end{proof}
\begin{corollary}\label{cor:pseudo_convexity_of_dist}
Let $\delta \leq \catRad$. Let $x_0,x_1,y_0,y_1 \in M$ with $d(x_0,x_1), d(y_0,y_1) < \delta$ and let $\gamma, \eta \colon [0,1] \to M$ be constant speed geodesics from $x_0$ to $x_1$ and from $y_0$ to $y_1$, respectively. Now, suppose that $ S \leq 1$ and that
\begin{align*}
	d(x_0,y_0) &< \delta \textnormal{ and } d(x_1,y_1), d(x_0,y_1), d(x_1,y_0) < S \delta.
\end{align*}
Then, for every $t \in [0,1]$, it holds that
\[
d(\gamma(t),\eta(t)) < (1-t)\delta + t S \delta = ((1-t)1 + t S)\delta.
\]
\end{corollary}
\begin{proof}
	 Applying \cref{lem:geodes_dist} to the points $x_0,x_1,y_0,y_1$ and the geodesics $\gamma$ and $\eta$, we obtain the inequality 
	\[
	d(\gamma(t),\eta(t))^2 < (1-t)^2\delta^2 + t^2(S\delta)^2 + t(1-t)2(S\delta)^2.
	\]
	Now, using that $S \leq 1$, we obtain 
	\begin{align*}
		d(\gamma(t),\eta(t))^2 &< (1-t)^2\delta^2 + t^2(S\delta)^2 + t(1-t)2(S\delta)^2 \\
		&\leq (1-t)^2(\delta)^2 + t^2(S\delta)^2 + t(1-t)2\delta(S\delta) \\
		& = ((1-t)\delta + t S \delta)^2,
	\end{align*} 
	from which the claim follows.
\end{proof}
\subsection{An alternative subdivision functor}
The evident next step would be to use \cref{cor:pseudo_convexity_of_dist} as indicated in \cref{block:small_scale_Theta_t} to verify \cref{property:ThetaXThetaY} for $\Theta_t$. 
Note, however, that this cannot succeed with the ordinary barycentric subdivision functor $\sd_b$. 
\begin{block}\label{block:failure_sd_b}
Indeed, suppose we use the pseudo-barycenter $\Theta$ as constructed via circumcenters in \cref{ex:existence_of_pseudo_bar_classical} and fix $\rho <\catRad$. We are given vertices $a,b \in \sd_b \Rips[\delta](M)_0$ that are connected by a $1$-simplex from $a$ to $b$. Then, by definition of $\sd_b$, we can write $a = (a_0, \dots, a_n)$ and $b = (b_0, \dots, b_m)$ with $\{a_0, \dots, a_n\} \subset \{b_0, \dots, b_m\}$ and $d(b_i,b_j) < \delta$ for all $i,j$. Write $A = \{a_0, \dots, a_n\}$ and $B = \{b_0, \dots, b_m\}$. Then, by definition of $\Theta$, we have $\Theta(a) = \circCent(A)$ and $\Theta(b) = \circCent(B)$ (the respective circumcenters), as well as $\lambda_b(a) = a_n$ and $\lambda_b(b) = b_m$. Now, set $x_0 = a_n$, $x_1 = \circCent(A)$ and $y_0 = b_m$, $y_1 = \circCent(B)$. 
By assumption, as well as \cref{lem:distance_bounds_circum}, we have bounds 
\begin{align*}
	d(x_0,y_0) &< \delta,\\
	d(x_1,y_1) &< \conShrM \delta,\\
	d(x_0,y_1) &< \conShrM \delta,\\
	d(x_1,y_0) &< \delta.
\end{align*}
Observe that these conditions are not sufficient to apply \cref{cor:pseudo_convexity_of_dist} to the geodesics from $x_0$ to $x_1$ and from $y_0$ to $y_1$. 
\end{block}
\begin{block}\label{block:missing_cond}
To apply \cref{cor:pseudo_convexity_of_dist}, we would need the additional condition that $d(x_1,y_0) < \conShrM \delta$. Note that, by \cref{lem:distance_bounds_circum}, it would suffice that $b_m \in A$, to achieve this condition. 
\end{block}
We now describe an alternative subdivision functor that ensures exactly this additional assumption. Roughly speaking, given a simplicial set $X$, this alternative subdivision is obtained by replacing every vertex $\sigma$ in the ordinary barycentric subdivision by a simplex of the dimension of $\sigma$ considered as a simplex of $X$ (see \cref{fig:alt_subdiv_detailed} below).
\begin{construction}[Alternative subdivision]
We first describe the alternative subdivision of simplices. To this end, recall first that the ordinary barycentric subdivision of a standard simplex $\Delta^n$ is the nerve of the set of non-empty subsets of $[n]$. Explicitly, this means that the $k$-simplices $\sigma \in (\sd_b \Delta^n)_k$ are precisely the sequences $\sigma = (\sigma_0, \dots, \sigma_k)$ with $\emptyset \neq \sigma_0 \subset \dots \subset \sigma_k \subset [n]$.
We define $ \sd_r \Delta^n$ as a subsimplicial set $ \sd_r \Delta^n \subset (\sd_b \Delta^n) \times \Delta^n$. Namely, a simplex $((\sigma_0, \dots, \sigma_k),\tau) \in ((\sd_b \Delta^n) \times \Delta^n)_k$ belongs to $(\sd_r \Delta^n)_k$ if $\tau\subset \sigma_0$, where we informally treat $\tau \colon [k] \to [n]$ as its image. That this definition does indeed define a subsimplicial set, i.e., is compatible with functoriality in $\Delta$, is a direct elementary verification. 
It is not hard to see that, given a morphism $[n] \to [n']$ in $\Delta$, the induced map $(\sd_b \Delta^n) \times \Delta^n \to (\sd_b \Delta^{n'}) \times \Delta^{n'}$ is such that it restricts to a simplicial map $\sd_r \Delta^n \to \sd_r \Delta^{n'}$. In this manner, we have defined a functor $\sd_r \colon \Delta \to \sSet$, mapping $[n] \mapsto \sd_r \Delta^n$. 
A natural transformation $\lambda_r \colon \sd_r(\Delta^n) \to \Delta^n$ is given by the composition
\[
\sd_r(\Delta^n) \hookrightarrow (\sd_b \Delta^n) \times \Delta^n \xrightarrow{\pi_{\Delta^n}} \Delta^n.
\]
The functor and transformation on $\Delta$ are then extended to a functor $\sd_r \colon \sSet \to \sSet$, together with a natural transformation $\lambda_r \colon \sd_r \Rightarrow 1_{\sSet}$ via left Kan extension. I.e., explicitly, we have $\sd_r (X) = \varinjlim_{\Delta^n \to X} \sd_r \Delta^n$, where the colimit is taken over the category of arrows $\Delta_{/X}$ from simplices into $X$.
\end{construction}
\begin{figure}
\begin{center}
\usetikzlibrary{calc}

\begin{tikzpicture}[
  scale=1.0,
  line join=round,
  line cap=round,
  every node/.style={font=\scriptsize}
]
  \coordinate (v0) at (0,0);
  \coordinate (v1) at (6,0);
  \coordinate (v2) at (3,5.196152423); 

  \coordinate (e01_0) at ($(v0)!1/3!(v1)$);
  \coordinate (e01_1) at ($(v0)!2/3!(v1)$);

  \coordinate (e02_0) at ($(v0)!1/3!(v2)$);
  \coordinate (e02_2) at ($(v0)!2/3!(v2)$);

  \coordinate (e12_1) at ($(v1)!1/3!(v2)$);
  \coordinate (e12_2) at ($(v1)!2/3!(v2)$);

  \coordinate (g)  at (3,1.732050808); 
  \coordinate (f0) at ($(g)!0.48!(v0)$);
  \coordinate (f1) at ($(g)!0.48!(v1)$);
  \coordinate (f2) at ($(g)!0.48!(v2)$);

  \coordinate (lf0) at ($(f0)!0.2!(g)$);
  \coordinate (lf1) at ($(f1)!0.2!(g)$);
  \coordinate (lf2) at ($(f2)!0.28!(g)$);

  \foreach \a/\b/\c in {
    v0/e01_0/f0,
    v0/e02_0/f0,
    v1/e01_1/f1,
    v1/e12_1/f1,
    v2/e02_2/f2,
    v2/e12_2/f2,
    e01_0/e01_1/f1,
    e01_0/f0/f1,
    e02_0/e02_2/f2,
    e02_0/f0/f2,
    e12_1/e12_2/f2,
    e12_1/f1/f2,
    f0/f1/f2}{
    \draw[thick] (\a)--(\b)--(\c)--cycle;
  }

  \draw[very thick] (v0)--(v1)--(v2)--cycle;

  \foreach \p in {v0,v1,v2,e01_0,e01_1,e02_0,e02_2,e12_1,e12_2,f0,f1,f2}
    \filldraw[fill=white,draw=black,line width=0.6pt] (\p) circle (2.1pt);

  \node[below left]  at (v0) {$({\{0\}},0)$};
  \node[below right] at (v1) {$({\{1\}},1)$};
  \node[above]       at (v2) {$({\{2\}},2)$};

  \node[below] at (e01_0) {$({\{0,1\}},0)$};
  \node[below] at (e01_1) {$({\{0,1\}},1)$};
  \node[left]  at (e02_0) {$({\{0,2\}},0)$};
  \node[left]  at (e02_2) {$({\{0,2\}},2)$};
  \node[right] at (e12_1) {$({\{1,2\}},1)$};
  \node[right] at (e12_2) {$({\{1,2\}},2)$};

  \node[above right] at (lf0) {$({\sigma},0)$};
  \node[above left]  at (lf1) {$({\sigma},1)$};
  \node[below]       at (lf2) {$({\sigma},2)$};
\end{tikzpicture}
\end{center}
\caption{An illustration of $\sd_r \Delta^2$, where $\sigma = \{0,1,2\}$.}
\label{fig:alt_subdiv_detailed}
\end{figure}
\begin{block}\label{rem:description_of_alt_simplices}
    For our purposes here, an explicit description of the simplicial sets $\sd_r X$ defined via left Kan extension will not be important. Instead, it suffices to make the following observations. By construction, $\sd_r X$ comes with a natural transformation \[\pi_b \colon \sd_r X \to \sd_b X,\] obtained on simplices via the composition \[\sd_r \Delta^n \subset \sd_b \Delta^n \times \Delta^n \xrightarrow{\pi_{\sd_b \Delta^n}} \sd_b \Delta^n.\] Together with the map $\lambda_r$, we obtain two maps on the level of vertices 
    \[
    (\pi_b)_0 \colon (\sd_r X)_0 \to (\sd_b X)_0 \textnormal{ and } (\lambda_r)_0 \colon (\sd_r X)_0 \to X_0.
    \]
    Together, these maps associate to a $0$-simplex $a \in (\sd_r X)_0$ a pair $((\pi_b)_0(a), (\lambda_r)_0(a))$, which we denote by $(\sigma_a, x_a)$.
    Recall that each vertex in $\sd_b(X)$ is equivalently a non-degenerate simplex of $X$, and we will treat $\sigma_a$ as such.
    By construction of $\sd_r X$, any such pair $(\sigma_a,x_a)$ has the property that $x_a$ is a vertex of $\sigma_a$. Furthermore, whenever there is a $1$-simplex from a vertex $a$ to a vertex $b$ in $\sd_r X$, then the associated pairs $(\sigma_a, x_a)$ and $(\sigma_b, x_b)$ are such that $x_b$ is also a vertex of $\sigma_a$, and such that $\sigma_a$ is a face of $\sigma_b$. 
\end{block}
\begin{block}
The requirement that $x_b$ is a vertex of $\sigma_a$ is the key additional compatibility condition that we were lacking in \cref{block:missing_cond} for the ordinary barycentric subdivision. 
\end{block}
To finish this section, let us verify that the functor $\sd_r$ together with the natural transformation $\lambda_r$ is indeed a subdivision functor in the sense of \cref{def:subdivision_functor}, i.e., that $\sd_r$ preserves colimits and monomorphisms, and that $\lambda_r$ is a weak equivalence. 
\begin{theorem}\label{thm:properties_alt_subdiv}
	The functor $\sd_r \colon \sSet \to \sSet$ preserves colimits and monomorphisms. Furthermore, the natural transformation $\lambda_r \colon \sd_r \Rightarrow 1_{\sSet}$ is a weak equivalence.
\end{theorem}
Proving this result will take up the remainder of this subsection. To do so, we will need the following well-known fact in simplicial homotopy theory. 
\begin{notation}\label{notation:partial_delta}
    Denote by $\partial \Delta^n$ the boundary of $\Delta^n$, given by the subsimplicial set of $\Delta^n$, whose $k$-simplices are the non-surjective maps $[k] \to [n]$, and denote by $i_n \colon \partial \Delta^n \to \Delta^n$ the canonical inclusion. Given a subset $S \subset [n]$, we will write $\Delta^{S}$ for the subsimplicial set of $\Delta^n $ given by the nerve of $S$. Then $\partial \Delta^n$ can alternatively be described as the colimit $\varinjlim_{S \subset [n], S \neq [n]} \Delta^S$, over the category of all proper subsets of $[n]$.
\end{notation}
\begin{lemma}\label{lem:comparison_lemma}
    Let $F \colon \Delta \to \sSet$ be a functor, and let $LF \colon \sSet \to \sSet$ be its left Kan extension. Then $LF$ preserves all colimits. Furthermore, $LF$ preserves monomorphisms of simplicial sets if and only if, for every $n \geq 0$, the induced morphism 
    \[
    LF(i_n) \colon LF( \partial \Delta^n) \to LF(\Delta^n)
    \]
    is a monomorphism. \\
    Suppose we are given another such functor $G \colon \Delta \to \sSet$, together with a natural transformation $\eta \colon F \Rightarrow G$. Suppose, furthermore, that the left Kan extensions $LG, LF \colon \sSet \to \sSet$ both preserve monomorphisms. 
    Then the induced natural transformation $L\eta \colon LF \Rightarrow LG$ is a weak homotopy equivalence at each $X \in \sSet$ if and only if $\eta \colon F(\Delta^n) \to G(\Delta^n)$ is a weak homotopy equivalence for each $n \geq 0$. 
\end{lemma}
\begin{proof}
	This is just a summary of several facts from the language of model categories. A sketch of a proof is provided in \cref{proof:comparison_lemma}.
\end{proof}
We can now give the proof of \cref{thm:properties_alt_subdiv}.
\begin{proof}[Proof of \cref{thm:properties_alt_subdiv}]
    We apply \cref{lem:comparison_lemma}. Observe that the identity functor $1_{\sSet}$ is the left Kan extension of the inclusion $\Delta \hookrightarrow \sSet$ and evidently preserves inclusions. Hence, it remains to show the following two statements, for any $n \geq 0$:
    \begin{enumerate}
        \item The simplicial map $\sd_r(i_n) \colon \sd_r (\partial \Delta^n) \to \sd_r(\Delta^n)$ is an inclusion. 
        \item The simplicial map $\lambda_r \colon \sd_r ( \Delta^n) \to \Delta^n$ is a weak homotopy equivalence.
    \end{enumerate}
    Let us prove the first assertion. 
    Let us first show that, for any $S \subset [n]$ the induced simplicial map $\sd_r \Delta^S \to \sd_r \Delta^{n}$ is an inclusion. Let $k'+1$ be the cardinality of $S$.
    Under the canonical isomorphism $\Delta^S \cong \Delta^{k'}$, we obtain an identification of $\sd_r \Delta^S$ with $\sd_r \Delta^{k'}$.
    This identification fits into a commutative diagram
\begin{diagram}
	\sd_r \Delta^S \\
	\\
	& {\sd_r \Delta^{k'}} & {\sd_r \Delta^{n}} \\
	& {(\sd_b \Delta^{k'})\times \Delta^{k'}} & {(\sd_b \Delta^{n}) \times \Delta^{n}}
	\arrow["\cong", from=1-1, to=3-2]
	\arrow["{\sd_r ( \Delta^S \hookrightarrow \Delta^n)}", curve={height=-24pt}, from=1-1, to=3-3]
	\arrow[curve={height=18pt}, hook, from=1-1, to=4-2]
	\arrow["{\sd_r ( \Delta^{k'}\hookrightarrow \Delta^n)}", from=3-2, to=3-3]
	\arrow[hook, from=3-2, to=4-2]
	\arrow[hook, from=3-3, to=4-3]
	\arrow[from=4-2, to=4-3]
\end{diagram}
    for an appropriate inclusion $\Delta^{k'} \hookrightarrow \Delta^n$.
    As both $\sd_b$ and taking products preserve inclusions, it follows that the lower horizontal is an inclusion. Hence, by commutativity, so is the upper horizontal and thus the map $\sd_r ( \Delta^S \hookrightarrow \Delta^n)$.
    Explicitly, the image of the map $\sd_r \Delta^S \hookrightarrow \sd_r \Delta^{[n]}$ is given by the pairs 
    $((\sigma_0, \dots, \sigma_k), \tau) \in \sd_r (\Delta^n)_k$ that fulfill $\sigma_0, \dots, \sigma_k, \tau \subset S$.
    We will identify $\sd_r \Delta^S$ with this subset in the following. 
    As monomorphisms in a presheaf category are precisely the indexwise injective maps, to show that $\sd_r (i_n)$ is an inclusion, we just need to show that 
    \[
    \sd_r(\partial \Delta^n)_k \to \sd_r (\Delta^n)_k \subset \sd_b ( \Delta^n)_k \times \Delta^n_k
    \]
    is injective for each $k \geq 0$. Given $S \subset [n]$, we write $U_S$ for $(\sd_r \Delta^{S})_k \subset (\sd_r \Delta^n)_k$. 
    It follows by \cref{notation:partial_delta} that $\sd_r(\partial \Delta^n)_k$ is given by the quotient of $\bigsqcup_{S \subset [n], S \neq [n]} U_S$, where the equivalence relation on $\bigsqcup_{S \subset [n], S \neq [n]} U_S$ is generated by identifying $x \in U_{S}$ with $x \in U_{S'}$, whenever $S' \subset S$. To prove injectivity, it thus suffices to show that for any pair $S,S' \subset [n]$ and any $x \in U_S \cap U_{S'}$, there exists $S'' \subset S \cap S'$ such that $x \in U_{S''}$. Indeed, by the description of the sets $U_{S}$ above, it holds that $U_S \cap U_{S'} = U_{S \cap S'}$. Hence, taking $S'': = S' \cap S$, the claim follows.\\
    Next, let us prove the second assertion. In fact, we prove that the map $\lambda_r \colon \sd_r \Delta^n \to \Delta^n$ is part of a strong deformation retraction with section $\Delta^n \hookrightarrow \sd_r \Delta^n$.
    Note that $\sd_b (\Delta^n) \times \Delta^n$ is the nerve of the product poset $\sd({[n]}) \times [n]$, where $\sd([n])$ is the set of non-empty subsets of $[n]$ ordered by inclusion. This poset admits a canonical inclusion $ \iota\colon [n] \to \sd([n]) \times [n]$, given by $k \mapsto ([n],k)$. This inclusion admits a retraction given by the projection $\pi \colon \sd([n]) \times [n] \to [n]$. The composition $\iota \circ \pi$ is homotopic as a map of posets to the identity on $\sd([n]) \times [n]$, in the sense that the map 
    \begin{align*}
        H \colon (\sd([n]) \times [n]) \times [1] &\to \sd([n]) \times [n] \\
        (x,i) &\mapsto \begin{cases}
            x & i = 0, \\
            \iota \circ \pi(x) & i=1,
        \end{cases}
    \end{align*}
    is an order-preserving map. Applying the nerve functor $N$ to $\iota, \pi$ exposes $\Delta^n = N([n])$ as a strong deformation retract of $N(\sd([n]) \times [n]) = \sd_b (\Delta^n) \times \Delta^n$, with the simplicial homotopy $N(\iota) \circ N(\pi) \simeq 1$ given by 
    \[
    N(H) \colon N(\sd([n]) \times [n] \times [1]) =  \sd_b (\Delta^n) \times \Delta^n \times \Delta^1 \to \sd_b (\Delta^n) \times \Delta^n =  N(\sd([n]) \times [n]).
    \]
    Next, observe that the inclusion $N(\iota) \colon \Delta^n \to \sd_b (\Delta^n) \times \Delta^n$ factors through $\sd_r \Delta^n$. Hence, restricting $N(\pi)$ to $\sd_r \Delta^n$, we expose $\Delta^n$ as a retract of $\sd_r \Delta^n$. Observe, furthermore, that this restriction of $N(\pi)$ is exactly $\lambda_r$. It turns out that $N(\iota)$ and $\lambda_r$ form a strong deformation retract. To see this, we show that $N(H)$ maps simplices of $\sd_r \Delta^n \times \Delta^1$ into $\sd_r \Delta^n$.
    As the restrictions of $N(H)$ at $0$ and $1$ are given the identity and $N(\iota \circ \pi)$, respectively, both of which map $\sd_r \Delta^n$ into itself, it suffices to verify that simplices that are not in $\sd_r \Delta^n \times \partial \Delta^1 = \sd_r \Delta^n \times (\{0\} \sqcup \{1\})$ map into $\sd_r \Delta^n$. 
    These are simplices of the form
    \[
    x = ((\sigma_0, \dots, \sigma_k, \sigma_{k+1} \dots \sigma_m), \tau, (0, \dots, 0,1,\dots, 1)).
    \]
    where the change from $0$ to $1$ happens at the $(k+1)$-th entry, for some $k\leq m$.
    Such a simplex is then mapped to
    \[
((\sigma_0, \dots, \sigma_k, [n], \dots, [n]), \tau).
    \]
    Clearly, the condition $\tau \subset \sigma_0$ is unaffected by this change. 
    Hence, it follows that $N(H)(x) \in \sd_r( \Delta^n)_m$ whenever $x \in ((\sd_r \Delta^n )\times \Delta^1)_m$, as was to be shown. To summarize, we have proven that $\lambda_r$ is even a simplicial homotopy equivalence, i.e., a homotopy equivalence with respect to elementary simplicial homotopies. In particular, it is a weak homotopy equivalence.  
\end{proof}
\subsection{Pseudo-barycenters for the alternative subdivision}
We now perform the construction alluded to in \cref{block:small_scale_Theta_t} for the alternative subdivision $\sd_r$.
\begin{construction}\label{con:Theta_t}
Let $M$ be a complete $\cabK$ space and let $\rho \leq \catRad$. Suppose we are given a pseudo-barycenter map $\Theta \colon \sd_b \Rips[\rho](M)_0 \to M$ at scale $\rho$. 
Given $t \in [0,1]$, we define a map $\Theta_t \colon \sd_r \Rips[\rho](M)_0 \to M$ as follows.
For a vertex $a \in (\sd_r \Rips(M))_0$, we write $(\sigma_a, x_a)$ for the associated pair as in \cref{rem:description_of_alt_simplices}. We then set
\[\Theta_t(a) := \gamma_{x_a, \Theta(\sigma_a)}(t),\]
where $\gamma_{x_a, \Theta(\sigma_a)} \colon [0,1] \to M$ is a constant speed geodesic from $x_a$ to $\Theta(\sigma_a)$. 
\end{construction}
Let us first verify that $\Theta_t$ is well-defined, i.e., that the geodesic $\gamma_{x_a, \Theta(\sigma_a)}$ exists, and uniquely so. To this end, we verify that $d(x_a, \Theta(\sigma_a)) < \delta \leq \catRad$. To see this, the following lemma will be of use.
\begin{lemma}\label{lem:simple_dist_bound}
In the situation of \cref{con:Theta_t}, let $\sigma \in \sd_b \Rips(M)_0$ be given by a sequence $(x_0, \dots, x_n)$ of elements in $M$. Then, for any $i \in \{0, \dots, n\}$, we have $d(x_i, \Theta(\sigma)) < \conShr \delta$.
\end{lemma}
\begin{proof}
Let $x \in M$, which we consider as a map $\Delta^0 \xrightarrow{x} \Rips[\delta](M)$, for any $\delta >0$. Then $x$ gives rise to a vertex of $\sd_b \Rips[\delta](M)$ given by the composition $\Delta^0 = \sd_b \Delta^0  \xrightarrow{ \sd(x)} \sd_b \Rips[\delta](M)$. We denote this vertex by $x'$. Observe that $\lambda_b(x') = x$. Now, note that when $x = x_i$, then there is a $1$-simplex from $x'$ to $\sigma$ in $\sd_b \Rips[\delta](M)$. Hence, \cref{property:lambdaXThetaY} implies that $d(x_i, \Theta(\sigma)) = d(\lambda_b(x'), \Theta(\sigma)) < \conShr \delta$, as was to be shown.
\end{proof}
\begin{theorem}\label{thm:pseudo_bar_for_alt_subdiv}
Suppose that $\Theta$ is a pseudo-barycenter map at scale $\rho$ with parameter $0 < \conShr \leq 1$. Then, for any $t \in [0,1]$, the map $\Theta_t$ is a pseudo-barycenter map at scale $\rho$ with parameters $\conCos_t = t \conShr$ and $\conShr_t = (1-t)1 + t\conShr$.
\end{theorem}
\begin{proof}
Next, let us verify \cref{property:ThetaXLambdaX,property:XY} for $\Theta_t$.
Let $a \in \sd_r (\Rips[\delta](M))_0$. Using \cref{rem:description_of_alt_simplices}, we write $(\sigma_a, x_a)$ with $x_a$ a vertex of $\sigma_a$. By \cref{lem:simple_dist_bound}, we have $d(x_a, \Theta(\sigma_a)) < \conShr \delta$. Hence, the geodesic $\gamma_{x_a, \Theta(\sigma_a)}$ is well-defined and unique, and, by construction, we have 
\[
d(\lambda_r(a), \Theta_t(a)) = d(x_a, \gamma_{x_a, \Theta(\sigma_a)}(t)) < t \conShr \delta = \conCos_t \delta.
\]
Next, let us prove \cref{property:XY}.
To this end, let $a$ and $b$ be the vertices of a $1$-simplex in $\sd_r(\Rips[\delta](M))$, from $a$ to $b$, where $\delta \leq \rho$. By \cref{rem:description_of_alt_simplices} we have that $\sigma_a$ is a face of $\sigma_b$ and that $x_a,x_b$ are vertices of $\sigma_a$. It thus follows that $x_a$ and $x_b$ are both elements of the underlying set of $\sigma_a$ in $M$, and thus that 
\[d(x_a,x_b) <\delta.\] Again, using \cref{lem:simple_dist_bound}, we furthermore obtain that
	\begin{align*}
		d(x_b, \Theta(\sigma_a)), d(x_a, \Theta(\sigma_b))   < \conShr \delta.
	\end{align*}
In addition to this, the face relation between $\sigma_a$ and $\sigma_b$, considered as simplices of $\Rips[\delta](M)$, implies the existence of a $1$-simplex from $\sigma_a$ to $\sigma_b$, when these are considered as vertices of $\sd \Rips[\delta](M)$. Thus, \cref{property:ThetaXThetaY} of $\Theta$ implies that
	\[ 
	d(\Theta(\sigma_a), \Theta(\sigma_b))  <\conShr \delta.
\]
We can now apply \cref{prop:distance_to_point_conv} to the setting 
\begin{equation*}
		x = x_a; \quad y_0 = x_b; \quad y_1 = \Theta(\sigma_b),
\end{equation*}
to obtain 
\[
d(\lambda_r(a), \Theta_t(b)) = d(x_a, \gamma_{x_b, \Theta(\sigma_b)}(t)) < ((1-t) + t S) \delta = \conShr_t \delta,\]
as required.
Furthermore, we can apply \cref{cor:pseudo_convexity_of_dist} to the setting
\begin{equation*}
	x_0 = x_a; \quad x_1 = \Theta(\sigma_a); \quad y_0 = x_b; \quad y_1 = \Theta(\sigma_b).
\end{equation*}
This yields
\[
d(\Theta_t(a), \Theta_t(b)) = d(\gamma_{x_a, \Theta(\sigma_a)}(t), \gamma_{x_b, \Theta(\sigma_b)}(t)) < ((1-t) + t S) \delta = \conShr_t \delta.\]
which finishes the proof.
\end{proof}
\subsection{Proof of the perturbative stability theorem}\label{subsec:proof_of_main_result}
We are now in position to give the proof of \cref{thm:local_stability_function_rips}. Before we begin with the actual proof, there are two degenerate cases that need to be addressed separately, but which are much easier to handle. First of, the case $\conShrM =1$, which can only happen when $\rho = \frac{\varpi_{\kappa}}{2}$. 
\begin{proof}[Proof of \cref{thm:local_stability_function_rips} in the case $\conShrM =1$]
	In this case, we have $\frac{1-\conShrM}{1 + \conShrM} = 0$ and thus $\epsmet = 0$. Hence, any correspondence $\fmetapprx \approx_{\epsmet,\epsfun} \fmet$ defines an isometry $\metapprx \cong M$. We can thus identify $\metapprx$ with $M$ under this isometry. Under this identification, the functional distortion $\epsfun$ then provides an upper bound on $\| \fapprx - f \|_{\infty}$. Hence, it immediately follows that the associated Rips simplicial sets $\Rips(\mapprxbull)$ and $\Rips(\mbull)$ are $\epsfun$-interleaved even through inclusions of simplicial sets, not just in the persistent homotopy category.
\end{proof}
The second degenerate case to be addressed is the one where $\conShrM = 0$, i.e., when $M$ is discrete (see \cref{rem:discrete_case,rem:discrete_later}). As our definition of a pseudo-barycenter map requires $\conShr >0$, and \cref{ex:existence_of_pseudo_bar_classical} does not apply.\footnote{Generally speaking, as $\Rips[0](M) = \emptyset$, we can never expect to procure shrinking transformations with $S=0$. This is ultimately a consequence of our choice of using open as opposed to closed boundary conditions for the Rips complex. These kind of boundary pathologies are usually best handled by passing to the observable setting (see \cite{ChazalCrawleyBoeveyDeSilva2016Observable}), but we do not want to add an additional layer of technicality to the discussion here.} Recall that in this case we assume that $\epsmet < \delta$. The special case is then straightforward, albeit a bit lengthy. We treat it in \cref{appendix:discrete_case}. We think that while one coherent proof covering both cases would be conceptually more pleasing, the discrete case serves as a valuable sanity check for the general case.
\begin{proof}[Proof of \cref{thm:local_stability_function_rips}]
By the previous two cases, we may now assume that $0 < \conShrM <1$. Let $\Theta \colon \sd_b \Rips(M)_0 \to M$ be the pseudo-barycenter map guaranteed by \cref{ex:existence_of_pseudo_bar_classical}. Set $\conCos = \conShr = \conShrM$. Applying, \cref{thm:pseudo_bar_for_alt_subdiv}, we obtain a family of pseudo-barycenter maps $\Theta_t \colon \sd_r \Rips(M)_0 \to M$, for $t \in [0,1]$, with parameters $\conCos_t = t\conShr$ and $\conShr_t = ((1-t) + t\conShr)$. 
Fix $t \in [0,1]$ as
\[
t = \frac{2 \epsmet}{(\delta+\epsmet)(1-\conShr)}. 
\]
Note that the assumption $\epsmet \leq \frac{1-\conShr}{1+\conShr} \delta$ is equivalent to 
\[
2\epsmet \leq  (\delta+ \epsmet) (1-\conShr)
\]
and thus 
\[
t = \frac{2 \epsmet}{(1-\conShr)(\delta+\epsmet)} \leq  1,
\]
making $t$ well-defined. Note, furthermore, that the definition of $t$ is equivalent to requiring that $t$ is such that
\[
\epsmet = \frac{1-S_t}{1+S_t} \delta.
\]
Now, we may proceed exactly as in the proof of \cref{thm:weaklocalStabFull} (using \cref{thm:back_propagation_rips} and \cref{lem:munchkin_lemma}). 
It follows that, any correspondence $\fmetapprx\approx_{\epsmet,\epsfun} \fmet $ (as in the statement of the theorem) gives rise to an interleaving in the persistent homotopy category 
\[
\Rips(\mapprxbull) \simeq_{\alpha} \Rips (\mbull)
\]
with $\alpha = \epsfun + \Lipf t\conShr(\epsmet + \delta)$. Using the specific choice of $t$, we obtain 
\begin{align*}
	\alpha  &=  \Lipf t\conShr(\epsmet + \delta) +\epsfun \\
	&= \Lipf \frac{2 \epsmet}{(1-\conShr)(\delta+\epsmet)} \conShr(\epsmet + \delta) +\epsfun \\
	& =  \Lipf \frac{2 \conShr}{1-\conShr} \epsmet +\epsfun \\
	& =  \Lipf \lipConstRho \epsmet +\epsfun,
\end{align*}
as was to be shown. 
\end{proof}

\section{The case of subspaces of Euclidean space}\label{sec:Euclidean}
From a data-analytical perspective, it is natural to be interested in persistent homotopy inference results for metric pairs $\fmet$, where $M$ is a subspace of Euclidean space~$\R^n$. Indeed, any subset $M \subset \R^n$ comes equipped with a canonical $1$-Lipschitz map $M \hookrightarrow \R^n$ and thus gives rise to a metric pair $\fmet$ by composition with any Lipschitz map $\R^n\to\RN$.
\\
So far, however, we have only discussed persistent homotopy inference results for the case of complete $\cabK$ spaces, which are, at least locally, always assumed to be geodesic. Note that this assumption fundamentally restricts us to the convex setting when working with submanifolds of Euclidean space. This is a consequence of the classical Tietze-Nakajima theorem.
\begin{theorem}\cite{Tietze1928Konvexheit,Nakajima1928KonvexeKurven}
	A closed path-connected subset $A\subset \R^n$ is convex if and only if it is locally convex.
\end{theorem}
\begin{remark}
	If $M\subset \R^n$ is equipped with the inherited metric is locally geodesic, then it is locally convex. If it is complete, then it is closed. Hence, in the Euclidean subspace case, our previous results only apply to disjoint unions of closed convex subsets of $\R^n$. 
\end{remark}
\begin{notation}
	For the remainder of this section, we let $n \geq 0$ and $M \subset \mathbb{R}^n$ be a closed subspace. We equip $M$ with the metric inherited from the Euclidean metric on $\mathbb{R}^n$, and let $f \colon M \to \mathbb{R}^N$ be a Lipschitz function with target a (possibly different) Euclidean space~$\RN$ and Lipschitz constant $\Lipf$. We can then consider the metric pair $\fmet[M]$ given by $M$ together with the map $f$.
\end{notation}
\medskip
We will now explain how our techniques translate to the case of a subspace of Euclidean space. This is primarily to illustrate their wide applicability, and we expect that better quantitative versions can be obtained in this setting.
\subsection{Reach and nearest point projections}
Following the general techniques of this article, the evident approach is to procure pseudo-barycenter maps for the Rips complex $\Rips[\bullet](M)$. To this end, let us recall the notion of reach, which is a classical regularity condition for subsets of Euclidean space (see, for example, \cite{Federer1959CurvatureMeasures}). 
\begin{definition}
Let $Y \subset \mathbb{R}^n$ be a closed subspace. The reach of $Y$ is defined as the supremum over all $\tau \geq 0$ such that for every $x \in \mathbb{R}^n$ with $d(x,Y) < \tau$, the function $ Y \to \mathbb{R}, y \mapsto d(x,y)$ admits a unique minimizer. We denote the reach of $Y$ by $\tau_Y$.
\end{definition}
\begin{notation}
Given a subset $Y \subset \mathbb{R}^n$ and $\tau \in (0, \infty]$, we write $U_{\tau}(Y)$ for the open $\tau$-neighborhood of $Y$, i.e., the set of all points in $\mathbb{R}^n$ that have distance less than $\tau$ to $Y$. We furthermore write $\overline{U}_{\tau}(Y)$ for the closed $\tau$-neighborhood of $Y$, i.e., the set of all points in $\mathbb{R}^n$ that have distance less than or equal to $\tau$ to $Y$.
\end{notation}
\begin{recollection}
On an open $\tau_Y$-neighborhood of $Y$, there is a well-defined nearest point projection $\pi_Y \colon U_{\tau_Y}(Y) \to Y$, which maps a point $x$ to the unique minimizer of the distance function $Y \to \mathbb{R}, y \mapsto d(x,y)$. 
\end{recollection}
\begin{example}
	Compact $C^2$-submanifolds $M$ of $\mathbb{R}^n$ have positive reach. In this case, the subset $\{x \in \mathbb{R}^n \mid d(x,M) < \tau_M\}$ is a tubular neighborhood of $M$ and the projection map $\pi_M$ provides its associated tubular neighborhood projection.
\end{example}
We will need the following classical result concerning the Lipschitz constant of the nearest point projection.
\begin{theorem}\cite[Thm 4.8]{Federer1959CurvatureMeasures}\label{thm:nearest_point_proj}
	Let $\delta < \tau_Y$. Then the nearest point projection \[\pi_Y|_{ \overline{U}_{\delta}(Y)} \colon \overline{U}_{\delta}(Y) \to Y\] is Lipschitz with Lipschitz constant 
	\[
	\frac{\tau_Y}{\tau_Y - \delta} = \left(1- \frac{\delta}{\tau_Y}\right)^{-1}.
	\]
\end{theorem}
Let us now begin by providing the Euclidean subspace version of the inference and stability results of this article. Our strategy here will be the same for all results: Transfer the techniques from the special case of the bounded curvature setting of $\mathbb{R}^n$ to $M$ by using the nearest point projection map $\pi_M$.
\begin{notation}
	For the remainder of this section, we write $\conShrN := \sqrt{\frac{n}{2(n+1)}} < 1$ for the Jung constant of $\mathbb{R}^n$, which is independent of scale. We also assume $n>0$, in order to not have to deal with the case of a single point separately. The latter is already covered by \cref{thm:local_stability_function_rips}.
\end{notation}
We will make use of the following correction term.
\begin{notation}\label{not:proj_error}
Given a subspace $M \subset \mathbb{R}^n$, with positive reach $\tau_M$ and $r>0$ with $r < \frac{\tau_M}{\conShrN}$, we denote 
\[
\projError[r] := \left(1- \conShrN \frac{r}{\tau_M}\right)^{-1}
\]
for the Lipschitz constant of the nearest point projection $\pi_M$ at scale $\conShrN r$ as in \cref{thm:nearest_point_proj}. If $\tau_M = \infty$ or $r =0$, $\projError[r]$ is to be read as $1$.
\end{notation}
\begin{remark}
Note that when $r$ is small relative to $\tau_M$, then $\projError[r]$ is close to $1$. It will turn out that $\projError[r]$ takes the role of a correction term, quantifying the deviation of our sub-Euclidean results compared to the case of $\cabK$ spaces.
\end{remark}
\subsection{Lipschitz stability at Euclidean subspaces with positive reach}
We begin with proving the Euclidean subspace version of \cref{thm:local_stability_function_rips}. 
\begin{notation}
We denote \[
\eucLipConst := \frac{2\conShrN}{1-\conShrN}.
\]
Observe that this is (part of) the Lipschitz factor in front of $\epsmet$ in the stability bound of \cref{thm:local_stability_function_rips} for the case of a (fully dimensional) convex subset of $\R^n$, independent of scale $\rho$. To reduce notational complexity, we will also use the short notation 
\[
\rho_n:= \frac{\rho}{1- \conShrN}.
\]
Recall furthermore the notation for the projection errors $\projError[r]$ from \cref{not:proj_error}. 
\end{notation}
\begin{notation}
The role of the Jung's constant of $M$ at a scale $\delta$ will now be taken by the product 
\[
\projError[\delta] \conShrN 
\]
Note that this constant is well-defined and finite for $\delta < \frac{\tau_M}{\conShrN}$, and that it is strictly less than $1$ for $\delta < \frac{1-\conShrN}{\conShrN}\tau_M$. Note also that for $\delta$ small relative to $\tau_M$, the expression $\projError[\delta] \conShrN $ is close to $\conShrN$. The Lipschitz factor $\Lipf\lipConstRho$ measuring the influence of metric distortion will also be modified by an error term. It is now given by 
\[
\projError[\rho_n]\Lipf\eucLipConst, 
\]
which approaches $\Lipf \eucLipConst$ as $\rho$ converges to $0$.
\end{notation} 
\begin{theorem}\label{thm:local_stability_Euclidean}
	Let $0 \leq \delta < \frac{1- \conShrN}{\conShrN}\tau_M$. Fix $\rho$ such that  $\delta \leq \rho < \frac{1- \conShrN}{\conShrN}\tau_M$.
	Let $\epsmet \geq 0$ be such that
	\[\epsmet \leq \rho - \delta \quad \text{and} \quad \epsmet \leq \frac{1-\projError[\rho]\conShrN}{1+\projError[\rho]\conShrN} \delta \]
	and let $\epsfun \geq 0$.
	Then any correspondence $\fmetapprx\approx_{\epsmet, \epsfun} \fmet[M]$ induces an interleaving
	\[
	\Rips(\mapprxbull) \simeq_{\projError[\rho_n]\Lipf\eucLipConst \epsmet + \epsfun} \Rips (\mbull)
	\]
	 in the persistent homotopy category.
\end{theorem}
\medskip
To prove \cref{thm:local_stability_Euclidean}, we will need the following lemma, which allows us to transfer pseudo-barycenter maps to a subspace of $\mathbb{R}^n$ by using the nearest point projection.
\begin{lemma}\label{lem:shift_to_euclid}
Let $M \subset \mathbb{R}^n$ be a closed subspace with reach $\tau_M > 0$. 
Let $0 < \conShr < 1$, $\conCos > 0$ and let $\rho$ be such that $\conCos\rho < \tau_M$ and such that $\rho < \frac{1 - \conShr}{\conCos}\tau_M$. 
Let $(\sd,\lambda)$ be a subdivision functor on $\sSet$ and let $\Theta \colon \sd \Rips[\rho](\mathbb{R}^n)_0 \to \mathbb{R}^n$ be a pseudo-barycenter map at scale $\rho$ with parameters $\conCos, \conShr$. Then the map
\begin{align*}
\Theta' \colon \sd \Rips[\rho](M)_0 &\to M \\
a &\mapsto \pi_M(\Theta(a))
\end{align*}
is well-defined and is a pseudo-barycenter map at scale $\rho$ with parameters 
\[\conCos' = (1 - \frac{\conCos\rho}{\tau_M})^{-1} \conCos \quad \text{and} \quad \conShr' = (1 - \frac{\conCos\rho}{\tau_M})^{-1}\conShr <1.\]
\end{lemma}
\begin{proof}
	To see that the map is well-defined, observe that, by naturality of $\lambda$, we have $\lambda(a) \in M$ for every vertex $a$ of $\sd \Rips[\rho](M) \subset \sd \Rips[\rho](\mathbb{R}^n)$. Hence, it follows from \cref{property:ThetaXLambdaX} that $\Theta(a) \in \overline{U}_{\conCos\delta}(M)$, for every vertex $a$ of $\sd \Rips[\delta](M)$, with $\delta \leq \rho$. By assumption, we have $\conCos\rho <  \tau_M$, and thus $\Theta(a) \in U_{\tau_M}(M)$, for every vertex $a$ of $\sd \Rips[\rho](M)$. Hence, the nearest point projection $\pi_M$ is well-defined on the image of $\sd \Rips[\rho](M)_0$ under $\Theta$, and thus $\Theta'$ is well-defined.
	\cref{property:ThetaXLambdaX,property:XY} for $\Theta'$ follow immediately from the corresponding properties for $\Theta$ and the Lipschitz constant of $\pi_M$ on $\overline{U}_{\conCos\rho}(M)$.
\end{proof}
In particular, we can apply this lemma to the pseudo-barycenter maps $\Theta_t$ from \cref{thm:pseudo_bar_for_alt_subdiv} arising from the pseudo-barycenter map of \cref{ex:existence_of_pseudo_bar_classical} on $\mathbb{R}^n$. Recall that in this case the shrinking constant $\conShr$ of $\Theta_t$ is $(1-t) + t \conShrN$ and the cost constant $\conCos$ is given by $t \conShrN$.

\begin{corollary}\label{lem:pseudo_bar_for_Euclidean_subspace}
Let $M \subset \mathbb{R}^n$ be a closed subspace with reach $\tau_M > 0$. 
Fix $0 < \rho < \frac{1- \conShrN}{\conShrN}\tau_M$  and let $t \in [0,1]$. Then $M$ admits a pseudo-barycenter map with respect to the alternative subdivision $\sd_r$ at scale $\rho$ with parameters 
\begin{align*}
\conCos_t &= \projError[t\rho] t \conShrN, \\
\conShr_t &= \projError[t\rho]((1-t) + t \conShrN).
\end{align*}
\end{corollary}
\begin{remark}
    Note that in the special case where $t = 0$, \cref{lem:pseudo_bar_for_Euclidean_subspace} is not a corollary of \cref{lem:shift_to_euclid}. Indeed, then $\conShr =(1-t) + t\conShrN = 1$. However, in this case the statement of \cref{lem:pseudo_bar_for_Euclidean_subspace} simply requires a (degenerate) pseudo-barycenter map with constants $\conShr_t =1$ and $\conCos_t = 0$. Such a map is always given by the last vertex map. 
\end{remark}
\medskip
The proof of \cref{thm:local_stability_Euclidean} is now analogous to that of \cref{thm:local_stability_function_rips}, making use of \cref{lem:pseudo_bar_for_Euclidean_subspace}.
We only provide the computation of the interleaving parameter.
\begin{proof} 
	Using notation as in \cref{lem:pseudo_bar_for_Euclidean_subspace}, we again set $t \in [0,1]$ such that 
	\[
	\epsmet = \frac{1-\conShr_t}{1+\conShr_t} \delta.
	\]
	That such a $t$ exists follows directly from the intermediate value theorem and the assumptions on $\rho,\epsmet$ and $\delta$. 
	Solving the defining equation of $\conShr_t$ for $t$ yields
\[
t = \frac{1-\conShr_t}{1-\conShrN - \conShr_t \frac{\conShrN \rho}{\tau_M}}.
\]
Furthermore, the defining equation of $t$ yields,
\[
1-\conShr_t = \frac{2\epsmet}{\delta + \epsmet}.
\]
In particular, we obtain $(1-\conShr_t)(\delta +\epsmet) = 2\epsmet$.
After subtracting $\epsfun$, the homotopy interleaving parameter thus computes to
\begin{align*}
	\Lipf\conCos_t (\delta+ \epsmet) &= \Lipf \projError[t\rho] t \conShrN (\delta+ \epsmet) \\
	&= \Lipf \projError[t \rho] \frac{\conShrN(1-\conShr_t)(\delta +\epsmet) }{1-\conShrN-  \conShr_t\frac{\conShrN \rho}{\tau_M} }  \\
    &= \Lipf \projError[t \rho] \frac{2\conShrN\epsmet }{1-\conShrN-   \conShr_t\frac{\conShrN \rho}{\tau_M}}.  \\
    &= (\Lipf \frac{2\conShrN}{1-\conShrN} \epsmet) \projError[t\rho]  (1 - \conShr_t \frac{\conShrN\rho_n}{\tau_M} )^{-1} \\
    &= (\Lipf \eucLipConst \epsmet)\projError[t\rho] \cdot \projError[\conShr_t \rho_n ].
\end{align*}
Next, let us show that 
\[
\projError[t\rho] \cdot \projError[\conShr_t \rho_n ] = \projError[\rho_n].
\]
Substituting $t$ in $\projError[t\rho]$ we indeed obtain 
\begin{align*}
\projError[t\rho] &= \huge (1 - \frac{(1-\conShr_t) \frac{\conShrN \rho}{\tau_M}}{1-\conShrN - \conShr_t \frac{\conShrN \rho}{ \tau_M}} \huge )^{-1} 
 = \huge( \frac{ 1- \conShrN - \conShr_t \frac{\conShrN\rho}{\tau_M} - (1-\conShr_t) \frac{\conShrN \rho}{\tau_M}}{1-\conShrN - \conShr_t \frac{\conShrN \rho}{ \tau_M}} \huge )^{-1} \\
&= \huge ( \frac{ 1- \conShrN - \frac{\conShrN \rho}{ \tau_M}}{1-\conShrN - \conShr_t \frac{\conShrN \rho}{ \tau_M}} \huge)^{-1} 
= \huge ( \frac{1 - \frac{\conShrN \rho}{(1-\conShrN)\tau_M}}{1 - \conShr_t\frac{\conShrN \rho}{(1-\conShrN)\tau_M}} \huge)^{-1} 
= \frac{\huge(1 - \frac{\conShrN \rho_n}{\tau_M})^{-1}}{ \huge(1 - \conShr_t\frac{\conShrN \rho_n}{\tau_M} \huge)^{-1}} \\
&= \frac{\projError[\rho_n]}{\projError[\conShr_t\rho_n]}.
\end{align*}
Thus, we obtain the interleaving parameter
\begin{align*}
    \Lipf\conCos_t (\delta+ \epsmet) +\epsfun =(\Lipf \eucLipConst \epsmet) \projError[t\rho] \cdot \projError[\conShr_t \rho_n ] + \epsfun= \projError[\rho_n] \Lipf \eucLipConst \epsmet + \epsfun.
\end{align*}

\end{proof}
\subsection{Persistent Hausmann's theorem for subspaces of Euclidean space} Next, let us prove the Euclidean subspace version of the persistent version of Hausmann's theorem (\cref{thm:persistent_hausmann}).
\begin{theorem}\label{thm:hausmann_Euclidean}
	Let $0 < \delta \leq \frac{1- \conShrN}{(1+ \conShrN)\conShrN} \tau_M$. Then there is an interleaving
	\[
	\mbull  \simeq_{\projError[\delta] \Lipf\conShrN \delta} \Rips(\mbull)
	\]
	 in the persistent homotopy category.
\end{theorem}
\begin{remark}
    The strategy of proof is the same as the one for \cref{thm:local_stability_Euclidean}: We use the constructions leveraged in the proof of the persistent Hausmann Theorem \cref{thm:persistent_hausmann}, and modify them through the nearest point projection to $M$. There is, however, a technical difficulty we need to contend with. In the proof of \cref{thm:persistent_hausmann}, we make use of a straight line homotopy between a measure and its center of mass. For this straight line homotopy to remain within the metric Rips at scale $\delta$, we need the distance of the center of mass to any point in the support to be strictly bounded by $\delta$ (see \cref{lem:distances_to_karcher}). However, after projecting onto $M$, this may not necessarily be the case anymore. We can nevertheless replicate the proof through the following trick. 
\end{remark}
\begin{definition}
    Let $\mathcal{C}$ be a homotopy theory and $r \colon X^{\bullet} \to Y^{\bullet}$ a morphism in $\ho(\mathcal{C}^{\RN})$. Let $\varepsilon \geq 0$. We say that another morphism $i \colon Y^{\bullet} \to X^{\bullet + \varepsilon}$ is an $\varepsilon$-section of $r$ if the diagram 
\begin{diagram}
	{}& {X^{\bullet + \varepsilon}} & \\
	{Y^{\bullet}} && {Y^{\bullet + \varepsilon}}
	\arrow["{r^{+ \varepsilon}}", from=1-2, to=2-3]
	\arrow["i", from=2-1, to=1-2]
	\arrow["s"', from=2-1, to=2-3]
\end{diagram}
    in $\ho(\mathcal{C}^{\RN})$ commutes.
\end{definition}
\medskip
 We make use of the following simple observation.
 \begin{lemma}\label{lem:approximate_epis}
     Let $\mathcal{C}$ be a homotopy theory and $a,b \colon X^{\bullet} \to Y^{\bullet}$ two morphisms in $\ho(\mathcal{C}^{\RN})$. Suppose that $r \colon Z^{\bullet} \to X^{\bullet}$ is such that 
     \[
     a \circ r = b \circ r.
     \]
     Suppose, furthermore, that $r$ admits an $\varepsilon$-section. Then the following diagram commutes:
\begin{diagram}
	{}& {X^{\bullet + \varepsilon}} & \\
	{X^{\bullet}} && {Y^{\bullet + \varepsilon}} \\
	& {X^{\bullet + \varepsilon}}
	\arrow["{a^{+\varepsilon}}", curve={height=-12pt}, from=1-2, to=2-3]
	\arrow["s", curve={height=-12pt}, from=2-1, to=1-2]
	\arrow["s"', curve={height=12pt}, from=2-1, to=3-2]
	\arrow["{b^{+\varepsilon}}"', curve={height=12pt}, from=3-2, to=2-3]
\end{diagram}
  \end{lemma}
  Conceptually speaking, \cref{lem:approximate_epis} states that morphisms that admit $\varepsilon$-sections act as a kind of approximate epimorphism: The identity $a \circ r = b \circ r$ guarantees that $a = b$ up to an error of $\varepsilon$.
  \footnote{This conceptual idea can be made more rigorous by using the quantitative categorical language of locally persistent categories of \cite{Scoccola2020LocallyPersistent}.}
    For the context of function-Rips complexes, the existence of shrinking transformations implies the following.
 \begin{proposition}\label{prop:sec_for_inc_rips}
     Let $0 < \delta < \frac{1- \conShrN}{\conShrN}\tau_M$. Denote $\conShr := \projError[\delta] \conShrN$ and $\varepsilon = \Lipf \conShr \delta$. Then the inclusion of function-Rips complexes
     \[
     s \colon \Rips[\conShr\delta](\mbull) \hookrightarrow \Rips[\delta](\mbull)
     \]
     admits an $\varepsilon$-section in the persistent homotopy category $\ho(\iSpaces^{\RN})$. 
 \end{proposition}
 \begin{proof}
     Combining \cref{lem:pseudo_bar_for_Euclidean_subspace} (for $t=1$) and \cref{thm:back_propagation_rips}, and setting $\conCos = \Lipf \conShr$
     we obtain a shrinking transformation
     \[
        \Rips[\bullet_1](\mbull[\bullet_2])|_{\prodPosRhoRN} \to \Rips[\conShr\bullet_1](\mbull[{\bullet_2 + \conCos \bullet_1}])|_{\prodPosRhoRN},
     \]
     for any $\delta \leq \rho < \frac{1- \conShrN}{\conShrN}\tau_M$. Note that $\conCos \delta = \varepsilon$. Evaluating at $\delta$, we obtain a morphism
     \[
     \Rips[\delta](\mbull[\bullet]) \to \Rips[\conShr\delta](\mbull[{\bullet + \varepsilon}]).
     \]
     That this morphism is an $\varepsilon$-section of $s$ (in the persistent homotopy category) is part of the definition of shrinking transformations.
 \end{proof}
 Given this result, we can now prove \cref{thm:hausmann_Euclidean}.
\begin{proof}[Proof of \cref{thm:hausmann_Euclidean}]
	In large parts, the proof proceeds analogously to the one in \cref{sec:Hausmann}, except that one replaces the center of mass retraction on $M$ with the composition of the center of mass retraction $K$ on $\mathbb{R}^n$ and the nearest point projection $\pi_M \colon U_{\tau_M}(M) \to M$. Again, we use the equivalence between metric and ordinary Rips thickenings to instead construct an interleaving at the level of metric Rips complexes (see \cref{prop:metric_vietoris_rips_equ_rips}). The two interleaving morphisms are defined as follows. Denote $\varepsilon :=\projError[\delta] \Lipf\conShrN \delta$. The first interleaving morphism is induced by 
    \begin{align*}
        \MRips(M) &\to M \\
        \mu &\mapsto K(\mu) \mapsto \pi_M(K(\mu)).
    \end{align*}
    Here $K$ is the center of mass map of \cref{lem:cont_of_karcher}, which in the Euclidean setting is just the ordinary mean, and $\pi_M \colon U_{\tau_M}(M) \to M$ is the nearest point projection map on the open $\tau_M$ neighborhood of $M$. Observe that this composition is indeed well-defined. 
    By \cref{lem:distances_to_karcher}, for any $\mu \in \MRips(M)$ there exists an $x \in \supp(\mu)$ such that 
    \begin{equation}\label{inequ:distance_to_mean_euclid_case}
        d(x, K(\mu)) \leq \conShrN \diam(\supp(\mu)) <\conShrN \delta.
    \end{equation}
     By the assumption on $\delta$, we have
    \[
    \conShrN \delta \leq (1 + \conShrN) \conShrN \delta \leq (1-\conShrN) \tau_M \leq \tau_M.
    \]
    Hence, it follows that for a measure $\mu$ supported on $M$ whose support has diameter smaller than $\delta$, the mean value in $\mathbb{R}^n$ lies in $U_{\tau_M}(M)$. Furthermore, \cref{inequ:distance_to_mean_euclid_case} also guarantees that $K(\mu)$ is contained in the closed $(\conShrN \delta)$-neighborhood of $M$.
	Applying the nearest point projection, we obtain 
	\[
	d(x, \pi_M(K(\mu))) = d(\pi_M(x), \pi_M(K(\mu))) \leq \projError[\delta] d(x, K(\mu)) \leq \projError[\delta] \conShrN \diam(\supp(\mu)).
	\]
    By the Lipschitz continuity of $f$, it thus follows that
    \[
    d( f(x), f(\pi_M(K(\mu)))) < \Lipf \projError[\delta] \conShrN \delta = \varepsilon. 
    \]
    Hence, on the level of filtered spaces, we obtain a morphism
    \[
    \varphi \colon \MRips(\mbull) \to \mbull[\bullet + \varepsilon],
    \]
    as required\footnote{Note that, to procure this morphism, it would have sufficed to have $\delta < \frac{1-\conShrN}{\conShrN} \tau_M$.}. 
    For the opposite interleaving direction, we use the morphism
    \[
    \psi \colon \mbull \to \MRips(\mbull[ \bullet + \varepsilon])
    \]
    induced by the inclusion $M \hookrightarrow \MRips$ (\cref{rec:inclusion_of_space_metric_rips})\footnote{Note that in the proof of \cref{thm:persistent_hausmann}, we could have constructed an asymmetric interleaving, using different constants in both directions, and used the interleaving parameter $0$ in this direction. The reason we do not do so will become apparent at the end of the proof.}.
    Evidently, we still have that 
    \[
    \varphi^{+ \varepsilon} \circ \psi = (\mbull \xrightarrow{s} \mbull[\bullet + 2 \varepsilon])
    \]
    even on the level of filtered spaces. For the converse composition, it will be useful to have a notational way of distinguishing between the structure morphisms $s \colon \MRips(\mbull) \to \MRips(\mbull[\bullet + \alpha])$ and $s \colon \mbull \to \mbull[\bullet + \alpha]$, for different degrees $\alpha \geq 0$. Hence, we denote the latter in the form $s_{\alpha}$. We are now looking to show that  
    \[
    \psi^{+ \varepsilon} \circ \varphi = s_{2 \varepsilon}
    \]
    in the persistent homotopy category. Note that we cannot expect to derive this from a straight line homotopy as in \cref{thm:persistent_hausmann}. Essentially, due to the expansive nature of the projection map $\pi_M$, we have no guarantee that the support of the measures $t \mu + (1-t) \pi_M K(\mu)$ is still of diameter smaller than $\delta$.
    Instead, we note that $\psi$ factors as 
\begin{diagram}
	{}& {\mbull[\bullet+ \varepsilon]} & \\
	\mbull && {\MRips( \mbull[ \bullet + \varepsilon])}
	\arrow["{\psi_0^{+\varepsilon}}", from=1-2, to=2-3]
	\arrow["s_{\varepsilon}", from=2-1, to=1-2]
	\arrow["\psi"', from=2-1, to=2-3]
\end{diagram}
    with $\psi_0 \colon \mbull \to \MRips(\mbull[ \bullet ])$ also induced by the inclusion $M \hookrightarrow \MRips(M)$.
    We can now rewrite
    \[
    \psi^{+\varepsilon} \circ \varphi = (\psi_0^{+ \varepsilon} \circ s_{\varepsilon})^{+ \varepsilon} \circ \varphi = \psi_0^{+2 \varepsilon} \circ s_{\varepsilon}^{+\varepsilon} \circ \varphi = (\psi_0^{+ \varepsilon} \circ \varphi)^{+ \varepsilon} \circ s_{\varepsilon}
    \]
    We can thus apply \cref{lem:approximate_epis} with $a = \psi_0^{+ \varepsilon} \circ \varphi$ and $b= s_{\varepsilon}$ and it follows that to see that 
    \[
    \psi^{+ \varepsilon} \circ \varphi  = s_{2 \varepsilon} ( = s_{\varepsilon}^{+\varepsilon} \circ s_{\varepsilon})
    \]
    it  suffices to show that 
    \[
    (\psi_0^{+\varepsilon} \circ \varphi) \circ r = s_{\varepsilon} \circ r,
    \]
    for some morphism $r \colon Z^{\bullet} \to \MRips(\mbull)$ that admits an $\varepsilon$-section in the persistent homotopy category. We take $Z^{\bullet} = \MRips[ \conShr\delta](\mbull)$ with $\conShr := \projError[\delta] \conShrN$ and $r$ to be given by the inclusion $\MRips[\conShr\delta](\mbull) \hookrightarrow \MRips[\delta](\mbull)$. By \cref{prop:metric_vietoris_rips_equ_rips}, $r$ is isomorphic (in the persistent homotopy category) to $\Rips[S \delta] (\mbull) \to \Rips( \mbull)$. By \cref{prop:sec_for_inc_rips}, the latter admits an $\varepsilon$-section. Hence $r$ admits an $\varepsilon$-section. It remains to prove that 
    \[
    (\MRips[\conShr\delta](\mbull) \xrightarrow{(\psi_0^{+\varepsilon} \circ \varphi) \circ r} \MRips[\delta](\mbull[\bullet + \varepsilon]) ) = (\MRips[\conShr\delta](\mbull) \xrightarrow{s_{\varepsilon} \circ r} \MRips[\delta] (\mbull[\bullet + \varepsilon]))
    \]
    in the persistent homotopy category. We again apply the straight line argument as in the proof of \cref{thm:persistent_hausmann}. To this end, we need to verify that for $t \in [0,1]$ and $\mu \in \MRips[\conShr\delta](\mbull)$ the diameter of $\supp(t \mu + (1-t) \pi_M K(\mu) )$ is bounded by $\delta$\footnote{Be aware that any such affine combination of measures is to be understood in the convex space $\MRips(M)$.}. Arguing as before, we obtain that $K(\mu) \in \overline{U}_{\conShrN S \delta}$. Thus, given $x \in \supp(\mu) \subset M$, we have  
    \[
    d(x, \pi_M(K(\mu))) =d( \pi_M(x), \pi_M(K(\mu))) \leq \projError[S\delta] d(x, K(\mu)) <  \projError[S\delta] \conShr \delta,
    \]
    where in the final inequality we again use that the maximal deviation of $x$ from the mean $K(\mu)$ is bounded by the diameter of $\mu$ (see \cref{lem:distances_to_karcher} for the general case). It remains to see that $\projError[S\delta] S \leq 1$. Unraveling the notation and simplifying the expression, this is equivalent to requiring that 
    \[
    \delta \leq \frac{1- \conShrN}{(1+ \conShrN)\conShrN } \tau_M.
    \]
\end{proof}
\subsection{The sub-Euclidean persistent Latschev theorem}
As a consequence, using \cref{thm:hausmann_Euclidean}, we also obtain the following persistent version of Latschev's theorem for subspaces of Euclidean space.
\begin{theorem}\label{thm:latschev_Euclidean}
	Let $0 < \delta \leq  \frac{1- \conShrN}{(1+ \conShrN) \conShrN}\tau_M$. Fix $\rho$ such that  $\delta \leq \rho < \frac{1- \conShrN}{\conShrN}\tau_M$.
	Let $\epsmet \geq 0$ be such that
	\[\epsmet \leq \rho - \delta \quad \text{and} \quad \epsmet \leq \frac{1-\projError[\rho]\conShrN}{1+ \projError[\rho]\conShrN}\delta,\] 
	and let $\epsfun \geq 0$.
	Then any correspondence $\fmetapprx\approx_{\epsmet,\epsfun} \fmet[M]$ induces an interleaving
	\[
	\Rips(\mapprxbull) \simeq_{\projError[\delta]\Lipf  \conShrN\delta + \projError[\rho_n] \Lipf \eucLipConst \epsmet + \epsfun} \mbull
	\]
	 in the persistent homotopy category.
\end{theorem}
\begin{remark}
	It is worthwhile to observe that when $\rho$ is small relative to $\tau_M$, then $\projError[\delta]$, $\projError[\rho]$ and $\projError[\rho_n]$ are close to $1$. Thus, the required bounds on $\epsmet$ and the Lipschitz constants of \cref{thm:hausmann_Euclidean,thm:local_stability_Euclidean,thm:latschev_Euclidean} converge to those of \cref{thm:persistent_hausmann,thm:local_stability_function_rips,thm:strengthened_persistent_latschev} in the case of a convex subset of $\R^n$ (of full dimension), as the scale $\rho$ (and thus $\delta$) goes to zero.
\end{remark}
\begin{remark}\label{rem:Rips-reconstruct_Eucl} 
Taking a constant map $f\colon M\to\R^0$ in the above theorem yields a result on the inference of the homotopy type of a subset~$M\subseteq\R^n$ of positive reach from the Rips complex of a Gromov-Hausdorff approximation~$\metapprx$. This recovers known results on homotopic reconstruction from a single Rips complex in Euclidean spaces~\cite{attali2011vietoris,HomotopyReconstructionChazal2020} (under possibly different bounds), using a very different proof technique. While these results apply more generally to subsets of positive \emph{$\mu$-reach}, it is plausible that the reduction trick to the positive reach case proposed in~\cite{HomotopyReconstructionChazal2020} may apply in our case as well. The comparison of the bounds on the Rips parameter~$\delta$ given by these results requires further investigation, and it is of interest insofar as the bounds given in~\cite{HomotopyReconstructionChazal2020} are worst-case tight, and we have no reason to suspect optimality for our bounds. 
\end{remark}
\section{Acknowledgements}
This work started when the two authors attended the \emph{TDA week} at Kyoto University in June 2025. It then continued as the two authors participated in the thematic program \emph{Topological Data Analysis, Persistence And Representation Theory Intertwined} (TP25TD) at the Okinawa Institute of Science and Technology, from June to August 2025. Lukas Waas acknowledges support from the Deutsche Forschungsgemeinschaft (DFG, German Research Foundation) under Germany’s Excellence Strategy EXC 2181/1 - 390900948 (the Heidelberg STRUCTURES Excellence Cluster). Lukas Waas is a member of the Oxford/Max Planck collaboration, and this research was funded in part by the EPSRC international center to center collaboration grant EP/Z531224/1. We acknowledge the use of ChatGPT 5.4 for the purpose of literature research, editing, proofreading and the creation of \cref{fig:alt_subdiv_detailed,fig:mergetrees}.
\printbibliography
\appendix
\section{Remaining Proofs}
\subsection{Remarks on relating simplicial sets and simplicial complexes}\label{app:simplicial_cplx_vs_set}
Let us make a few remarks on the interaction of simplicial complexes and simplicial sets, which are well-known but often surprisingly hard to find in the literature. In particular, we will use them to derive \cref{obs:universal_property_rips_filtered_sset}.
\begin{remark}\label{rem:adjoint_to_NS}
    Recall that the category of simplicial complexes $\sCplx$ has all colimits. The colimit of a diagram of simplicial complexes $K^{\bullet} \colon I \to \sCplx$ is given by taking the colimit of the underlying sets of vertices, $K:=\varinjlim (K^i_0)$, and then inserting a simplex $\sigma$ in $K$ whenever $\sigma$ is the image of a simplex $\sigma_i $ in $K^i$ under the canonical map $K_0^i \to K_0$. It follows that the simplicial nerve functor $\sNerve \colon \sCplx \to \sSet$ has a left adjoint, which we denote by $\mathcal{L}$ here (for example, by \cite[Thm 1.1.10]{Cisinski2019HCHA}.) Explicitly, this left adjoint is given by mapping a simplicial set $S$ to the simplicial complex $\mathcal{L}(S)$ whose set of vertices is $S_0$, and where a subset $\sigma \subset S_0$ forms a simplex if and only if it is the set of $0$-dimensional faces of some simplex $\tilde{\sigma}$ in $S$.
    It is easy to see that the counit of adjunction $\mathcal{L}(\sNerve(K)) \to K$ is an isomorphism of simplicial complexes for every simplicial complex $K$. Equivalently, this means that $\sNerve$ is fully faithful (\cite[Prop. 3.1]{nlab:reflective_subcategory}). Observe also that the adjunction $\mathcal{L} \dashv \sNerve$ induces an adjunction on the level of categories of persistent objects $\mathcal{L} \colon \sSet^{\prodPosR} \rightleftharpoons \sCplx^{\prodPosR} \colon \sNerve$, which we denote the same by abuse of notation.
\end{remark}
Let us now give proof of \cref{obs:universal_property_rips_filtered_sset}.
\begin{proof}[Proof of \cref{obs:universal_property_rips_filtered_sset}]
    In this proof, we will write $\Rips[\bullet](\mapprxbull)$ for the bivariate Rips complex, considered as a simplicial complex, and $\sNerve \Rips[\bullet](\mapprxbull)$ to refer to its simplicial set version. By \cref{rem:adjoint_to_NS}, it follows that there is a canonical bijection
    \[
    \sSet^{\prodPosR}(X^{\bullet,\bullet},\sNerve\Rips[\bullet](\mapprxbull) ) \cong \sCplx^{\prodPosR}( \mathcal{L}(X^{\bullet,\bullet}),\Rips[\bullet](\mapprxbull)).
    \]
    Now, observe that when $X^{\bullet,\bullet}$ is a filtered simplicial set, then $\mathcal{L}(X^{\bullet,\bullet})$ is a filtered simplicial complex. It is an easily verifiable and classically known fact that filtered simplicial maps from a $\prodPosR$-filtered simplicial complex $\mathcal{L}(X^{\bullet,\bullet}) = K^{\bullet,\bullet}$ into the $\prodPosR$-filtered simplicial complex $\Rips[\bullet](\mapprxbull)$ are in (the obvious natural) bijection with maps $K_0^{\infty} \to \mathbb{M}$, that fulfill precisely the inequalities of \cref{con:requirements_for_universal_prop}. Composing this natural bijection with the one we have just derived from the adjunction $\mathcal{L} \dashv \sNerve$ yields $\eta$. 
\end{proof}
Finally, let us finish this section by describing the canonical equivalence between the realization of a simplicial complex and the realization of its simplicial nerve. 
\begin{construction}\label{rem:equ_of_nerve_w_complex}
    The classical way to consider a simplicial complex $K$ as a simplicial set is to first equip the simplices of $K$ with (compatible) total orderings. Indeed, the category of such ordered simplicial complexes (and order-preserving simplicial maps), which we denote $\sCplx^o$ here, also embeds into simplicial sets by mapping an ordered simplicial complex $O$ to the simplicial set $\sNerve^o(O)$ given at $n \geq 0$ by $\sCplx^o( \Delta^n_o, O)$ (with functoriality in $\Delta$ given by precomposition). Here, $\Delta^n_o$ denotes the ordered simplicial complex given by the set of non-empty subsets of $\{0, \dots, n\}$. The advantage of this approach is that there is a canonical natural isomorphism $\real{O} \cong \real{\sNerve^o(O)}$ (see \cite{OtterMagnitudePersistence2022}). 
    Now, suppose one chooses such an ordering for $K$, and denotes the associated ordered simplicial complex by $K^o$. Then there is a canonical inclusion of simplicial sets $\sNerve^o(K^o) \hookrightarrow \sNerve(K)$. This map is not an isomorphism. Far from it: The non-degenerate simplices of $\sNerve^o(K^o)$ are in bijection with the simplices of $K$, but the non-degenerate simplices of $\sNerve(K)$ are in bijection with all sequences of vertices $(x_0, \dots, x_n)$, for which $x_i \neq x_{i+1}$ for all $0\leq i <n$, such that the set $\{x_0, \dots,x_n\}$ is a simplex in $K$. 
    However, $\sNerve^o(K^o) \hookrightarrow \sNerve(K)$ is a weak homotopy equivalence of simplicial sets (see \cite{CamarenaSsetsFromComplexes}). In particular, one obtains a homotopy equivalence \[\real{K} = \real{K^o} \cong \real{\sNerve^o(K^o)} \xrightarrow{\simeq} \real{\sNerve(K)}.\]
    The problem with this approach is that this map is not natural (at least outside of the homotopy category), as it relies on a choice of ordering of $K$. However, it turns out that it admits a retraction (a map $r$ such that $r\circ i = 1_{\real{K}}$) that is natural and also a homotopy equivalence, as it is a one-sided inverse of a homotopy equivalence. Explicitly, it is defined as follows: Recall that every simplex $\sigma_f$ of $\sNerve(K)$ is a simplicial map $f \colon\Delta^n_c \to K$. The topological realization of $X:=\sNerve(K)$ is the colimit of the diagram 
    \begin{align*}
        \Delta_{/X} &\to \Top \\
         (\sigma_f \colon \Delta^n \to \sNerve{(K)}) &\mapsto \real{\Delta^n}
    \end{align*}
    indexed over the slice category $\Delta_{/X}$ of arrows from $\Delta$ into $X$. A cocone with target $\real{K}$ on this diagram is given by $\real{f} \colon \real{\Delta^n_c} \to \real{K}$ at $\sigma_f \colon \Delta^n \to \sNerve(K)$. The universal property of the colimit then defines a continuous map $\real{\sNerve{(K)}} \to \real{K}$. This map defines a retraction of $i\colon \real{K} \cong \real{\sNerve^o(K^o)} \xhookrightarrow{\simeq} \real{\sNerve(K)}$ and is evidently natural.
\end{construction}
\subsection{Proof of \cref{lem:comparison_lemma}}\label{proof:comparison_lemma}
Here we provide a proof of \cref{lem:comparison_lemma}. 
\begin{proof}[Proof of \cref{lem:comparison_lemma}]
	To prove this statement, we will need to freely make use of the language of model categories (see \cite{Hirschhorn2003}, for a good overview). 
    That the left Kan extension preserves colimits is immediate from it being left adjoint (see, for example, \cite[Thm 1.1.10]{Cisinski2019HCHA}). That $LF$ preserves inclusions if and only if it maps boundary inclusions into inclusions follows from the fact that inclusions in simplicial sets are generated by boundary inclusions under pushouts and transfinite compositions (see, for example, \cite[Ch. 11]{Hirschhorn2003}). For the final claim, observe that the map $L\eta_{X} \colon LF(X) \to LG(X)$ is the colimit of the morphism of diagrams obtained by pulling $\eta \colon F \to G$ back along the canonical functor $\pi \colon \Delta_{/X} \to \Delta$. Let us denote this morphism of $\Delta_{/X}$ indexed diagrams of simplicial sets by $\eta' \colon F \circ \pi \to G \circ \pi$. The condition that $LF(i_n)$ and $LG(i_n)$ are inclusions ensures that $F \circ \pi$ and $G \circ \pi$ are cofibrant with respect to the Reedy model structure and the Kan-Quillen model structure on the target (see the techniques in  \cite[Ch. 15]{Hirschhorn2003} for details). Thus, by making use of the assumption, we obtain that the map $\eta'$ is a weak equivalence of cofibrant diagrams.
    As $\Delta_{/X}$ has fibrant constants (see \cite[Prop 15.10.4]{Hirschhorn2003}), it follows from \cite[Thm. 15.8.10]{Hirschhorn2003} that $\varinjlim \eta' = \eta_X$ is also a weak equivalence.
\end{proof}
\subsection{Convexity of distance functions on the sphere at small scales}\label{appendix:elem_trig_proof}
Let us now provide the proof of \cref{prop:distance_to_point_conv}. 
	\begin{proof}
    By the assumption on $r$, we have that $\closedball[r](x)$ is a $\catK$ space.
	Note first that it suffices to show the case where $\kappa = 1$. Indeed, then the cases where $\kappa >0$ follow by rescaling the metric by $\frac{1}{\sqrt{\kappa}}$. As every $\catK[\kappa']$ space is also a $\catK$ space for $\kappa \leq \kappa'$, this also implies the case $\kappa \leq 0$, but at all radii less than $\frac{\varpi_{\kappa'}}{2}$, for some $\kappa' >0$. Letting $\kappa' \to 0$, for $\kappa' >0$ shows the case of arbitrarily large radius for $\kappa \leq 0$. Next, note that by Reshetnyak's majorization theorem (see \cite[Thm. 2.18]{BridsonHaefliger1999}), we can furthermore reduce to proving the claim in $M_{1}$, the $2$-dimensional unit sphere. \\
	Fix $x \in M_1$ and let $d_x$ be the distance function to $x$, which in this case is given by $d_x(y) = \arccos(\langle x,y \rangle)$, where $\langle -,- \rangle$ denotes the standard inner product on $\mathbb{R}^3$.
	Let $\gamma \colon [0,1] \to M_1$ be a constant speed geodesic, and denote $y_0 = \gamma(0)$ and $y_1 = \gamma(1)$. By assumption, we can furthermore assume that $d_x(y_0), d_x(y_1)< \frac{\pi}{2}$ and $d(y_0,y_1) < \pi$. We now show that $\varphi \colon [0,1] \to \mathbb{R}$, defined by $\varphi(t):= d_x(\gamma(t))$, is a convex function. Note that $\varphi$ is smooth, so it suffices to show that $\varphi''(t) \geq 0$ for every $t \in [0,1]$. \\
	We denote $\alpha = d(y_0,y_1)$. We assume $0<\alpha$, as otherwise the statement is trivial.
	Now, denote by $u \colon [0,1] \to \mathbb{R}$ the composition of $\varphi$ with the cosine function, i.e., $u(t) = \cos(\varphi(t)) = \langle x, \gamma(t) \rangle$.
	Observe first that $u$ satisfies the differential equation \[u'' = -\alpha^2 u.\] Indeed, by the classical parameterization of geodesics on the sphere, we have $\gamma(t) = \frac{\sin((1-t)\alpha)}{\sin(\alpha)}y_0 + \frac{\sin(t\alpha)}{\sin(\alpha)}y_1$, where $\alpha = d(y_0,y_1)$. Hence, we have 
\begin{align*}
u(t) & = \langle x, \gamma(t) \rangle = \frac{\sin((1-t)\alpha)}{\sin(\alpha)}\langle x,y_0 \rangle + \frac{\sin(t\alpha)}{\sin(\alpha)}\langle x,y_1 \rangle
\end{align*}
from which the identity $u'' = -\alpha^2 u$ follows by direct computation. Now, consider the identities
\begin{align*}
	u' = (\cos(\varphi))' & = -\sin(\varphi) \varphi',\\
u'' = (\cos(\varphi))'' & = -\cos(\varphi)(\varphi')^2 - \sin(\varphi) \varphi''.
\end{align*}
Plugging in the identity $u'' = -\alpha^2 u$ into the second equation, we obtain
\begin{align*}
\alpha^2 \cos (\varphi) & = \cos(\varphi)(\varphi')^2 + \sin(\varphi) \varphi'',
\end{align*}
or equivalently
\begin{align*}
\sin(\varphi) \varphi'' & = (\alpha^2-(\varphi')^2) \cos(\varphi).
\end{align*}
Now, observe that $\sin(\varphi),\cos(\varphi) > 0$ by the assumption that $d(x,y_0) < \frac{\pi}{2}$ and $d(x,y_1) < \frac{\pi}{2}$, which implies that $\varphi(t) = d_x(\gamma(t)) < \frac{\pi}{2}$ for every $t \in [0,1]$ by the convexity of balls of radius less than $\frac{\pi}{2}$.
Hence, to see that $\varphi$ is convex, it suffices to show that $(\alpha^2-(\varphi')^2) \geq 0$.
Now, recall that $\gamma$ is a constant speed geodesic with speed $\abs{\gamma'} = \abs{\alpha}$, it follows from the triangle inequality that $\abs{\varphi'} \leq \abs{\alpha}$, and thus that $(\alpha^2-(\varphi')^2) \geq 0$. 
\end{proof}
\subsection{Proof of \cref{thm:local_stability_function_rips} in the case $\conShrM = 0$}\label{appendix:discrete_case}
Let us provide the outstanding proof of \cref{thm:local_stability_function_rips} in the case where $\conShrM = 0$, with the assumption that $\epsmet < \delta$. 
\begin{proof}
	In this case, the assumptions imply $\lipConstRho = 0$. Hence, we need to construct an $\epsfun$-interleaving. Let us make a few simplifying observations, ignoring the filtrations for now.
	\begin{enumerate}
		\item By definition, any open ball of radius $\catRad$ in $M$ is a $\catK$ space of radius less than $\frac{\varpi_{\kappa}}{2}$, and thus path-connected. As $M$ is assumed to be discrete, this makes each such open ball a singleton. 
		\item As a first consequence, using that $\rho \leq \catRad$, it follows that any two points in $M$ have distance at least $\rho$ from each other. In particular, any such distance is at least $\delta$.
		\item By assumption, we have an $\epsmet$-correspondence between $\metapprx$ and $M$, where $\epsmet \leq \rho - \delta \leq \catRad - \delta$ and $\epsmet < \delta$. Through this correspondence, we can assign to each point in $\metapprx$ a point in $M$, obtaining a decomposition of $\metapprx = \bigsqcup_{x \in M}\metapprx_x.$ Note that for any two such clusters, $\metapprx_{x}, \metapprx_{x'}$ with $x,x' \in M$, and $y\in \metapprx_x, y' \in \metapprx_{x'}$, we have 
		\[ d(y,y') \geq d(x,x') - \epsmet \geq \rho - \epsmet \geq \delta. \]
		\item Summarizing the previous two observations, we obtain that both $\Rips(M)$ and $\Rips(\metapprx)$ are, respectively, coproducts of the Rips simplicial sets $\Rips[\delta](\{x\})$ and $\Rips[\delta](\metapprx_x)$, for $x \in M$. This statements immediately lifts to the filtered setting. The interleaving map induced by the correspondence (as in the proof of \cref{prop:bivariate_interleaving}) is compatible with this decomposition.
		\item Hence, we can without loss of generality reduce to the case where $M = \{x\}$ is a singleton.
		\item Having made this reduction, observe furthermore that for any two $y,y' \in \metapprx$, we have $d(y,y') \leq d(x,x)+ \epsmet = \epsmet < \delta$. Hence, $\Rips(\metapprx)$ is contractible.
	\end{enumerate}
	Let us now observe what these reductions produce on the level of persistent homotopy theory. Let $v = f(x)$. Then, $\Rips(\mbull)$ is the persistent homotopy type given by a point, for $u \geq v$ and by the empty set, otherwise. By assumption, we have that for every $y \in \mathbb{M}$, we have $\|f(x) - \fapprx(y)\|_{\infty} \leq \epsfun$. It follows that $\Rips(\mapprxbull[u])$ is contractible for $u \geq v + \epsfun(1,\dots, 1)$ and empty for any $u$, that does not fulfill $u \geq v - \epsfun(1,\dots, 1)$. We will not have to worry about any other values of $u$. 
	For any two such persistent homotopy types, there is an interleaving of degree $\epsfun$, and in fact only one such interleaving (up to homotopy). In fact, for any of the pairs \begin{align*}
	(X^\bullet,Y^\bullet) \in \{(\Rips(\mbull), \Rips(\mapprxbull[\bullet + \epsfun])), (\Rips(\mapprxbull), \Rips(\mbull[\bullet + \epsfun])), \\
	 (\Rips(\mbull), \Rips(\mbull[\bullet + 2\epsfun])), (\Rips(\mapprxbull), \Rips(\mapprxbull[\bullet + 2\epsfun]))\}	
	\end{align*}
	there is a single morphism in $\ho(\iSpaces^{\RN})$ from $X^{\bullet}$ to $Y^{\bullet}$ (see \cref{lem:contractible_homs} below). This immediately implies the existence and uniqueness of the desired interleaving. 
\end{proof}
	\begin{lemma}\label{lem:contractible_homs}
		Let $X^{\bullet},Y^{\bullet} \in \iSpaces^{\RN}$ be two persistent homotopy types, such that for every $u \in \RN$ the implication 
		\[
		X^u \textnormal{ non-empty} \implies Y^u \textnormal{ contractible}
		\]
		holds. Then the hom-set 
		\[
		\ho( \iSpaces^{\RN})(X^{\bullet},Y^{\bullet})
		\]
		is a singleton.
	\end{lemma}
	\begin{proof}
	Observe first that the set $\{ u \in \RN \mid X^u \textnormal{ non-empty}\}$ is an upset. Denote it by $\inPos \subset \RN$. $X^\bullet$ is the left Kan-extension of its restriction to $\inPos$ (both in the $1$-categorical and in the homotopy theoretical, i.e. $\infty$-categorical, sense). It follows that 
	\[
	\ho( \iSpaces^{\RN})(X^{\bullet},Y^{\bullet}) \cong \ho( \iSpaces^{\inPos})(X^{\bullet}|_{\inPos},Y^{\bullet}|_{\inPos}).
	\]
	Now, by assumption, $Y^{\bullet}|_{\inPos}$ is the terminal object in $\ho( \iSpaces^{\inPos})$. Hence,
	\[
	\ho( \iSpaces^{\inPos})(X^{\bullet}|_{\inPos},Y^{\bullet}|_{\inPos}) \cong \ast.
	\] 
	\end{proof}
\end{document}